\documentclass[11pt,reqno]{amsart}
\usepackage{amsmath,amsfonts,amsthm, color}
\usepackage{enumitem}
\usepackage{hyperref}
\allowdisplaybreaks
\usepackage{tensor}

\usepackage[nocompress]{cite}

\makeatletter 
\def\@tocline#1#2#3#4#5#6#7{\relax
  \ifnum #1>\c@tocdepth 
  \else
    \par \addpenalty\@secpenalty\addvspace{#2}%
    \begingroup \hyphenpenalty\@M
    \@ifempty{#4}{%
      \@tempdima\csname r@tocindent\number#1\endcsname\relax
    }{%
      \@tempdima#4\relax
    }%
    \parindent\z@ \leftskip#3\relax \advance\leftskip\@tempdima\relax
    \rightskip\@pnumwidth plus4em \parfillskip-\@pnumwidth
    #5\leavevmode\hskip-\@tempdima
      \ifcase #1
       \or\or \hskip 1em \or \hskip 2em \else \hskip 3em \fi%
      #6\nobreak\relax
    \hfill\hbox to\@pnumwidth{\@tocpagenum{#7}}\par
    \nobreak
    \endgroup
  \fi}
\makeatother

\newtheorem{theorem}{Theorem}[section]

\newtheorem{proposition}[theorem]{Proposition}
\newtheorem{lemma}[theorem]{Lemma}
\newtheorem{definition}[theorem]{Definition}

\newtheorem{corollary}[theorem]{Corollary}

\newcommand{\loja}{\L{}ojasiewicz}
\newcommand{\pa}{P_{g_0}}
\newcommand{\cf}{\mathcal{F}}
\newcommand{\dvol}{\hspace{1mm}dv_{g_0}}

\newcommand{\wtt}{W^{2,2}(M)}
\newcommand{\lone}{L^{\frac{2n}{n-4}}(M)}
\newcommand{\ltwo}{L^{\frac{2n}{n+4}}(M)}

\newcommand{\wgt}{u_\infty^{\frac{8}{n-4}}}
\newcommand{\uznu}{\bar{u}_{z_\nu}}
\DeclareMathAlphabet{\mathmybb}{U}{bbold}{m}{n}
\newcommand{\1}{\mathmybb{1}}

\newcommand{\ddiv}{\mathrm{div}}

\newcommand{\mfc}{\mathfrak{c}}
\newcommand{\tmfc}{\tilde{\mathfrak{c}}}

\def\tr{\mathrm{tr}}
\def\ric{\mathrm{Ric}}
\def\rm{\mathrm{Rm}}
\def\supp{\mathrm{supp}}
\def\contr{\mathrm{contr}}

\newcommand{\<}{\left\langle}
\renewcommand{\>}{\right\rangle}

\newcommand{\mci}{\mathcal{I}}
\newcommand{\mcz}{\mathcal{Z}}

\newcommand{\cL}{\mathcal{L}}
\newcommand{\vol}{\hspace{1mm}dv}

\newcommand{\RN}[1]{%
\textup{\uppercase\expandafter{\romannumeral#1}}%
}
\newcommand{\KN}{\mathbin{\bigcirc\mspace{-15mu}\wedge\mspace{3mu}}}

\makeatletter
\@namedef{subjclassname@2020}{%
\textup{2020} Mathematics Subject Classification}
\makeatother

\begin{document}

\title{Global Convergence of the Gursky-Malchiodi $Q$-curvature flow}

\author{Liuwei Gong}
\address[Liuwei Gong]{Department of Mathematics, Chinese University of Hong Kong, Shatin, NT, Hong Kong}
\email{lwgong@math.cuhk.edu.hk}

\author{Sanghoon Lee}
\address[Sanghoon Lee]{Department of Mathematics, Chinese University of Hong Kong, Shatin, NT, Hong Kong}
\email{sanghoonlee@cuhk.edu.hk}

\author{Juncheng Wei}
\address[Juncheng Wei]{Department of Mathematics, Chinese University of Hong Kong, Shatin, NT, Hong Kong}
\email{wei@math.cuhk.edu.hk}

\begin{abstract}

In their seminal work \cite{MR3420504}, Gursky and Malchiodi introduced a non-local conformal flow in dimensions $n \ge 5$ to resolve the constant $Q$-curvature problem. They proved sequential convergence of the flow for initial metrics with positive scalar curvature and $Q$-curvature, provided the energy was sufficiently small. 

In this paper, we prove the global convergence of the flow for arbitrary initial energy under the same positivity assumptions by establishing a non-local version of the Łojasiewicz-Simon inequality for the Paneitz-Sobolev quotient along the flow.

We construct test bubbles and estimate their Paneitz-Sobolev quotients, a strategy that was carried out in the celebrated work of Brendle \cite{MR2357502} in the context of the Yamabe flow. We develop a more geometric and systematic proof that addresses the algebraic and computational complexity inherent in the $Q$-curvature and the Paneitz operator. Along the way, we derive a stability inequality for the Paneitz-Sobolev quotient using a higher-order Koiso-Bochner formula established in recent work of Bahuaud, Guenther, Isenberg, and Mazzeo \cite{BGIMazzeo}.
\end{abstract}

\date{\today}
\maketitle
\tableofcontents

\section{Introduction}

\subsection{Background and main results}
Geometric flows in conformal geometry have been extensively studied over the past decades as powerful tools for addressing fundamental problems such as uniformization theorems and prescribed curvature problems, and they have played a central role in the development of both geometry and analysis. From a geometric perspective, such flows provide a systematic approach to selecting canonical representatives within a given conformal equivalence class. From an analytic perspective, their dynamical behavior has had broad influence on the study of many model partial differential equations, including semilinear and non-local heat equations.

The most well-known example of such flows is the Yamabe flow, introduced by Hamilton to solve the Yamabe problem. As the negative gradient flow of the normalized total scalar curvature, its long-time existence and convergence have been extensively studied. Global convergence was proved by Chow \cite{MR1168117} in the locally conformally flat case with positive Ricci curvature, by Ye \cite{MR1258912} for metrics with negative scalar curvature or in the locally conformally flat setting, and by Schwetlick and Struwe \cite{MR2011332} below the two-bubble energy threshold. A complete convergence theory was later established by Brendle \cite{MR2168505, MR2357502}. The convergence rate was subsequently investigated by Carlotto, Chodosh, and Rubinstein \cite{MR3352243}.

Beyond the Yamabe flow, a variety of conformal geometric flows have been introduced to address curvature prescription problems of different orders and nonlinearities. In the fractional setting, flows associated with the fractional conformal Laplacian and fractional $Q$-curvature have been studied by Jin and Xiong \cite{MR3276166} and Chan-Sire-Sun \cite{MR4288645} to treat non-local curvature equations motivated by scattering theory and conformal geometry on asymptotically hyperbolic manifolds. On the other hand, fully nonlinear conformal flows, such as those driven by symmetric functions of the Schouten tensor, have been developed by series of works Guan and Wang \cite{MR1978409}, Sheng-Trudinger-Wang \cite{MR2362323}, Ge and Wang \cite{MR2290138, MR2348093}, and Ge-Lin-Wang \cite{MR2629509}  to address problems related to fully nonlinear Yamabe-type equations. Substantial progress has been made on existence, regularity, and convergence for these flows under various structural and geometric assumptions. Together, these developments highlight the rich landscape of conformal geometric flows beyond the classical Yamabe setting and motivate the study of higher-order and non-local curvature flows. In this work, we focus on the conformal flow associated with the fourth-order $Q$-curvature and the Paneitz operator.

In conformal geometry, the $Q$-curvature of a $n$-dimensional Riemannian manifold $(M, g)$ is a curvature associated to the fourth order conformally covariant operator, called the Paneitz operator $P_g$. In the pioneering work by Paneitz \cite{Paneitz} and Branson \cite{MR832360}, they are defined respectively by 
\begin{align*}
Q_g =& -\frac{1}{2(n-1)} \Delta_g R_g -\frac{2}{(n-2)^2} |\ric_g|^2 + \frac{n^2(n-4) + 16(n-1)}{8(n-1)^2(n-2)^2} R_g^2,\\
P_g \phi= & \Delta_g^2 \phi - \ddiv_g \bigg[\bigg(\frac{(n-2)^2+4}{2(n-1)(n-2)} R_g g - \frac{4}{n-2} \ric_g  \bigg)d\phi \bigg] +\frac{n-4}{2} Q_g \phi,
\end{align*}
for any smooth function $\phi$ on $M$ when $n \ge 3$. Here, $R_g$ is the scalar curvature, $\ric_g$ is the Ricci curvature, $\Delta_g=\ddiv_g\nabla_g$ is the Laplace-Beltrami operator, and $\ddiv_g$ is minus of the adjoint of the exterior derivative $d$. 
The Paneitz operator and the $Q$-curvature exhibit the following conformal covariance property:
\begin{align*}
& Q_{\hat{g}} = e^{-4u} (P_g u + Q_g ) & &\text{for $n=4$ and $\hat{g} = e^{2u}g$} \\
& Q_{\hat{g}} = \frac{2}{n-4} u^{-\frac{n+4}{n-4}} P_g u & &\text{for $n \ne 4$ and $\hat{g} = u^{\frac{4}{n-4}}g$}.
\end{align*}

In dimension four, the $Q$-curvature is known to serve as a four-dimensional analogue of the Gaussian curvature of Riemannian surfaces. Indeed, just as the Gaussian curvature appears in the conformal transformation law of the Laplace-Beltrami operator in dimension two and governs the Gauss-Bonnet formula, the 
$Q$-curvature arises naturally in the conformal covariance of the Paneitz operator and enters the Chern-Gauss-Bonnet formula in dimension four. In particular, the total 
$Q$-curvature is a global conformal invariant paralleling the role of total Gaussian curvature on surfaces. For geometric and topological applications of the $Q$-curvature, especially to the conformal sphere theorem, we refer to Gursky \cite{MR1724863} and Chang-Gursky-Yang \cite{MR1923964, MR2031200}. For its relationship with isoperimetric inequalities, see Chang-Qing-Yang \cite{MR1763657} and Wang \cite{MR3366853}. Recent applications to the study of Poincar\'e-Einstein manifolds can be found in Chang-Ge \cite{MR3886176, fillin}, Chang-Ge-Qing \cite{MR4130465}, and Lee-Wang \cite{leewang}, and the references therein.

The existence of solutions to the constant $Q$-curvature equation was first studied in dimension $n=4$ by Chang and Yang \cite{MR1338677}, in connection with the existence of extremal metrics for the determinant of the Paneitz operator and a Moser-Trudinger inequality, under the assumptions that $\ker P_g = \mathbb{R}$, $P_g$ non-negative, and the total $Q$-curvature is less than $8\pi^2$, the total $Q$-curvature of the standard sphere. Gursky \cite{MR1724863} later showed that sufficient conditions for these assumptions to hold are the nonnegativity of the Yamabe constant and the total $Q$-curvature. The existence theorem was further extended by Djadli and Malchiodi \cite{MR2456884} to the case $\ker P_g = \mathbb{R}$ and when the total $Q$-curvature is not a positive integer multiple of that of the standard sphere, using the compactness result established by Malchiodi \cite{MR2248155}. Later, Li-Li-Liu \cite{MR2964636} settled the  case in which the total $Q$-curvature is exactly $8\pi^2$. 

For existence results for the constant $Q$-curvature equation in dimension $n=3$, we refer to Hang and Yang \cite{MR3465087} and the references therein. In dimensions $n \ge 5$, the first existence theorem was proved by Qing and Raske \cite{MR2232210} under the assumptions of local conformal flatness, positive Yamabe constant, and additional geometric conditions. Gursky and Malchiodi \cite{MR3420504} subsequently established existence under the positivity assumptions that the scalar curvature is nonnegative and the $Q$-curvature is semipositive, using a flow method. Soon after, Hang and Yang \cite{MR3518237}  provided a variational proof. Below, we describe the result of Gursky and Malchiodi in more detail.

One of the fundamental difficulties in the study of fourth-order elliptic operators is the lack of a maximum principle in general. Surprisingly, Gursky and Malchiodi \cite{MR3420504} established a strong maximum principle for the Paneitz operator $P_g$ under the assumptions  $Q_g$ is semi-positive---that is, $Q_g \ge 0$ with $Q_g > 0$ somewhere---and that the scalar curvature satisfies $R_g \ge 0$.  Moreover, under the same assumptions, they proved the positivity of the operator and of its Green's function, which implies that the flow defined below \eqref{uflowintro} is well defined. They also established the positive mass theorem for the Paneitz operator $P_g$ in dimensions $5 \le n \le 7$.
 These results enabled them to study the non-local flow  \eqref{uflowintro}, which is the negative $W^{2,2}$-gradient flow of the Paneitz-Sobolev quotient
\begin{equation}\label{eq:PS quotient def}
\cf[u]
= \frac{\displaystyle \int_M u P_{g_0} u \dvol}
{\left(\displaystyle \int_M |u|^{\frac{2n}{n-4}}  \dvol \right)^{\frac{n-4}{n}}}.
\end{equation}

The stationary points of this non-local flow correspond to conformal metrics with constant positive $Q$-curvature. Gursky and Malchiodi proved long-time existence of the flow and carefully constructed initial data satisfying the positivity conditions described above together with a small initial energy assumption, which enabled them to establish sequential convergence of the flow to a stationary point. However, the question of global convergence for arbitrary initial energy has remained open. The main result of this paper resolves this problem.

\begin{theorem}\label{thm:globalconv}
Let $(M,g_0)$ be a closed manifold of dimension $n \ge 5$ that is not conformally equivalent to the standard sphere $S^n$. Assume that $(M,g_0)$ has semi-positive $Q$-curvature and nonnegative scalar curvature. Then the non-local flow
\[
g_t(p) := u(t,p)^{\frac{4}{n-4}} g_0(p), \qquad p \in M,
\]
where $u$ satisfies
\begin{align} \label{uflowintro}
\left\{
\begin{array}{lll}
\displaystyle \frac{\partial u}{\partial t}
= -u + \mu[u]\, P_{g_0}^{-1}\!\left( |u|^{\frac{n+4}{n-4}} \right), \\[6pt]
u(0, \cdot) = 1,
\end{array}
\right.
\end{align}
with
\begin{align} \label{mudef}
\mu[u]
= \frac{ \displaystyle \int_M u\, P_{g_0} u \, \dvol }
{ \displaystyle \int_M |u|^{\frac{2n}{n-4}} \, \dvol },
\end{align}
exists for all time and converges to a metric with constant $Q$-curvature.
\end{theorem}

We remark that the strong maximum principle plays a crucial role in preventing the solution $u$ from becoming non-positive. Without such positivity assumptions, both the long-time existence of the flow and the existence theory for the constant $Q$-curvature problem remain largely open.  

We further note that higher-order conformal flows, including the fourth-order $Q$-curvature flow and, more generally, flows whose leading operator is $(-\Delta_g)^{\frac{n}{2}}$  in even dimensions $n$, were studied by Brendle \cite{MR1999924}. He proved global convergence under the assumptions that the associated higher-order Paneitz operator is non-negative with trivial kernel and that the total higher-order $Q$-curvature is strictly less than that of the standard sphere. The critical case in which the total higher order $Q$-curvature equals that of the standard sphere was subsequently investigated by Bi and Li \cite{MR4687333}.  
The $Q$-curvature flow on the standard sphere $S^4$ was studied by Malchiodi and Struwe \cite{MR2217518}, while Xiong \cite{MR4752331} analyzed non-local integral flows, including the $Q$-curvature flow, on the standard sphere $S^n$. We also refer to \cite{MR4935608, MR2803837, MR4983363} and the references therein for further developments.

The analytic theory of the fourth-order Paneitz operator and $Q$-curvature is far richer than can be exhaustively surveyed here. For sharp Sobolev inequalities, we refer to Djadli and Hebey \cite{MR1769728}. Compactness results in dimension $n = 4$ were obtained by Malchiodi \cite{MR2248155}. In higher dimensions, compactness was established by Li and Xiong \cite{MR3899029} for $5 \le n \le 9$, and more recently by Gong, Kim, and Wei \cite{gong2025compactnessnoncompactnesstheoremsfourth} for $5 \le n \le 24$. In contrast, Wei and Zhao \cite{MR3016505} constructed examples exhibiting noncompactness for $n \ge 25$.  Uniqueness results were obtained by V\'etois \cite{MR4800966} and Case \cite{MR4803233}, while non-uniqueness phenomena were studied by Bettiol, Piccione, and Sire \cite{MR4251294}. We also refer to \cite{MR2241307,  MR3431611, MR2232210, MR4840445, MR3509928, MR2372906, MR4170788, MR3694645, MR4761862} and the references therein for further analytic developments.

\subsection{Novelties of the method}

In his seminal work, Simon \cite{MR727703} proved an infinite-dimensional version of the \loja{} inequality and applied it to establish a general convergence theorem for gradient flows of analytic functionals. In the context of conformal geometric flows, however, the relevant energy functionals are not a priori analytic unless a uniform positive lower bound along the flow is available. For the Yamabe flow, the beautiful and profound works of Brendle \cite{MR2168505, MR2357502} succeeded in establishing a \L ojasiewicz--Simon inequality even in the presence of possible bubble formation. In this paper, we prove a non-local version of the \loja{}--Simon  inequality for the Paneitz--Sobolev quotient along the flow.

In this article, $W^{2,2}(M)$-norm of a function $u$ is defined by 
\begin{equation*}
\| u \|_{W^{2,2}(M)} := \bigg(\int_M u \pa u \vol_{g_0} \bigg)^{\tfrac{1}{2}}
\end{equation*}
which is equivalent to the standard Sobolev norm due to the positivity of the Paneitz operator proved in \cite{MR3420504}. For a sequence of time slices $t_\nu \rightarrow \infty$, we apply the fourth-order version of Struwe’s concentration-compactness principle \cite{MR760051}, established by Hebey and Robert \cite{MR1867939}, to decompose $u_{t_\nu} = v_{t_\nu} + w_{t_\nu}$ where $v_{t_\nu}$ is approximately equal to $u_\infty$ plus a finite sum of test bubbles, and $w_{t_\nu}$ is a small error converging to $0$ strongly in $W^{1,2}(M)$. Here $u_\infty$ denotes the weak $W^{2,2}(M)$-limit of $u_{t_\nu}$. We further refine this decomposition by minimizing $W^{2,2}(M)$-norm of $w_{t_\nu}$, following the argument of Brendle  \cite{MR2168505}.  

The main ingredients of the proof are twofold. Firstly, we establish a uniform lower bound on the second variation of $\cf$, defined in \eqref{eq:PS quotient def}:
\begin{equation}\label{ineq7}
c \, \| w_{t_\nu} \|_{W^{2,2}(M)}^2
\le D^2 \cf_{v_{t_\nu}} \big[ w_{t_\nu}, w_{t_\nu} \big],
\end{equation}
that is, a coercivity estimate for the second variation of the Paneitz--Sobolev quotient at $v_{t_\nu}$ in the direction $w_{t_\nu}$. Secondly, we prove a uniform upper bound for the Paneitz--Sobolev quotient in the form of a \loja{}--Simon inequality:
\begin{equation}\label{ineq8}
\cf[v_{t_\nu}] - \cf_\infty
\le C \bigg\| \frac{\partial u}{\partial t}(t_\nu , \cdot) \bigg\|_{W^{2,2}(M)}^{1+\gamma},
\end{equation}
where $0 < \gamma < 1$. Combining these two ingredients with suitable algebraic inequalities yields the desired \loja{}--Simon inequality. We remark that the estimate \eqref{ineq8} is non-local.

The geometric difficulty in proving the estimate \eqref{ineq8} in dimensions $n \ge 8$ lies in constructing a family of test bubbles whose Paneitz--Sobolev quotient is strictly less than that of the standard sphere. This was accomplished by Brendle \cite{MR2357502} in the case of the Sobolev quotient. In the present work, we establish the following analogue for the Paneitz--Sobolev quotient:

\begin{theorem}\label{thm:energy estimate}
Let $n \ge 8$. There is a small constant $\delta_0>0$ such that for any $p\in M$, $0<2\epsilon\leq \delta\leq \delta_0$, there is a positive test function $v_{p,\epsilon,\delta}$ with the following properties:
\begin{enumerate}[label=(\alph*),leftmargin=*]
\item There is a positive constant $C$ depending only on $(M,g)$ such that
\begin{align*}
\int_{M} v_{p,\epsilon,\delta}P_g v_{p,\epsilon,\delta}\vol_g \leq & Y_4(S^n)\left(\int_M v_{p,\epsilon,\delta}^{\frac{2n}{n-4}}\vol_g\right)^{\frac{n-4}{n}}-\mci(p,\delta)\epsilon^{n-4}\\
&-\frac{1}{C} \epsilon^{n-4}\int_{B_{\delta}(p)}|W_g(\tilde{p})|^2(\epsilon+d(\tilde{p},p))^{8-2n}\vol_g\\
&+C\delta^{2d+6-n}\epsilon^{n-4}+C\left(\frac{\epsilon}{\delta}\right)^{n-4}\log^{-1}\left(\frac{\delta}{\epsilon}\right),
\end{align*}
where $\mci(p,\delta)$ can be seen in Definition \ref{def:mci}.
\item There exists a continuous function $\mci: \mcz \rightarrow \mathbb{R}$ such that 
$$
\sup_{p\in \mcz} |\mci(p,\delta)-\mci(p)| \rightarrow 0 \text{ as }\delta \rightarrow0.
$$ 
\item If $p \in \mcz$, then the manifold $(M\setminus\{p\}, \hat{g}:=G_p^{\frac{4}{n-4}}g)$ has a well-defined $Q$-mass $m(\hat{g})$, which can be seen in Definition \ref{def:Q-mass}. Moreover, 
$$
m(\hat{g})=\frac{4(n-1)}{n-4}\mci(p).
$$
\item Fix $\delta$ such that $0 < \delta \leq \delta_0$. As $\epsilon \rightarrow 0$, the functions $v_{p,\epsilon,\delta}$ converge to a ``standard bubble'' on $\mathbb{R}^n$ after rescaling. More precisely,
$$
\lim_{\epsilon\rightarrow0} \epsilon^{\frac{n-4}{2}}v_{p,\epsilon,\delta}(\exp_{p}(\epsilon x))=\left(\frac{1}{1+|x|^2}\right)^{\frac{n-4}{2}}.
$$
for any $p\in M$ and $x\in T_pM$.
\end{enumerate}
\end{theorem}

Here, we denote by $Y_4(S^n)$ the Paneitz--Sobolev constant of the standard sphere, and we set $d=\lfloor \frac{n-4}{2}\rfloor$. Define $\mcz$ to be the set of all points $p\in M$ such that 
$$
\limsup_{\tilde{p}\rightarrow p} d(\tilde{p},p)^{2-d}|W_g(\tilde{p})|=0.
$$ 
The function $\mci$, defined on $\mcz$, can be interpreted as the mass of the Paneitz operator. Its positivity is known in the locally conformally flat case by Humbert and Raulot \cite{MR2558328}, for
$5 \le n \le 7$ by Gursky and Malchiodi \cite{MR3420504}, and for $n \ge 8$ in the form needed here by Avalos-Laurain-Lira \cite{MR4375794}. If $p \notin \mcz$, then the Weyl curvature term provides the dominant negative contribution. If $p \in \mcz$, the positive mass term instead guarantees that the Paneitz-Sobolev quotient of $v_{p, \epsilon, \delta}$ is strictly less than $Y_4(S^n)$. We remark that the positive mass theorem is available under our positivity assumptions and therefore does not need to be imposed as an additional assumption.

The idea of obtaining upper energy bounds by constructing suitable test functions can be traced back to the work of Aubin \cite{MR431287} and Schoen \cite{MR788292} in the context of the scalar curvature problem. For the Paneitz-Sobolev quotient, analogous estimates were proved by Esposito and Robert \cite{MR1942129} at non-locally conformally flat points in dimensions $n \ge 8$, and by Gursky and Malchiodi \cite{MR3420504} in the cases $5 \le n \le 7$ or when the manifold is locally conformally flat. Our theorem provides a complete geometric picture in dimensions $n \ge 8$, fully incorporating the background geometry, including the vanishing order of the Weyl curvature.

In the construction of the test bubbles, we add an auxiliary term solving a linearized equation that reflects the boundary geometry, and we estimate the energy using the second variation of the normalized total $Q$-curvature. The key difference from Brendle’s proof \cite{MR2357502} is that, due to the algebraic complexity of the first and second variations of the $Q$-curvature and the Paneitz operator, it is impossible to track the individual terms by hand. To overcome this difficulty and to further reveal the geometric structure of the $Q$-curvature, we develop a more geometric and systematic approach, which we now describe in more detail.

The starting point is the works of Lin and Yuan \cite{MR3529121, MR4380034}, where they proved stability results at Einstein manifolds for a suitably normalized total $Q$-curvature. In fact, rigidity and stability results for quadratic curvature functionals were done by Gursky and Viaclovsky \cite{MR3318510} prior to \cite{MR3529121, MR4380034}. These works extend the deformation theory of Einstein metrics for the total scalar curvature developed by Fischer and Marsden \cite{MR380907}. Roughly speaking, the normalized total $Q$-curvature is stable at the round sphere, in the sense that any perturbation of the metric decreases it in a neighborhood of the sphere. On the other hand, when testing the Paneitz-Sobolev quotient using the standard bubble, the conformally deformed metric induced by the standard bubble is no longer sufficiently close to the round sphere, and Weyl curvature terms intervene. As a result, the desired inequality holds only in low dimensions $5 \le n \le 7$ or when the background metric is locally conformally flat. This necessitates a much more delicate correction of the test bubble.

We use the fourth-order linearized PDE discovered by Gong-Kim-Wei \cite{gong2025compactnessnoncompactnesstheoremsfourth} (see Lemma~\ref{lem:correction term}) to construct a correction term, or equivalently a vector field $X$, which encodes the background geometry via the conformal Killing operator on vector fields. The first difficulty we encounter is computing the second variation of the total $Q$-curvature. This is accomplished by exploiting the diffeomorphism invariance of the $Q$-curvature (see Proposition~\ref{prop:D2Q S direction}) in a subtle way. Since we cannot use Ebin’s slicing theorem as in \cite{MR4380034}, we instead derive a full formula for the deformation of the total $Q$-curvature with respect to the corrected metric in order to obtain precise estimates for the error terms. 

These error terms are ultimately controlled by the Weyl curvature via the higher-order Koiso-Bochner formula established in the recent work of Bahuaud-Guenther-Isenberg-Mazzeo \cite{BGIMazzeo} (see Lemma \ref{lem:two koiso identities}). Koiso-Bochner formula was used to study the linear stability of the Einstein-Hilbert action at Einstein metrics. The stability, or more generally the index, of such critical metrics is determined by the spectral properties of the Lichnerowicz Laplacian. We remark that, since the positive $Q$-mass theorem requires the Weyl tensor to vanish up to lower order than in the positive ADM mass theorem, it is necessary to employ a higher-order Koiso--Bochner formula in order to obtain sharp coercivity estimates for the connection bi-Laplacian acting on symmetric two-tensors; see Proposition~\ref{prop:DW sphere lower bound}.

More precisely, assume a metric $g = \exp(h)$ in the conformal normal coordinate of the small ball $B_\delta(0)$ with radius $\delta$. Let $H$ denote the polynomial of degree $d$ approximating $h$. Then we can find a vector field $X$ such that
\begin{equation*}
H = T + S, \qquad
S = \mathcal{L}_X g_e - \frac{2}{n} (\operatorname{div} X)\, g_e,
\end{equation*}
where $g_e$ is the standard Euclidean metric and $\cL_X$ is the Lie derivative. If we lift these two symmetric tensors to $\bar T$ and $\bar S$ on the standard sphere via conformal transformation, then $\bar T$ is divergence-free and $\bar S$ corresponds to the deformation of the metric induced by the diffeomorphism generated by the vector field $X$, see Definition \ref{def:barg metric} and Lemma \ref{lem:barT barS}. In view of the stability result, the $\bar S$ component should not contribute to the interior integral over the small ball when computing the deformation of the normalized total $Q$-curvature, and therefore should appear only as a boundary term. The significance of Proposition \ref{prop:D2Q S direction} is that these procedures can be carried out in a very geometric and streamlined way, especially in the present $Q$-curvature setting, where the amount of computation is more than triple that in Brendle \cite[Proposition 5]{MR2357502}. By carefully examining the interior term, we arrive at the stability inequality stated in Proposition~\ref{prop:secondvonsphere}. In the end, with the help of the higher-order Koiso-Bochner formula, we obtain a sharp estimate on the interior term. The boundary terms can then be estimated relatively easily, as they are higher-order error terms.

To complete the global error estimate, we next require a careful analysis of the structure of the total $Q$-curvature in order to show that its first and second variations indeed dominate the global quantity. At this stage, we take advantage of the algebraic structure arising from the conformal covariance of the $Q$-curvature and the Paneitz operator. The key observation is that, after adding the correction term $w_{\epsilon,\delta}$, the conformally deformed total $Q$-curvature depends quadratically on $w_{\epsilon,\delta}$.
Lemma \ref{lem:algebraicspl} provides a systematic way to sharply isolate the effect of adding the correction term $w_{\epsilon,\delta}$, and offers a conceptual framework for organizing and classifying the remaining terms in the difference between the global total $Q$-curvature and its second-order approximation obtained in Proposition \ref{prop:secondvonsphere}. It turns out that these remaining terms are higher-order error terms, arising either from the smallness of $w_{\epsilon,\delta}$, from higher-order errors in the variation of Paneitz operator, or from errors related to the volume density in conformal normal coordinates. We emphasize that it is unnecessary to compute the second variation of the Paneitz operator for these estimates.

Hence, we make full use of the geometric structure of the $Q$-curvature and the Paneitz operator, together with their diffeomorphism invariance and conformal covariance properties.

Finally, an analytic difficulty in proving the estimate \eqref{ineq7} arises from the fact that the flow is non-local. In particular, there is no general maximum principle available in the analysis. Nevertheless, we are able to exploit the structure of this non-local PDE to derive several fundamental properties. For instance, we show that $u(t_\nu,\cdot)$ is a Palais--Smale sequence as $t_\nu \to \infty$, and obtain estimates for the energy of the bubbling component of $u$; see Proposition \ref{lem:LnormPf} and Lemma \ref{lem:proj} for examples.

A key observation is a simple but important identity
\begin{equation*}
\pa\!\left( \frac{\partial u}{\partial t} \right)
= - \pa u + \mu[u] \, u^{\frac{n+4}{n-4}},
\end{equation*}
or equivalently,
\begin{equation*}
\pa\!\left( \frac{\partial u}{\partial t} + u \right)
= \mu[u] \, u^{\frac{n+4}{n-4}},
\end{equation*}
which rewrites the original evolution equation as a seemingly local elliptic equation. This reformulation allows us to derive all the necessary estimates using techniques similar to those for local elliptic equations. Thus, although the resulting estimates are non-local, the methods used to obtain them are not fundamentally different from those for local parabolic equations. This reflects the fact that the \loja{} inequality we aim to establish is more elliptic in nature than dependent on the parabolicity of the flow.

\subsection{Organization of this paper}

In Section 2, we recall some properties of the flow---including long-time existence, monotonicity---and prove that the standard bubbles are the only possible blow-up profiles. 

In Section 3, we provide the second variations of total $Q$-curvature at Einstein metrics, both in transverse-traceless two-tensor directions and Lie derivative directions. Moreover, we recall higher-order Koiso-Bochner formula and use it to derive the second variations of the $L^2$-norm of Weyl curvature tensor.

In Section 4, for $n\geq 8$, we establish the upper bounds of the Paneitz--Sobolev quotient by constructing suitable test functions. These are obtained by modifying the standard bubbles by the solutions of the linearized equations and the Green's functions. 

More specifically, in Subsection 4.1 we introduce a one-parameter family of metrics passing through the round sphere metric; these provide the leading terms in our estimates. In Subsection 4.2, we apply the tools from Section 3 to control the leading terms, and in Subsection 4.3 we estimate the error terms. After modifying with Green’s functions in Subsection 4.4, we complete the proof of Theorem \ref{thm:energy estimate} in Subsection 4.5.

In Section 5, we construct standard test functions when $n=5,6,7$. We then apply Struwe's decomposition (concentration--compactness principle) to a sequence of diverging time slices of the non-local flow and write it into two parts $u_{t_\nu} = v_{t_\nu} + w_{t_\nu}$ for general dimensions.

In Section 6, we establish a coercivity estimate for the second variation of the Paneitz--Sobolev quotient at $v_{t_\nu}$ in the direction $w_{t_\nu}$, given by inequality \eqref{ineq7}. Combining with the estimates of the Paneitz--Sobolev quotient in Section 4, we obtain a uniform energy estimate for $v_{t_\nu}$, see inequality \eqref{ineq8}.

In Section 7, with the two ingredients above, we obtain a \loja{} type inequality and complete the proof of Theorem \ref{thm:globalconv}.

\subsection{Notations}

We use the following notations in this introduction and the rest part of this paper.
\begin{itemize}
\item[-] The flow notations:
$$ f = \frac{\partial u}{\partial t},\hspace{4mm} u_\nu =  u(t_\nu, \cdot), \hspace{4mm} f_\nu = f(t_\nu, \cdot);$$

$$\mu [u] = \frac{ \int u P_{g_0} u\ \vol_{g_0} }{\int |u|^{\frac{2n}{n-4}} \vol_{g_0} } , \hspace{4mm}\mu(t) = \mu[u(t, \cdot)];$$

$$\cf [u] = \frac{ \int u P_{g_0} u\ \vol_{g_0} }{\big(\int |u|^{\frac{2n}{n-4}} \vol_{g_0} \big)^{\frac{n-4}{n}}}, \hspace{4mm} \cf(t) = \cf[u(t, \cdot)];$$

$$\mu_\infty = \lim_{t\rightarrow \infty} \mu(t), \hspace{4mm} \cf_\infty = \lim_{t\rightarrow \infty} \cf(t).$$
\item[-] $g_e$ denotes the Euclidean metric.
\item[-] $\ddiv_g$, $\tr_g$ denote the trace and divergence with respect to metric $g$. For the Euclidean case, we simply write $\ddiv_{g_e}$ as $\ddiv$ and $\tr_{g_e}$ as $\tr$.
\item[-] $u_\epsilon(x) = \big(\frac{\epsilon}{\epsilon^2 + |x|^2} \big)^{\frac{n-4}{2}}$ on $\mathbb{R}^n$ and $g_c=u_\epsilon^{\frac{4}{n-4}} g_e$.
\item[-] Constants $\mfc(n)=(n-4)(n-2)n(n+2)$ and $\tmfc(n)=\frac{n+4}{n-4}\mfc(n)$.
\item[-] Constants $\alpha_n:=-\frac{1}{2(n-1)}$, $\beta_n:=-\frac{2}{(n-2)^2}$, $\gamma_n:=\frac{n^3-4n^2+16n-16}{8(n-2)^2(n-1)^2}$. So, $Q$-curvature is given by: for $n\geq 3$, 
$$
Q_g=\alpha_n\Delta_g R_g+\beta_n|\ric_g|_g^2+\gamma_n R_g^2.
$$
\item[-] The Paneitz--Sobolev constant of the standard sphere:
$$
Y_4(S^n)=\mfc(n)\left(\int_{\mathbb{R}^n}u_{\epsilon}^{\frac{2n}{n-4}}dx\right)^{\frac{4}{n}}.
$$
\item[-] Denote $S_2(M)$ to be the space of symmetric two-tensors and 
$$
S^{T,T}_{2,g}(M):=\{h\in S_2(M)\mid \ddiv_gh=0,\tr_g h=0\}
$$
to be the space of transverse-traceless tensors. 
\item[-] Repeated indices follow the Einstein summation convention. Commas denote differentiations in local coordinates, and semicolons denote covariant differentiations.
\item[-] Let $h,\bar{h}\in S_2(M)$. Denote cross product $(h\times_g \bar{h})_{ij}:=h\indices{_{i}^k}\bar{h}_{kj}$, and Kulkarni-Nomizu product:
\begin{align*}
(h\KN \bar{h})_{ijkl}&:=h_{il}\bar{h}_{jk}+h_{jk}\bar{h}_{il}-h_{ik}\bar{h}_{jl}-h_{jl}\bar{h}_{ik}.
\end{align*}
\item[-] Schouten tensor: $A_g = \frac{1}{n-2}(\ric_g - \frac{1}{2(n-1)}R_gg)$. Weyl tensor: $W_g=\rm_g-A_g\KN g$. Also, $\sigma_k(A_g)$ the $k$-th symmetric function of the eigenvalues of $A_g$.
\item[-] Let $d=\lfloor \frac{n-4}{2}\rfloor$, and $\mcz$ be the set of all points $p\in M$ such that 
$$
\limsup_{\tilde{p}\rightarrow p} d(\tilde{p},p)^{2-d}|W_g(\tilde{p})|=0.
$$
\item[-] Let $h \in S_2(M)$. Denote the Lichnerowicz Laplacian to be
\begin{align*}
(\Delta_L h)_{ij}:=h\indices{_{ij;k}^k}+2\rm_{iklj}h^{kl}-(\ric\times h)_{ij}-(h\times \ric)_{ij}.
\end{align*}
\item[-] Let $G_p$ be the Green’s function for the Paneitz's operator with pole at $p$. We normalize the Green's function such that 
$$
\lim_{d(\tilde{p},p)\rightarrow 0}d(\tilde{p},p)^{4-n} G_p(\tilde{p}) = 1.
$$
\end{itemize}

\section{Preliminaries}

Throughout this paper, we always assume the following positivity conditions:
$$\begin{cases}
& Q_{g_0} \ge 0 \text{ and } Q_{g_0}>0 \text{ somewhere,} \\
& R_{g_0} \ge 0.
\end{cases}$$
By Gursky and Malchiodi \cite[Proposition~B]{MR3420504}, the Paneitz operator $\pa$ is positive, and its Green’s function is positive by \cite[Proposition~C]{MR3420504}. By a standard argument, the positivity of the Paneitz operator implies that our $W^{2,2}(M)$-norm is equivalent to the standard Sobolev norm. 

\subsection{Elementary estimates}

In this subsection, we recall several properties of the flow established in \cite{MR3420504} and prove some additional analytic properties.

We define
\begin{align} \label{fdef}
f := \frac{\partial u}{\partial t} =  - u + \mu P_{g_0}^{-1} \big( u^{\frac{n+4}{n-4}} \big).
\end{align}

\begin{proposition} [{\cite[Proposition 3.4]{MR3420504}}]
The flow 
\begin{align*} 
\left\{ \begin{array}{lll} \displaystyle \frac{\partial u}{\partial t} =  - u + \mu[u] P_{g_0}^{-1} \big( u^{\frac{n+4}{n-4}} \big), \\
\\
u(0, \cdot) = 1,
\end{array}
\right.
\end{align*}
with
\begin{align*}
\mu[u] = \frac{ \int_M u P_{g_0} u \dvol }{\int_M |u|^{\frac{2n}{n-4}} \dvol }.
\end{align*}
exists for all time and stays $u>0$. 

\end{proposition}

Next, we record the basic invariance and monotonicity properties of the flow.

\begin{lemma}[{\cite[Lemma 3.2]{MR3420504}}] \label{W22invariance} 
We have: 
\begin{align}  \label{dPdt}
\frac{d}{dt} \int u P_{g_0} u\ \vol_{g_0} = 0.
\end{align}
\end{lemma}

\begin{lemma} [{\cite[Lemma 3.3]{MR3420504}}] \label{volmon}   The conformal volume satisfies
\begin{align} \label{dVdt}
\frac{d}{dt} V = \frac{d}{dt} \int u^{\frac{2n}{n-4}}\ \vol_{g_0} = \frac{2n}{n-4} \frac{1}{\mu} \int f P_{g_0} f\ \vol_{g_0} \geq 0.
\end{align}
In particular, the volume is increasing along the flow, while both $\mu$ and the Paneitz--Sobolev quotient are both decreasing:
\begin{align} \label{mudown} \begin{split}
\frac{d}{dt} \mu &= \frac{d}{dt} \Big(  \frac{ \int u P_{g_0} u\ \vol_{g_0}}{V} \Big) \leq 0, \\
\frac{d}{dt} \mathcal{F}_{g_0}[u] &= \frac{d}{dt} \Big( \frac{ \int u P_{g_0} u\ \vol_{g_0} }{ V^{\frac{n-4}{n}}}\Big) \leq 0.
\end{split}
\end{align}
Finally, the volume is bounded above:
\begin{align} \label{Vabove}
V \leq C_0(g_0).
\end{align}
\end{lemma}
\vskip.2in

\begin{corollary}[{\cite[Corollary  3.3]{MR3420504}}] \label{STCor}  
We have the space-time estimates
\begin{align} \label{stestimate} \begin{split}
&\int_0^T  \| f \|_{W^{2,2}}^2\ dt \leq C_1(g_0), \\
&\int_0^T \Big( \int |f|^{\frac{2n}{n-4}}\ \vol_{g_0} \Big)^{\frac{n-4}{n}}\ dt \leq C_2(g_0).
\end{split}
\end{align}
\end{corollary}

We begin with the following uniform estimates on the time derivatives of  $u$:
\begin{lemma}\label{lem:fint}
For the flow $u$, the following estimates hold:

(a) The quantities $\int f P_{g_0} f \dvol$, and $\int |f|^{\frac{2n}{n-4}} \dvol$ are uniformly bounded in $t$.

(b) $\big|\frac{d \mu}{dt}\big|$ is uniformly bounded in $t$.

(c) $\int \frac{\partial f}{\partial t} P_{g_0} (\frac{\partial f}{\partial t}) \dvol$, and $\int |\frac{\partial f}{\partial t}|^{\frac{2n}{n-4}} \dvol$ are uniformly bounded in $t$.
\end{lemma}

\begin{proof}

For (a), we compute 
\begin{align*}
\int f P_{g_0} f &= \int -f P_{g_0} u + \mu \int f u^{\frac{n+4}{n-4}} \\
&\le \frac{1}{4} \int f P_{g_0} f +  \int u P_{g_0} u + \mu \big( \int |f|^{\frac{2n}{n-4}} \big)^{\frac{n-4}{2n}} \cdot \big( \int |u|^{\frac{2n}{n-4}} \big)^{\frac{n+4}{2n}}    
\end{align*}
using Young's inequality.

Hence we have:
\begin{align*}
\int f P_{g_0} f & \le 4 \int u P_{g_0} u + 2\mu \big( \int |f|^{\frac{2n}{n-4}} \big)^{\frac{n-4}{2n}} \cdot \big( \int |u|^{\frac{2n}{n-4}} \big)^{\frac{n+4}{2n}} \\
& \le 4 \int u P_{g_0} u + 2 C_0 \mu \big( \int f P_{g_0} f \big)^{\frac{1}{2}} \cdot \big( \int |u|^{\frac{2n}{n-4}} \big)^{\frac{n+4}{2n}} 
\end{align*}
by the Sobolev inequality.
Since $\mu$ and the volume are bounded above, and $\int u P_{g_0} u$ is invariant along the flow (Lemma \ref{W22invariance}), this yields the desired bound.

(b) follows directly from Lemma \ref{volmon}, and (a).

For (c), we use the identity
$$\frac{\partial f}{\partial t} = - f + \frac{d\mu}{dt} P_{g_0}^{-1} \big( u^{\frac{n+4}{n-4}} \big) +  \mu \frac{(n+4)}{(n-4)} P_{g_0}^{-1} \big(  u ^{\frac{8}{n-4}} f \big)$$
to obtain
\begin{align*}
&\int \frac{\partial f}{\partial t} P_{g_0} ( \frac{\partial f}{\partial t} ) \\
=& \int  - \frac{\partial f}{\partial t} P_{g_0}f + \frac{d\mu}{dt} \frac{\partial f}{\partial t}  u^{\frac{n+4}{n-4}}  +  \mu \frac{(n+4)}{(n-4)} \frac{\partial f}{\partial t}  u ^{\frac{8}{n-4}} f \\
\le & \int \frac{1}{4} \frac{\partial f}{\partial t} P_{g_0}(\frac{\partial f}{\partial t})  +  f P_{g_0} f + \big| \frac{d\mu}{dt} \big| \big( \int |\frac{\partial f}{\partial t}|^{\frac{2n}{n-4}} \big)^{\frac{n-4}{2n}} \cdot \big( \int |u|^{\frac{2n}{n-4}} \big)^{\frac{n+4}{2n}}     \\
&+  \mu \frac{(n+4)}{(n-4)} \big( \int |\frac{\partial f}{\partial t}|^{\frac{2n}{n-4}} \big)^{\frac{n-4}{2n}} \cdot \big( \int |u|^{\frac{2n}{n-4}} \big)^{\frac{8}{2n}} \cdot \big( \int |f|^{\frac{2n}{n-4}} \big)^{\frac{n-4}{2n}}.
\end{align*}
All integrals involving $u$ and $f$ on the right-hand side are uniformly bounded by parts (a) and (b). Therefore, by the Sobolev inequality, we conclude that   $\int \frac{\partial f}{\partial t} P_{g_0} ( \frac{\partial f}{\partial t} )$ is uniformly bounded, which completes the proof.

\end{proof}

An immediate consequence of the above lemma is the following result, which plays an important role in showing that the standard bubbles are the only possible blow-up profiles.

\begin{corollary} \label{lem:flim}
$\int f P_{g_0} f$ and  $\int |f|^{\frac{2n}{n-4}} \rightarrow 0$ as $t \rightarrow \infty$.
\end{corollary}

\begin{proof}
We observe that 
$$\big| \frac{d}{dt} \int f P_{g_0} f \big| = \big|  \int \frac{\partial f}{\partial t} P_{g_0} f \big| \le \int f P_{g_0} f + \int \frac{\partial f}{\partial t} P_{g_0} ( \frac{\partial f}{\partial t}  ) $$
which is uniformly bounded by Lemma~\ref{lem:fint}. 
Similarly, by H\"older's inequality, $\big| \frac{d}{dt} \int |f|^{\frac{2n}{n-4}} \big|$ is uniformly bounded in $t$.
Now, the assertion follows from  Corollary \ref{STCor}.

\end{proof}

The following result will be used to apply the concentration-compactness theorem of Hebey and Robert.
\begin{proposition}\label{lem:LnormPf}
We have
\begin{equation*}
\int_M \big|P_{g_0} u - \mu_\infty u^{\frac{n+4}{n-4}} \big|^{\frac{2n}{n+4}} \dvol = C(n) \int_M u^{\frac{2n}{n-4}} \big| Q(t) - \overline Q_\infty \big|^{\frac{2n}{n+4}} \dvol  \rightarrow 0
\end{equation*}

or equivalently,
\begin{equation*}
\int_M |P_{g_0} f|^{\frac{2n}{n+4}} \dvol  \rightarrow 0
\end{equation*}
as $t \rightarrow \infty$.
\end{proposition}

\begin{proof}

By direct computation,
$$\int_M |P_{g_0} f|^{\frac{2n}{n+4}} \dvol =  \int_M \big|P_{g_0} u - \mu(t) u^{\frac{n+4}{n-4}} \big|^{\frac{2n}{n+4}} \dvol.$$
Since $\mu(t) \rightarrow \mu_\infty$, and $\int_M u^{\frac{2n}{n-4}} \dvol$ is uniformly bounded, it suffices to prove that 
\begin{equation*}
\int_M |P_{g_0} f|^{\frac{2n}{n+4}} \dvol  \rightarrow 0.
\end{equation*}

We compute
\begin{align*}
&\frac{d}{dt} \int_M |P_{g_0} f|^{\frac{2n}{n+4}} \dvol \\
=& \frac{2n}{n+4} \int_M |P_{g_0} f |^{-\frac{8}{n+4}}P_{g_0}f \big(-P_{g_0} f + \frac{d\mu}{dt} u^{\frac{n+4}{n-4}} + \mu \frac{n+4}{n-4} u^{\frac{8}{n-4}} f \big) \\ 
=& -\frac{2n}{n+4} \int_M |P_{g_0} f|^{\frac{2n}{n+4}} + \frac{2n}{n+4} \int_M |P_{g_0} f |^{-\frac{8}{n+4}}P_{g_0}f \big( \frac{d\mu}{dt} u^{\frac{n+4}{n-4}} + \mu \frac{n+4}{n-4} u^{\frac{8}{n-4}} f \big) \\
\le & -\frac{2n}{n+4} \int_M |P_{g_0} f|^{\frac{2n}{n+4}} +C(n) \big| \frac{d\mu}{dt} \big| \int_M |P_{g_0} f|^{\frac{n-4}{n+4}} u^{\frac{n+4}{n-4}} +C(n) \mu \int_M |P_{g_0} f|^{\frac{n-4}{n+4}} u^{\frac{8}{n-4}} |f|.
\end{align*}

By H\"older's inequality, we estimate
$$\int_M |P_{g_0} f|^{\frac{n-4}{n+4}} u^{\frac{n+4}{n-4}} \le \big(\int_M |P_{g_0} f|^{\frac{2n}{n+4}} \big)^\frac{n-4}{2n} \cdot \big( \int_M u^{\frac{2n}{n-4}} \big)^{\frac{n+4}{2n}}, $$
and
$$\int_M |P_{g_0} f|^{\frac{n-4}{n+4}} u^{\frac{8}{n-4}} |f| \le \big(\int_M |P_{g_0} f|^{\frac{2n}{n+4}} \big)^\frac{n-4}{2n} \cdot \big( \int_M u^{\frac{2n}{n-4}} \big)^{\frac{8}{2n}} \cdot \big(\int_M |f|^{\frac{2n}{n-4}} \big)^{\frac{n-4}{2n}}.$$

By Corollary \ref{lem:flim} and  Lemma \ref{volmon}, for any $\epsilon>0$, we have

$$\frac{d}{dt} \int_M |P_{g_0} f|^{\frac{2n}{n+4}} \le   -\frac{2n}{n+4} \int_M |P_{g_0} f|^{\frac{2n}{n+4}} + \epsilon \big(\int_M |P_{g_0} f|^{\frac{2n}{n+4}} \big)^\frac{n-4}{2n}$$
for all sufficiently large $t$. 

Let $F(t) = \int_M |P_{g_0} f|^{\frac{2n}{n+4}}$. Then the above inequality can be written as
$$F'(t) \le - \frac{2n}{n+4} F(t) + \epsilon F^{\frac{n-4}{2n}}. $$
Equivalently,
$$\big(e^{\frac{2n}{n+4}t} F \big)' \le \epsilon \big(e^{\frac{2n}{n+4}t} F \big)^{\frac{n-4}{2n}} e^t.$$
Integrating this differential inequality yields
$$e^{\frac{2n}{n+4}t} F(t) \le \epsilon (e^t -e^{t_0}) + e^{\frac{2n}{n+4}t_0} F(t_0)$$
for any sufficiently large $t_0$. This shows that $F(t) < \epsilon$ for all  sufficiently large  $t$.
\end{proof}

\subsection{Blow-up profile}

In this subsection, we prove that if the initial energy is sufficiently small, or if there is no volume concentration, then the flow converges globally. This is achieved by showing that the standard bubbles, which are the stationary solutions of the flow, are the only possible blow-up profiles.

We first observe that a uniform $C^0$ bound yields uniform lower bounds and  H\"older estimates for the flow. 

\begin{proposition}\label{prop:C0impl}
Assume that $\|u\|_{L^\infty([0, \infty) \times M )}<\infty$. Then, 
$$\sup_{t\in[0, \infty)}  \|u(t, \cdot) \|_{C^{5. \alpha}(M)}, \sup_{t\in[0, \infty)}  \|\frac{\partial u}{\partial t} (t, \cdot) \|_{C^{5. \alpha}(M)} <\infty $$
and 
$$\inf_{(t, x) \in [0, \infty) \times M } u(t, x) > 0$$
where $\alpha \in (0,1)$ depends on $n$.
\end{proposition}

\begin{proof}
Let $L := \| u\|_{L^\infty([0, \infty) \times M )}$. From the equation
\begin{equation}\label{eq:P4u}
\frac{\partial}{\partial t} P_{g_0} u = - P_{g_0} u + \mu u^{\frac{n+4}{n-4}},
\end{equation}
we obtain,
$$0 \le \frac{\partial}{\partial t} P_{g_0}u + P_{g_0} u \le \mu L^{\frac{n+4}{n-4}}.$$
Integrating this differential inequality yields
\begin{equation*}
P_{g_0} u(t, \cdot) = e^{-t} P_{g_0} u(0, \cdot) +   \int_0^t e^{s-t} \mu u(s, \cdot)^{\frac{n+4}{n-4}} ds .
\end{equation*}
Consequently,
\begin{equation*}
\|P_{g_0} u(t, \cdot) \|_{L^\infty(M)} \le e^{-t} \|P_{g_0} u(0, \cdot)\|_{L^\infty(M)} +   \int_0^t e^{s-t} \mu \|u(s, \cdot)^{\frac{n+4}{n-4}}\|_{L^\infty(M)} ds.
\end{equation*}
and similarly,
\begin{equation*}
\|P_{g_0} u(t, \cdot) \|_{C^{1, \alpha}(M)} \le e^{-t} \|P_{g_0} u(0, \cdot)\|_{C^{1, \alpha}(M)} +   \int_0^t e^{s-t} \mu \|u(s, \cdot)^{\frac{n+4}{n-4}}\|_{C^{1, \alpha}(M)} ds.
\end{equation*}

The first inequality shows that
$$\|P_{g_0} u \|_{L^\infty([0, \infty) \times M )}<\infty.$$
By the elliptic regularity, this implies that 
$\|u(t, \cdot)\|_{C^{3}(M)}$ is bounded uniformly. Applying a standard bootstrapping argument to the above inequality, we further obtain that $\|u(t, \cdot)\|_{C^{5,\alpha}(M)}$ is uniformly bounded for some $\alpha \in (0, 1)$ depending on the dimension $n$.

Recall that $f = \frac{\partial u}{\partial t}=- u + \mu P_{g_0}^{-1} \big( u^{\frac{n+4}{n-4}} \big)$. From the uniform $C^{5,\alpha}$-bound on $u$, we immediately obtain a uniform bound of $\|f(t, \cdot)\|_{C^{5,\alpha}(M)}$, which follows from the fact that
$$P_{g_0} f = - P_{g_0} u + \mu u^{\frac{n+4}{n-4}} \in L^\infty([0, \infty); C^{1,\alpha}(M) ).$$

For $\frac{\partial f}{\partial t}$, we compute
$$\frac{\partial f}{\partial t} = - f + \frac{d\mu}{dt} P_{g_0}^{-1} \big( u^{\frac{n+4}{n-4}} \big) +  \mu \frac{(n+4)}{(n-4)} P_{g_0}^{-1} \big(  u ^{\frac{8}{n-4}} f \big) \in L^\infty([0, \infty); C^{4, \alpha}(M) ). $$

We now show that $u$ is uniformly bounded below by a positive constant. Suppose, by contradiction, that this fails. Then there exists a sequence $(t_i,p_i) \in [0,\infty) \times M$ such that
\[
t_i \to \infty, \qquad p_i \to p \in M, \qquad u(t_i,p_i) \to 0.
\]
By the Arzel\`a--Ascoli theorem and the uniform $C^{5,\alpha}$-bounds, after passing to a subsequence, we have
\[
u(t_i,\cdot) \to u_\infty \quad \text{in } C^{5,\beta}(M), \qquad \beta < \alpha.
\]
It follows that $u_\infty \ge 0$, $u_\infty(p) = 0$, and
\[
P_{g_0} u_\infty = \mu_\infty u_\infty^{\frac{n+4}{n-4}} .
\]
By the strong maximum principle for the Paneitz operator \cite[Theorem A]{MR3420504}, we conclude that $u_\infty \equiv 0$. This contradicts the fact that the conformal volume is nondecreasing along the flow. Therefore, $u$ is uniformly bounded away from zero.

\end{proof}

\begin{proposition}\label{lem:globalC0bound}
Assume that either
\begin{equation*}
\cf(0) = \cf(1) < Y_4(S^n)
\text{ or } 
\limsup_{r \to 0}\,
\sup_{(t,p)\in[0,\infty)\times M}
\int_{B_r(p)} |u(t,x)|^{\frac{2n}{n-4}} \, \vol=0.
\end{equation*}
Then
\[
\sup_{(t,x)\in[0,\infty)\times M} u(t,x) < \infty .
\]
\end{proposition}

\begin{proof}

Suppose for contradiction that there exist $(t_i,p_i)$ with $t_i \to \infty$, $p_i \to p\in M$, and
\[
u(t_i,p_i)\to \infty,
\qquad
u(t_i,p_i)=\sup_{(t,x)\in[0,t_i]\times M} u(t,x).
\]
Set $\lambda_i := u(t_i,p_i)$, so that $\lambda_i \to \infty$, and define the rescaled functions
\[
u_i(t,x) := \lambda_i^{-1}\, u(t_i+t,x),
\qquad (t,x)\in[-t_i,0]\times M.
\]
Then $u_i(0,p_i)=1$ and $0<u_i\le 1$ on $[-t_i,0]\times M$.

Let $g_i := \lambda_i^{\frac{4}{n-4}} g_0$. We recall the conformal covariance of the Paneitz operator: 
$$P_{\lambda^{\frac{4}{n-4}} g_0 } \phi = \lambda^{-\frac{n+4}{n-4}} P_{g_0} (\lambda \phi ) = \lambda^{-\frac{8}{n-4}} P_{g_0} \phi;$$

$$P_{g_0}^{-1} \phi = \lambda^{-\frac{8}{n-4}} P_{\lambda^{\frac{4}{n-4}} g_0 }^{-1} \phi;$$

\begin{equation*}
\frac{\partial u_{\lambda_i}}{\partial t} = - u_{\lambda_i} + \mu P_{g_i}^{-1} \big( u_{\lambda_i}^{\frac{n+4}{n-4}}\big).
\end{equation*}

With this notation, the rescaled function $u_i$ satisfies
\begin{equation}\label{eq:rescaled_flow}
\frac{\partial u_i}{\partial t}
= -u_i + \mu(t_i+t)\, P_{g_i}^{-1}\!\big(u_i^{\frac{n+4}{n-4}}\big)
\qquad \text{on } [-t_i,0]\times M.
\end{equation}
Moreover, the quantities $\int |u|^{\frac{2n}{n-4}}$, $\int uP_{g_0} u$, and $\mu$
are invariant under the above constant rescaling.

Fix a small geodesic ball $B_r(p) \subset (M,g_0)$ centered at $p$. In geodesic normal coordinates near $p$, under the rescaled metric $g_i = \lambda_i^{\frac{4}{n-4}} g_0$, the ball $B_r(p)$ is mapped to
\[
B_{\lambda_i^{\frac{2}{n-4}} r}(p) \subset (M,g_i).
\]
Letting $i \to \infty$, so that $\lambda_i \to \infty$, we obtain the pointed convergence
\[
\left( B_{\lambda_i^{\frac{2}{n-4}} r}(p), g_i, p \right)
\longrightarrow (\mathbb{R}^n, g_e, 0)
\quad \text{in the pointed } C^\infty \text{ topology}.
\]

We use the uniform bound $\sup_{[-t_i,0]\times M} u_i \le 1$ to derive uniform H\"older estimates on compact subsets of space--time. Fix any $L>0$. For sufficiently large $i$, the function $u_i$ is well defined on the cylinder
\[
[-3L,0] \times B_{3L}(0),
\]
and satisfies
\[
0 \le \frac{\partial}{\partial t} P_{g_i} u_i + P_{g_i} u_i
= \mu\, u_i^{\frac{n+4}{n-4}}
\le \mu .
\]

Integrating the above differential inequality yields
\begin{equation*}
P_{g_i} u_i(t, \cdot) = e^{-t-t_i} P_{g_i} u_i(-t_i, \cdot) +   \int_{-t_i}^t e^{s-t} \mu u_i(s, \cdot)^{\frac{n+4}{n-4}} ds.
\end{equation*}
Since $P_{g_i} u_i ( -t_i, \cdot) \rightarrow 0$ (because the rescaled manifold is expanding), it follows that $P_{g_i} u_i \ge 0$ and uniformly bounded. By standard $L^p$-estimate, this implies
$$\|u_i(t, \cdot)\|_{C^{1, \alpha}(B_{\frac{5}{2} L} (0) )} < C(L).$$

Next, we have  the estimate
\begin{align*}
&\|P_{g_i} u_i(t, \cdot) \|_{C^{1, \alpha}(M)} \\
\le&e^{-t-t_i} \|P_{g_0} u_i(-t_i, \cdot)\|_{C^{1, \alpha}(M)} +   \int_{-t_i}^t e^{s-t} \mu \|u(s, \cdot)^{\frac{n+4}{n-4}}\|_{C^{1, \alpha}(M)} ds.
\end{align*}
Since $u_i$ is uniformly bounded, the inequality above together with interior Schauder estimates (and the fact that $P_{g_i}$ converges to $\Delta_{g_e}^2$ on compact sets) yields
\[
\|u_i(t, \cdot)\|_{C^{5,\alpha}(B_{2L}(0))} \le C(L).
\]

We next estimate $f_i$. Note that $f_i$ scales in the same way as $u_i$. Therefore, by Lemma~\ref{lem:fint}, the quantities
\[
\int_M |f_i|^{\frac{2n}{n-4}} \, \vol_{g_i}
\qquad \text{and} \qquad
\int_M \Big| \frac{\partial f_i}{\partial t} \Big|^{\frac{2n}{n-4}} \, \vol_{g_i}
\]
are uniformly bounded.

The equation 
$$P_{g_i} f_i = - P_{g_i} u_i + \mu u_i^{\frac{n+4}{n-4}},$$ 
and the local $L^p$-Schauder estimate
\begin{align*}
&\|f_i(t, \cdot)\|_{C^{5, \alpha}(B_{L} (0) )} \\
\le& C(L) \bigg( \|f_i(t, \cdot) \|_{L^{\frac{2n}{n-4}}(B_{2L} (0) )}  + \|\big(- P_{g_i} u_i + \mu u_i^{\frac{n+4}{n-4}} \big)(t, \cdot) \|_{C^{1, \alpha}(B_{2L}(0)} \bigg),
\end{align*}
provide that
$$\|f_i(t, \cdot)\|_{C^{5, \alpha}(B_{L} (0) )} < C(L).$$

Now we estimate $\frac{\partial f}{\partial t}$ from the equation
$$P_{g_i} \big( \frac{\partial f}{\partial t} \big)  = - P_{g_i} f + \frac{d\mu}{dt}   u^{\frac{n+4}{n-4}}  +  \mu \frac{(n+4)}{(n-4)}    u ^{\frac{8}{n-4}} f $$
and Lemma \ref{lem:fint}. Using the local $L^p$-Schauder estimate, we prove that 
$$\big\|\frac{\partial f_i}{\partial t}(t, \cdot) \big \|_{C^{4, \alpha}(B_{L/2} (0) )} < C(L).$$

In conclusion, we have shown that there exists a function $w: (-\infty, 0] \times \mathbb{R}^n \rightarrow \mathbb{R}$ such that, after passing to a subsequence,
\[
u_i(t,\cdot) \longrightarrow w(t,\cdot)
\quad \text{in } C^{4,\beta}_{\mathrm{loc}}\big((-\infty,0]\times\mathbb{R}^n\big),
\]
and
\[
\partial_t u_i(t,\cdot) \longrightarrow \partial_t w(t,\cdot)
\quad \text{in } C^{4,\beta}_{\mathrm{loc}}\big((-\infty,0]\times\mathbb{R}^n\big),
\]
for some $0<\beta < \alpha<1$.

Moreover, the limit function $w$ satisfies
\[
\frac{\partial}{\partial t} (\Delta^2 w)
= - \Delta^2 w + \mu_\infty \, w^{\frac{n+4}{n-4}},
\qquad
\mu_\infty := \lim_{t\to\infty} \mu(t).
\]
On the other hand, by Fatou’s lemma and Lemma~\ref{lem:flim},
\[
\int_{\mathbb{R}^n} \Big| \frac{\partial w}{\partial t} \Big|^{\frac{2n}{n-4}} \, dx
\le \lim_{t\to\infty} \int_M \Big| \frac{\partial u}{\partial t} \Big|^{\frac{2n}{n-4}} \, \vol
= 0 .
\]
Therefore $\frac{\partial w}{\partial t} \equiv 0$, and $w$ is in fact a stationary solution. Consequently,
\begin{equation}\label{eq:blowup}
\Delta^2 w = \mu_\infty \, w^{\frac{n+4}{n-4}}
\qquad \text{on } \mathbb{R}^n .
\end{equation}

Note that $w(0) = 1$ and $w \ge 0$. Moreover, since the positivity of the scalar curvature is preserved along the flow, it follows that $\Delta w \le 0$. By the strong maximum principle, we conclude that $w > 0$ everywhere.

By the Liouville theorem \cite{MR1679783} applied to equation~\eqref{eq:blowup}, the limit profile $w$ must be a standard bubble of the form
\[
w(x) = C \left( \frac{\lambda}{1 + \lambda^2 |x|^2} \right)^{\frac{n-4}{2}} .
\]
Let $\mathcal{F}_{g_e}$ denote the Paneitz--Sobolev quotient with respect to the Euclidean metric. Then
\[
\mathcal{F}_{g_e}[w]
= \mu_\infty \left( \int_{\mathbb{R}^n} |w|^{\frac{2n}{n-4}} \, dx \right)^{\frac{4}{n}}
\le \mu_\infty V_\infty^{\frac{4}{n}}
= \mathcal{F}_\infty
\le \mathcal{F}(0)
< Y_4(S^n),
\]
where the inequality follows from Fatou’s lemma.

This contradicts the sharp Sobolev inequality on $\mathbb{R}^n$, which implies that any standard bubble achieves the optimal constant $Y_4(S^n)$. Therefore, the assumption that $u$ is unbounded from above is false, and $u$ must be uniformly bounded above.

\end{proof}

\begin{proposition}[\loja--Simon inequality]\label{lem:LSineq}
Under assumption of $C^0$ upper bound, and suppose we have $\mathcal{F}_\infty = \mathcal{F}(w)$ for some limiting function $w$. If $\|u(t) - w\|_{C^{4,\alpha}} < \delta$, for sufficiently small $\delta$, we have 
$$ \big\|\frac{\partial u}{\partial t} \big\|_{W^{2,2}(M)} \ge C |\mathcal{F}(u(t)) -\mathcal{F}_\infty |^{1-\theta}$$
for some $\theta \in (0, \frac{1}{2})$.
\end{proposition}

\begin{proof} 
Given that the flow stays inside the open subset of H\"older space:
$$\Omega := \{ v \in C^{5, \alpha(n)}(M) \mid \|v\|_{C^{5, \alpha(n)}(M)} < C_0, \inf v > 1/C_0 \},$$
the energy, or volume-normalized total $Q$-curvature is an analytic functional on this domain $\Omega$. Simon's general theorem on the uniqueness of the limit of the flow should imply that our flow converges to a unique limit. Since a  Łojasiewicz-Simon inequality for more general cases will be proven in the later sections, we omit the proof.

\end{proof}
Given that the \loja--Simon inequality holds, it is standard to prove that the flow converges globally.

\section{Variational theories for $Q$-curvature and Weyl tensor}

In this section, we compute the first and the second variations of the
$Q$-curvature, and also establish a fourth-order Bochner identity, which
plays a crucial role in the stability inequality in the next section.

\subsection{Deformations of $Q$-curvature}

\begin{definition}
For any geometric quantity or operator $F_g$ (such as the Ricci curvature,
volume density, Paneitz operator, etc.) depending on the metric $g$, and for
$h \in S_{2}(M)$, we define
\begin{align*}
DF_g[h] &= \frac{d}{dt}\bigg|_{t=0} F_{g+th}, \\
D^2 F_g[h,h] &= \frac{d^2}{dt^2}\bigg|_{t=0} F_{g+th}.
\end{align*}
For $h_1, h_2 \in S_2(M)$, we further define
\[
D^2 F_g[h_1,h_2]
= \frac{1}{2}\Big(
D^2 F_g[h_1+h_2,h_1+h_2]
- D^2 F_g[h_1,h_1]
- D^2 F_g[h_2,h_2]
\Big).
\]
\end{definition}

According to \cite{MR380907}, also see, \cite{MR3521090}, we have the following lemma. 
\begin{lemma}\label{lem:einstein R}
Let $h\in S_2(M)$, then
$$
DR_g[h]=\ddiv^2 h-\Delta(\tr h)-\< h,\ric\>.
$$
Moreover, if $h\in S_{2,g}^{T,T}(M)$, we have
\begin{align*}
D\ric_g[h]=&-\frac{1}{2}\Delta_Lh,\\
D^2R_g[h,h]=&-2DR_g[h \times h]-\Delta(|h|^2)-\frac{1}{2}|\nabla h|^2+\nabla_{i}h_{jk}\nabla^jh^{ik},
\end{align*}
\end{lemma}

If there is no ambiguity, we will simply write $D\ric$, $D^2\ric$, $DR$, $D^2R$ without $g$ and $h$.

According to \cite{gong2025compactnessnoncompactnesstheoremsfourth}, we have the following lemma.
\begin{lemma}\label{lem:expansion of P}
Let $\phi $ be a function and $\tr h=0$. There exist positive constants $C$ depending only on $n$, such that
\begin{align*}
(P_{\exp(h)}-\Delta^2)\phi\leq& C\Bigg(\sum_{\substack{0 \le \alpha,\beta \le 4 \\ \alpha+\beta=4}} |\partial^{\alpha}h| |\partial^{\beta}\phi|\Bigg),\\
(P_{\exp(h)}-\Delta^2-DP_{g_e}[h]) \phi\leq& C\Bigg(\sum_{\substack{0 \le \alpha_1,\alpha_2,\beta \le 4 \\ \alpha_1+\alpha_2+\beta=4}} |\partial^{\alpha_1}h||\partial^{\alpha_2}h| |\partial^{\beta}\phi|\Bigg),
\end{align*}
and
\begin{align*}
&\bigg( P_{\exp(h)} - \Delta^2 - DP_{g_e}[h] - \frac{1}{2} D^2P_{g_e} [h, h]  
-  \frac{1}{2} DP_{g_e} [h\times h]\bigg)\phi\\
\leq &C\Bigg(\sum_{\substack{0 \le \alpha_1,\alpha_2,\alpha_3,\beta \le 4 \\ \alpha_1+\alpha_2+\alpha_3+\beta=4}} |\partial^{\alpha_1}h||\partial^{\alpha_2}h||\partial^{\alpha_3}h| |\partial^{\beta}\phi|\Bigg),
\end{align*}
where the linearization of Paneitz operator at the Euclidean metric is
\begin{align*}
DP_{g_e}[h] \phi&=-\Delta(\partial_i(h_{ij}\partial_j \phi))-\partial_i(h_{ij}\partial_j \Delta \phi)+\frac{4}{n-2}D\ric[h]_{ij}\partial^2_{ij}\phi\\
&-\alpha_n^1DR[h]\Delta \phi- \alpha_n^2\partial_k(DR[h])\partial_k \phi-\frac{n-4}{4(n-1)}\Delta (DR[h]) \phi,
\end{align*}
Here, $\alpha_n^1 = \frac{4+(n-2)^2}{2(n-2)(n-1)}$ and $\alpha_n^2 = \frac{n-6}{2(n-1)}$.
\end{lemma}
\noindent\textbf{Remark.}
The explicit formula of $D^2P_{g_e}[h, h] $ can be seen in \cite{gong2025compactnessnoncompactnesstheoremsfourth}. However, it is not needed in our proof. One of our main contributions is to shorten the computations of the second variation using $D^2Q_{g_c}$ instead of $D^2P_{g_e}$. 

\begin{lemma}\label{lem:einstein Q}
Let $g$ be Einstein such that $\ric_g=\lambda g$. Let $c(n)=\frac{2}{\alpha_n}(\beta_n+n\gamma_n)=-\frac{n^2-4}{2(n-1)}$. Then,
\begin{equation}\label{eq:DQ einstein}
DQ_g[h]=\alpha_n(\Delta+c(n)\lambda)DR_g[h].  
\end{equation}
Moreover, if $h\in S^{T,T}_{2,g}(M)$, then linearization vanishes
$$
DQ_g[h]=DR_g[h]=0,
$$
and the second variation is
\begin{equation}\label{eq:D2Q einstein}
\begin{aligned}
D^2Q_g[h,h]=&\alpha_n(\Delta+c(n)\lambda)D^2R\\
+&2\beta_n(\lambda^2|h|^2-2\lambda\< D\ric, h\>+|D\ric|^2).
\end{aligned}
\end{equation}
\end{lemma}
\begin{proof}
Using $R_g=\tr_g\ric_g$, there are
\begin{align}
DR&=\tr D\ric-\<\ric,h\> , \label{eq:DR, DRic}\\
D^2R&=\tr D^2\ric-2\< D\ric,h\>+2\< \ric,h\times h\>. \label{eq:D2R D2Ric} 
\end{align}

According to \cite[Proposition 3.3]{MR3529121}, the linearization of $Q$-curvature is
\begin{align*}
DQ_g[h]&=\alpha_n\left(\Delta DR-\< h, \nabla^2 R \>-\< B[h], dR\> \right)
\\&+2\beta_n\< \ric, D\ric-\ric\times h\>+2 \gamma_n  R DR,
\end{align*}
where $B[h]_{i}:= \ddiv_gh-\frac{1}{2}d\tr_gh$. Combining with $\ric=\lambda g$ and \eqref{eq:DR, DRic}, we can obtain \eqref{eq:DQ einstein}.

According to \cite[Lemma 5.2]{MR3529121}, the second variation of $Q$-curvature is
\begin{align*}
D^2Q[h,h]&=\alpha_n\bigg[\Delta D^2R-2\< h, \nabla^2 DR \>-2\< B[h], d(DR)\> +2\< h\times h,\nabla^2 R\> \\
&\hspace{60pt}+h^{ij}(2h_{lj;i}-h_{ij;l})g^{lk}R_{,k}+h^{ij}(2h_{ik;l}-h_{kl;i})g^{lk}R_{,j}\bigg]\\
&+2\beta_n\bigg[\<\ric, D^2\ric\>+\< D\ric, D\ric-4\ric\times h\>+|\ric\times h|^2\\
&\hspace{60pt}+2\< \ric\times \ric,h\times h\>\bigg]+2 \gamma_n  [R D^2R+(DR)^2].
\end{align*}
Combining with $\ric=\lambda g$, \eqref{eq:D2R D2Ric}, and $h\in S^{T,T}_{2,g}(M)$, we obtain \eqref{eq:D2Q einstein}. 
\end{proof}

\begin{proposition}\label{prop:D2Q TT direction}
Let $g$ be Einstein with $\ric_g=\lambda g$, and $h\in S^{T,T}_{2,g}(M)$.
Let $U\subset M$ be a domain with smooth boundary, and let $\hat n$ denote the outward
unit normal vector field along $\partial U$. Then
\begin{align*}
\int_U D^2 Q_g[h, h] \vol_g=&\frac{n^2-4}{4(n-1)^2}\lambda\int_U-\frac{1}{2}|\nabla h|^2+\lambda |h|^2+\rm_{iklj}h^{kl}h^{ij}\vol_g\\
&-\frac{1}{(n-2)^2}\int_{U}|\Delta_L h+2\lambda h|^2\vol_g+\textup{BT},
\end{align*}
where the boundary terms $\textup{BT}$ are given by
\begin{align*}
\textup{BT}=&-\frac{1}{2(n-1)}\int_{\partial U}\partial_{\hat{n}} D^2R_g[h,h]dS_g\\
&+\frac{n^2-4}{4(n-1)^2}\lambda\int_{\partial U}\partial_{\hat{n}}|h|^2-\ddiv(h\times h)(\hat{n})dS_g.
\end{align*}
\end{proposition}
\begin{proof}
Since $g$ is Einstein and $h\in S^{T,T}_{2,g}(M)$, we have that
$$
DR[h]=0,\quad \text{ and } D\ric[h]=-\frac{1}{2}\Delta_Lh.
$$
Combining with Lemma \ref{lem:einstein Q}, we obtain
$$
D^2Q_g[h,h]=\alpha_n(\Delta+c(n)\lambda)D^2R[h,h]+\frac{1}{2}\beta_n|\Delta_Lh+2\lambda h|^2.
$$

Also, by Lemma \ref{lem:einstein R},
\begin{align*}
D^2R_g[h,h]&=-2\ddiv^2(h\times h)+\Delta(|h|^2)+2\lambda |h|^2-\frac{1}{2}|\nabla h|^2+\nabla_{i}h_{jk}\nabla^jh^{ik}.
\end{align*}
We can use divergence theorem, Ricci identity and $\ddiv h=0$, to obtain
\begin{align*}
\int_{U} D^2R[h,h] \vol_{g}=&\int_{U}-\frac{1}{2}|\nabla h|^2+\lambda |h|^2+\rm_{iklj}h^{kl}h^{ij}\vol_{g}\\
+&\int_{\partial U}-\ddiv(h\times h)(\hat{n})+\partial_{\hat{n}}|h|^2dS_g.
\end{align*}
Also,
$$
\int_{U}\Delta D^2R[h,h] \vol_{g}=\int_{\partial U}\partial_{\hat{n}}D^2R[h,h]dS_g.
$$
Combining the two integration by parts formulas above with the computation of $D^2Q_g[h,h]$, we can conclude this proposition.
\end{proof}

\subsection{Diffeomorphism invariance of curvatures}

Let $X$ be a vector field on $M$, and let $\{\phi_t\}_{t\in\mathbb{R}}$ be the associated one-parameter family of diffeomorphisms generated by $X$.

By diffeomorphism invariance of curvature, we have
$$Q_g (\phi_t(x) ) = Q_{\phi_t^* g}(x).$$
Differentiating at $t=0$ yields
\begin{equation*}
D{Q}_g [ \mathcal{L}_X g] = \mathcal{L}_X Q_g= \nabla_X Q.
\end{equation*}

More generally, for any metric-dependent geometric quantity or operator $F_g$, if we set $g(t)=\phi_t^*g$, then
\[
\frac{d}{dt}F_{g(t)}\Big|_{t=0}=DF_g[g'(0)],
\]
and
\begin{equation*}
\frac{d^2}{dt^2} \bigg|_{t= 0 }F_{g(t)} = D^2{F_g}[g'(0), g'(0) ] + DF_g[g''(0)].
\end{equation*}

In particular, since $g'(0)=\mathcal{L}_X g$ and $g''(0)=\mathcal{L}_X^2 g$, we obtain
\begin{equation*}
\cL_X^2 F_g = D^2 {F_g}[\cL_X g, \cL_X g] + DF_g[\cL_X^2 g].
\end{equation*}

\begin{lemma}
Let $g(t)$ be a one-parameter family of metrics with $g(0)=g$ and $g'(0)=h$. Then:
\begin{enumerate}[label=(\alph*),leftmargin=*]
\item $D\vol_g  [h] = \frac{1}{2} \tr_g h \vol_g;$
\item $D^2\vol_g[h, h] = \big( \frac{1}{4}(\tr_g h)^2 - \frac{1}{2} |h|^2_g \big) \vol_g ;$
\item ${\left(\vol_{g(t)}\right)}''(0) = \big( \frac{1}{4}(\tr_g h)^2 - \frac{1}{2} |h|^2_g  + \frac{1}{2} \tr_g g''(0) \big) \vol_g .$
\end{enumerate}

\end{lemma}

\begin{lemma}\label{lem:Lie derivatives of g}
Let $X$ be a vector field. Then the Lie derivatives of the metric satisfy:
\begin{enumerate}[label=(\alph*),leftmargin=*]
\item $(\cL_X g)_{ij} = \nabla_i X_j + \nabla_j X_i $;
\item $(\cL_X^2g)_{ij} = X^k \nabla_k (\nabla_i X_j + \nabla_j X_i) + \nabla_i X^k (\nabla_k X_j + \nabla_j X_k) + \nabla_j X^k (\nabla_k X_i + \nabla_i X_k) $;
\item $\tr_g(\cL_X^2 g) = 2 X(\ddiv X) +  |\cL_X g|^2$.
\end{enumerate}
\end{lemma}
\begin{proof}
We use the general identity for any $(0,2)$-tensor $T$,
$$
(\cL_X T)_{ij} = X^k \nabla_k T_{ij} +   \nabla_i X^k T_{jk} +\nabla_j X^k T_{ik}.
$$
\end{proof}

\begin{lemma}\label{prop:lievol}
We have the following formulas for the derivatives or the volume density:
\begin{enumerate}[label=(\alph*),leftmargin=*]
\item $D{\vol}_g [ \cL_X g] = \cL_X ( \vol_g )= (\ddiv X )\vol_g;$
\item $\cL_X ( \cL_X \vol_g ) = X(\ddiv X) \vol_g + (\ddiv X)^2 \vol_g.$
\end{enumerate}
\end{lemma}

\begin{proof}
The proof is standard.
\end{proof}

\begin{proposition}\label{prop:D2Q S direction}
Let $X$ be a vector field on a closed manifold $M$, and let $U\subset M$ be a domain with smooth boundary.
Let $\hat n$ be the outward unit normal vector field along $\partial U$.
Assume that $g$ is Einstein with $\ric_g=4(n-1)\,g$.
\begin{enumerate}[label=(\alph*),leftmargin=*]
\item For the first differential of the $Q$-curvature, we have
\begin{equation*}
\int_U D{Q}_g[\cL_X g] \vol_g = 0. 
\end{equation*}
\item For the second differential of the $Q$-curvature, we have
\begin{equation*}
\begin{aligned}
\int_U D^2{Q}_g[\cL_X g, \cL_X g] \vol_g =&  4(n^2-4)\int_U - 2 (\ddiv X)^2 + |\cL_X g|^2 \vol_g +\textup{BT}
\end{aligned}
\end{equation*}
where the \textup{BT} (boundary terms) are
\begin{align*}
\textup{BT}=&\frac{n^2-4}{(n-1)^2}\int_{\partial U}\partial_{\hat{n}}\tr(\cL^2_X g)-\ddiv(\cL^2_X g)(\hat{n})dS_g\\
+&\frac{1}{2(n-1)}\int_{\partial U}\partial_{\hat{n}} DR_g[\cL_X^2 g]dS_g+8(n^2-4)\int_{\partial U}\ddiv X\<X,\hat{n}\>dS_g.
\end{align*}
\item If $h\in S_{2,g}^{T,T}(M)$, then we have
\begin{equation*}
\int_U D^2 Q_g[ \cL_X g, h] \vol_g = 0 
\end{equation*}
\item If $g$ is conformally flat, then
$$
DW_g[\cL_X g]=\cL_X W_g= 0.
$$
\end{enumerate}
\end{proposition}

\begin{proof}

For (a), we have  
\begin{align*}
\int_U D{Q}_g[\cL_X g] \vol_g & = \int_U X \cdot \nabla Q_g \vol_g = 0.
\end{align*}

Next, for (b), we compute
\begin{align*}
& \int_U D^2{Q}_g[\cL_X g, \cL_X g] \vol_g  \\
=&  \int_U \cL_X ( \cL_X Q_g ) - DQ_g [ \cL_X^2 g ] \vol_g \\
=& 4(n^2-4)\int_U  tr_g(\cL_X^2 g) \vol_g +\frac{1}{2(n-1)}\int_{\partial U}\partial_{\hat{n}} DR_g[\cL_X^2 g]dS_g \\
& +\frac{n^2-4}{(n-1)^2}\int_{\partial U}\partial_{\hat{n}}\tr(\cL^2_X g)-\ddiv(\cL^2_X g)(\hat{n})dS_g.
\end{align*}
Moreover, by Lemma~\ref{lem:Lie derivatives of g},
\begin{align*}
\int_U  tr(\cL_X^2 g) \vol_g &= \int_U 2 X(\ddiv X) +  |\cL_X g|^2 \vol_g\\
&= \int_U - 2 (\ddiv X)^2 + |\cL_X g|^2 \vol_g + 2\int_{\partial U}\ddiv X\<X,\hat{n}\>dS_g.
\end{align*}
Combining the two identities yields (b).

We now prove (c). Let $g(s)$ be a one-parameter family of metrics with $g(0)=g$ and $g'(0)=h$.
From the identity
\begin{equation*}
Q_{\phi^*_t g(s)}(x) = Q_{g(s)} ( \phi_t(x))  
\end{equation*}
we obatin
\begin{align*}
\frac{d}{dt} \bigg|_{t=0} Q_{\phi_t^* g(s)} (x)= \cL_X Q_{g(s)}(x)
\end{align*}
Therefore,
\begin{align*}
\frac{d}{ds} \bigg|_{s=0} \frac{d}{dt} \bigg|_{t=0} Q_{\phi_t^* g(s)} (x) & = \frac{d}{ds} \bigg|_{s=0} \cL_X Q_{g(s)} (x) \\
& = \frac{d}{ds} \bigg|_{s=0} dQ_{g(s)} (X) = d \bigg( \frac{d}{ds} \bigg|_{s=0} Q_{g(s)} \bigg)(X) \\
& = \nabla_X  (D{Q}_g [ h])
\end{align*}
Consequently,
\begin{align*}
&\int_U D^2 Q_g[ \cL_X g, h] \vol_g  = \int_U \nabla_X (D{Q}_g [ h]) \vol_g \\
=& -\int_U \ddiv X (D{Q}_g [ h]) \vol_g + \int_{\partial U} X  \cdot \hat{n} (D{Q}_g [ h]) dS_g = 0 
\end{align*}
where the last equality follows from Lemma~\ref{lem:einstein Q}.

\end{proof}

\subsection{Higher-order Koiso-Bochner formula}

As in \cite{BGIMazzeo}, we define
\begin{align*}
(\nabla^Kh)_{ijk}:=&h_{ij;k}-h_{kj;i},\\
\mathcal{A}[h]_{ijkl}:=&\nabla_l(\nabla^Kh)_{ijk}=h_{ij;kl}-h_{kj;il}.
\end{align*}

The following lemma is well known for experts, see e.g. Besse's book \cite[Chapter 1.K]{MR867684}. 
\begin{lemma}
The linearization of Riemann tensor is
$$
D\rm_g[h]_{ijkl}=\frac{1}{2}(h_{ki;lj}+h_{lj;ki}-h_{li;kj}-h_{kj;li}+\rm\indices{_{ijk}^{p}}h_{pl}+\rm\indices{_{ijl}^p}h_{pk}).
$$
Moreover, if $h\in S_{2,g}^{T,T}(M)$ and  $g$ is with constant sectional curvature, i.e. $(\rm_g)_{ijkl}=c(g_{il}g_{jk}-g_{ik}g_{jl})$, then there are
\begin{align}
DA_g[h]=&-\frac{1}{2(n-2)}(\Delta -c n)h,\nonumber\\
DW_g[h]_{ijkl}=&\frac{1}{2}(\mathcal{A}[h]_{kilj}+\mathcal{A}[h]_{ljki})+\frac{1}{2(n-2)}((\Delta -cn)h\KN g)_{ijkl}.\label{eq:DW}
\end{align}
\end{lemma}

Given a $k$-tensor $h$ with components $h_{\alpha}$ where $\alpha$ is a multi-index with $|\alpha|=k$. We denote the $L^2$-norm over an open domain $U\subset M$:
$$
\|h\|_U^2:=\int_{U}h_{\alpha}h^{\alpha}\vol_g.
$$
Moreover, we define three types of boundary integrals:
\begin{align*}
\textup{BT}_1:=&\int_{\partial U}\contr(g^{-1}\otimes g^{-1}\otimes \nabla h\otimes h\otimes \hat{n})dS_g\\
\textup{BT}_2:=&\int_{\partial U}\contr(g^{-1}\otimes g^{-1}\otimes g^{-1}\otimes \nabla^2 h\otimes \nabla h\otimes \hat{n})dS_g\\
\textup{BT}_3:=&\int_{\partial U}\contr(g^{-1}\otimes g^{-1}\otimes g^{-1}\otimes \nabla^3 h\otimes h\otimes \hat{n})dS_g,
\end{align*}
where $\otimes$ denotes tensor product and $\contr$ denotes full tensor contraction. Then, Koiso-Bochner formula and the higher order version can be written as the following.
\begin{lemma}[{\cite[Lemmas 4.3, 4.4]{BGIMazzeo}}]\label{lem:two koiso identities}
Let $h\in S_{2,g}^{T,T}(M)$ and $g$ be with constant sectional curvature. Then,
\begin{align}
\frac{1}{2}\|\nabla^Kh\|_U=&\|\nabla h\|_U^2+cn\|h\|_U^2+\textup{BT}_1,\label{eq:koiso identity 2nd order}\\
\frac{1}{2}\|\mathcal{A}[h]\|_U^2=&\|\Delta h\|_U^2+3c\|\nabla h\|_U^2-c^2n(n-3)\|h\|_U^2+\sum_{i=1}^3\textup{BT}_i.\label{eq:koiso identity 4th order}
\end{align}
\end{lemma}

Then, we can study the second variation of $|W_g|^2$, which is one type of conformally variational Riemannian invariants (CVI) defined in \cite{MR3955546}.

\begin{proposition}\label{prop:DW koiso identity}
Let $h\in S_{2,g}^{T,T}(M)$ and $g$ be with constant sectional curvature. Then,
$$
\|DW_g[h]\|_U^2=\frac{n-3}{n-2}\left(\|\Delta h\|_U^2+c(n+2)\|\nabla h\|_U^2+2c^2n\|h\|_U^2\right)+\sum_{i=1}^3\textup{BT}_i.
$$
\end{proposition}
\begin{proof}
To simplify the notations, we drop the subscript $U$ and write $\|\cdot\|:=\|\cdot\|_U$ in this proof. According to \eqref{eq:DW}, there is
\begin{align*}
\|DW_g[h]\|^2=&\frac{1}{2}\|\mathcal{A}\|^2+\frac{1}{4(n-2)^2}\|(\Delta -cn)h\KN g\|^2\\
&+\frac{1}{2}\int_U\mathcal{A}_{kilj}\mathcal{A}^{ljki}\vol_g+\frac{1}{n-2}\int_U\mathcal{A}_{kilj}((\Delta-cn)h\KN g)^{ijkl}\vol_g\\
:=&\RN{1}+\RN{2}+\RN{3}+\RN{4}.
\end{align*}

For term \RN{1}, we can directly use equation \eqref{eq:koiso identity 4th order}. For term \RN{2}, we just use the definition of Kulkarni-Nomizu product to compute:
\begin{align*}
\|(\Delta-cn)h\KN g\|^2=&4(n-2)(\|\Delta h\|^2+2cn\|\nabla h\|^2+c^2n^2\|h\|^2)+\textup{BT}_1.
\end{align*}

Using Ricci identity, we can deal with term \RN{4}.
$$
\int_U\mathcal{A}[h]_{kilj}((\Delta -cn)h\KN g)^{ijkl}\vol_g=-2(\|\Delta h\|^2+2cn\|\nabla h\|^2+c^2n^2\|h\|^2)+\textup{BT}_1.
$$
Applying Ricci identity twice, we have
\begin{align*}
h\indices{_{ik;jl}^k}=&c(n+1)h_{ij;l}+cnh_{il;j}.
\end{align*}
Combining \eqref{eq:koiso identity 2nd order} with the equation above, we can deal with term \RN{3}:
$$
\int_U\mathcal{A}[h]_{kilj}\mathcal{A}[h]^{ljki}\vol_g=2cn\|\nabla h\|^2+2c^2n^2\|h\|^2+\textup{BT}_1+\textup{BT}_2.
$$
Combining the four terms together, we obtain that
$$
\|DW_g[h]\|^2=\frac{n-3}{n-2}\left(\|\Delta h\|^2+c(n+2)\|\nabla h\|^2+2c^2n\|h\|^2\right)+\sum_{i=1}^3\textup{BT}_i.
$$
\end{proof}

\section{Construction of test bubbles: $n \ge 8$ case}

Consider a metric $g=\exp (h)$ in a small normal neighborhood $B_{\delta}(p)$, where $h$ is a symmetric 2-tensor satisfying $\tr h=O(|x|^{2d+2})$ and $x_ih_{ij}(x)=0$, for any $x$, $j$.

Define a positive function $u_\epsilon : \mathbb{R}^n \rightarrow \mathbb{R}$ by
$u_\epsilon (x) =  \big( \frac{\epsilon}{\epsilon^2 + |x|^2} \big)^{\frac{n-4}{2}}$.
It is straightforward to see that
\begin{equation}\label{eq:bubble pde}
\Delta^{2} u_{\epsilon} = \mfc(n) u_{\epsilon}^{\frac{n+4}{n-4}}, \quad \text{in } \mathbb{R}^n.
\end{equation}

Given $k=2,\ldots d$, a multi-index $\alpha$, we define
\begin{equation*}
H^{(k)}_{ij}(x) = \sum_{|\alpha|=k} h_{ij,\alpha} x_{\alpha}, \quad H_{ij}(x) = \sum_{k=2}^{d} H^{(k)}_{ij}(x),
\end{equation*}
and
\begin{equation*}
|H^{(k)}|^2 = \sum_{|\alpha|=k} |h_{ij,\alpha}|^2,
\end{equation*}
Then, 
$$
h_{ij}(x) = H_{ij}(x) + O(|x|^{d+1}),
$$
and $H_{ij}(x) = H_{ji}(x)$, $H_{ij}(x)x_j = 0$, and $ H_{ii}= 0$.

Fix a non-negative smooth function $\chi$, such that
$$
\chi(s)=\begin{cases}
1, & \text{when }s\leq \frac{4}{3}\\
0, & \text{when }s\geq \frac{5}{3}. 
\end{cases}
$$
Denote $\chi_{\delta}(x)=\chi(|x|/\delta)$.

\subsection{Splitting of $S_2$ on standard sphere}

Brendle's splitting can be summarized as the following lemma. 
\begin{lemma}[{\cite[Corollaries 22, 24]{MR2357502}}]\label{lem:brendle splitting}
Given any $\delta>0$, consider a traceless symmetric 2-tensor $\chi_\delta H$ on $\mathbb{R}^n$ with compact support $\supp (\chi_{\delta} H) \subseteq B_{2\delta}$, there exists a smooth vector field $X=X_{\epsilon,\delta}$ on $\mathbb{R}^n$ such that $\chi_{\delta}H=T+S$ and
\begin{align*}
\tr T&=0, \text{ and }\ddiv(u^{\frac{2n}{n-4}}_{\epsilon} T)=0,\\
\tr S&=0, \text{ and }S=\cL_Xg_e-\frac{2}{n}\ddiv X g_e.
\end{align*}
Moreover, for every multi-index $\alpha$, there is a positive constant $C=C(n,|\alpha|)$ such that
\begin{align*}
|\partial^{\alpha}X_{\epsilon, \delta} |(x)\leq C\sum_{k=2}^{d}|H^{(k)}|(\epsilon+|x|)^{k+1-|\alpha|} \text{ for all } x\in \mathbb{R}^n.
\end{align*}
\end{lemma}

According to \cite{gong2025compactnessnoncompactnesstheoremsfourth}, we have the following lemma.
\begin{lemma}\label{lem:correction term}
Given $S=\cL_Xg_e-\frac{2}{n}\ddiv X g_e$, the linearized equation
\begin{equation}\label{eq:corrector pde}
\Delta^{2} w_{\epsilon,\delta}[S] -\tmfc(n) u_{\epsilon}^{\frac{8}{n-4}} w_{\epsilon,\delta}[S]=-DP_{g_e}[S]u_{\epsilon}, \quad \text{in } \mathbb{R}^n,
\end{equation}
has a solution in the following form
\begin{equation}\label{eq:w[s] proof}
w_{\epsilon,\delta}[S]=X_i\partial_i u_{\epsilon}+\frac{n-4}{2n}\ddiv X u_{\epsilon}.
\end{equation}
We simply denote $w[S]=w_{\epsilon,\delta}[S]$.
\end{lemma}

Denote the sphere metric $g_c:=u_{\epsilon}^{\frac{4}{n-4}}g_{e}$ and the connection $\bar{\nabla}:=\nabla^{g_c}$. Then, the curvatures are
\begin{align*}
(\rm_{g_c})_{ijkl}&=4(g_c)_{il}(g_c)_{jk}-4(g_c)_{ik}(g_c)_{jl},\\
\ric_{g_c}&=4(n-1)g_c,\\
Q_{g_c}&=2n(n-2)(n+2).
\end{align*}

Recall that the metric $g=\exp(h)$ where $h_{ij}(x) = H_{ij}(x) + O(|x|^{d+1})$. It is nature to study the interpolation between the approximated metric $\exp(\chi_{\delta}H)$ and the Euclidean metric $g_e$. However, the computations with respect to $g_e$ are lengthy. So, we choose to define a interpolation between a modified metric and the sphere metric $g_c$. It is the motivation of the following definition.
\begin{definition}\label{def:barg metric}
Recall the splitting $\chi_{\delta}H=T+S$. Consider a family of metrics on $\mathbb{R}^n$:
$$
\bar{g}(t):=(u_{\epsilon}+tw[S])^{\frac{4}{n-4}}\exp(t \chi_{\delta} H),
$$
and the localized total $Q$-curvature
$$
I_{\epsilon,\delta}(t):=\int_{B_{\delta}}Q_{\bar{g}(t)}\vol_{\bar{g}(t)}.
$$
\end{definition}

Directly, we have that
\begin{align*}
\bar{g}(0)=&g_c,\\
\bar{g}'(0)=&u_{\epsilon}^{\frac{4}{n-4}}T+u_{\epsilon}^{\frac{4}{n-4}}S+\frac{4}{n-4}u_{\epsilon}^{\frac{8-n}{n-4}}w[S] g_{e}.
\end{align*}

\begin{lemma}\label{lem:barT barS}
Define
\begin{align*}
\bar{T}:=&u_{\epsilon}^{\frac{4}{n-4}}T,\\
\bar{S}:=&u_{\epsilon}^{\frac{4}{n-4}}S+\frac{4}{n-4}u_{\epsilon}^{\frac{8-n}{n-4}}w[S] g_{e}.
\end{align*}
Then, $\bar{g}'(0)=\bar{T}+\bar{S}$. Also,
$$
\bar{T}\in S^{T,T}_{2,g_c} \text{ and } \bar{S}=\cL_Xg_c.
$$
Consequently,
$$
DP_{g_e}[T]u_{\epsilon}=DQ_{g_c}[\bar{T}]=0.
$$
\end{lemma}
\begin{proof}
In this proof, we abbreviate $u=u_{\epsilon}$.

Firstly, we will prove $\bar{S}=\cL_Xg_c$. Combining with Lemma \ref{lem:brendle splitting} and \eqref{eq:w[s] proof}, we have that
\begin{align*}
\bar{S}_{ij}=&u^{\frac{4}{n-4}} \left(X_{i,j}+X_{j,i}-\frac{2}{n}\ddiv X (g_e)_{ij}\right)\\
&+\frac{4}{n-4}u^{\frac{8-n}{n-4}}\left(X_l\partial_l u+\frac{n-4}{2n}\ddiv X u\right)(g_e)_{ij}\\
=&u^{\frac{4}{n-4}} \left(X_{i,j}+X_{j,i}\right)+X_l\partial_l(u^{\frac{4}{n-4}}(g_e)_{ij})\\
=&\bar{\nabla}_iX_j+\bar{\nabla}_jX_i=(\cL_Xg_c)_{ij}.
\end{align*}

Secondly, we will prove $\bar{T}\in S^{T,T}_{2,g_c}$. It is straightforward that $\tr_{g_c} \bar{T}=0$. Next, we just need to compute the divergence $\ddiv_{g_c} \bar{T}$. By direct computations, we have
\begin{equation}\label{eq: Qijk}
\begin{aligned}
\bar{\nabla}_{k}\bar{T}_{ij}=u^{\frac{8-n}{n-4}}\bigg(u T_{ij,k}\bigg.&-\frac{2}{n-4}u_{,j}T_{ik}+\frac{2}{n-4}u_{,l}T_{il}(g_{e})_{jk}\\
&-\frac{2}{n-4}u_{,i}T_{jk}+\bigg.\frac{2}{n-4}u_{,l}T_{jl}(g_{e})_{ik}\bigg).
\end{aligned}
\end{equation}
So, $\ddiv_{g_c} \bar{T}=u^{-\frac{2n}{n-4}}(u^{\frac{2n}{n-4}}T_{ij})_{,j}$, which vanishes by Lemma \ref{lem:brendle splitting}.

Thirdly, we will show that $DP_{g_e}[T]u=0$. Consider a family of metrics $g(t)$ with $g(0)=g_e$ and $g'(0)=T$. Doing conformal change, we have that
$$
Q_{u^{\frac{4}{n-4}}g(t)}=\frac{2}{n-4}u^{-\frac{n+4}{n-4}}P_{g(t)}u.
$$
Applying $\partial_t|_{t=0}$ to it, we have that
$$
DQ_{g_c}[\bar{T}]=\frac{2}{n-4}u^{-\frac{n+4}{n-4}}DP_{g_e}[T]u.
$$
Since $\bar{T}\in S^{T,T}_{2,g_c}$ and Lemma \ref{lem:einstein Q}, $DQ_{g_c}[\bar{T}]=0$. Combining with the equation above, we can conclude this lemma.
\end{proof}

To study the localized total $Q$-curvature, we can also do conformal change.
\begin{align*}
I_{\epsilon,\delta}(t)&=\frac{2}{n-4}\int_{B_{\delta}}(u_{\epsilon}+tw[S])P_{\exp (tH)}(u_{\epsilon}+tw[S]) dx.
\end{align*}
Apparently, we can compute $I_{\epsilon,\delta}(0)$ in both Euclidean coordinate and sphere coordinate.
\begin{align*}
I_{\epsilon,\delta}(0)&=\frac{2}{n-4}\int_{B_{\delta}}u_{\epsilon}\Delta^2u_{\epsilon}=\frac{2\mfc(n)}{n-4}\int_{B_{\delta}}u_{\epsilon}^{\frac{2n}{n-4}}\\
&=Q_{g_c}\int_{B_{\delta}}\vol_{g_c}=\frac{2\mfc(n)}{n-4}\int_{B_{\delta}}\vol_{g_c}.
\end{align*}

\begin{proposition}\label{prop: properties of barg}
We have
\begin{equation}\label{eq:barg''}
\bar{g}''(0)=u_{\epsilon}^{\frac{4}{n-4}}\left(\chi_{\delta}H\times H+\frac{8}{n-4}\frac{w[S]}{u_{\epsilon}}\chi_{\delta}H-\frac{4(n-8)}{(n-4)^2}\frac{w^2[S]}{u_{\epsilon}^2}g_e\right),
\end{equation}
\begin{equation}\label{eq:trace barS}
\ddiv_{g_c} X=\frac{1}{2}\tr_{g_c}\bar{S}=\frac{2n}{n-4}\frac{w[S]}{u_{\epsilon}}.
\end{equation}
\end{proposition}
\begin{proof}
The proofs are straightforward.
\end{proof}

Also, we can compute $I_{\epsilon,\delta}'(0)$ in both Euclidean coordinate and sphere coordinate. For the Euclidean one, we need to use equations \eqref{eq:bubble pde} and \eqref{eq:corrector pde}. For the sphere one, we need to use Lemma \ref{lem:einstein Q}, Propositions \ref{prop:D2Q S direction} and \ref{prop: properties of barg}.
\begin{align*}
&I_{\epsilon,\delta}'(0)\\
=&\frac{2}{n-4}\int_{B_{\delta}}u_{\epsilon} DP_{g_e}[S]u_{\epsilon} + u_{\epsilon}\Delta^2 w[S] +w[S]\Delta^2 u_{\epsilon}=\frac{4n\mfc(n)}{(n-4)^2}\int_{B_{\delta}}w[S] u_{\epsilon}^{\frac{n+4}{n-4}}\\
=&\int_{B_\delta}DQ_{g_c}[\bar{T}+\bar{S}]+\frac{1}{2}Q_{g_c}\tr_{g_c}(\bar{T}+\bar{S})\vol_{g_c}=\frac{1}{2}Q_{g_c}\int_{B_\delta}\tr_{g_c}(\bar{S})\vol_{g_c}.
\end{align*}

\subsection{The second variation of total $Q$-curvature}

Due to the complexity of $D^2P_{g_e}[H,H]$, it becomes much harder to compute $I_{\epsilon,\delta}''(0)$ in Euclidean coordinate. In this subsection, we will compute the second derivative $I''_{\epsilon,\delta}(0)$ using sphere coordinate.

There are two direct integration by parts formulas.
\begin{lemma}
Let $\hat{\bar{n}}$ be the unit normal vector with respect to $g_c$.
\begin{align}
\int_{B_\delta}\tr_{g_c}(\bar{S})\vol_{g_c}=&2\int_{B_\delta} \ddiv_{g_c}X\vol_{g_c}=2\int_{\partial B_{\delta}}\<X,\hat{\bar{n}}\>_{g_c}dS_{g_c},\label{eq:trace S IBP}\\
\int_{B_\delta} \<\bar{T},\bar{S}\>_{g_c}\vol_{g_c}=&2\int_{B_\delta} \<\bar{T},\bar{\nabla}X\>_{g_c}\vol_{g_c}=2\int_{\partial B_{\delta}}\bar{T}(X,\hat{\bar{n}})dS_{g_c}. \label{eq:TS IBP}
\end{align}
\end{lemma}

\begin{lemma}\label{lem:first step in I''}
We have
\begin{align*}
I''_{\epsilon,\delta}(0)= & \int_{B_{\delta}} D^2Q_{g_c} [\bar{T}, \bar{T}]+D^2Q_{g_c} [\bar{S}, \bar{S}] \vol_{g_c} \\
&-4(n^2-4)\int_{B_{\delta}}  |\bar{T}+\bar{S}|^2_{g_c}  -\frac{4n}{n-4} \frac{w^2[S]}{u_{\epsilon}^2}\vol_{g_c}\\
&+\frac{4n\tmfc(n)}{(n-4)^2}\int_{B_{\delta}} \frac{w^2[S]}{u_{\epsilon}^2}\vol_{g_c}+\textup{BT}
\end{align*}
where the \textup{BT} (boundary terms) are
\begin{align*}
\textup{BT}=&-\frac{1}{2(n-1)}\int_{\partial B_{\delta}}\partial_{\hat{\bar{n}}} DR_{g_c}[\bar{g}''(0)]dS_{g_c}\\
&-\frac{n^2-4}{(n-1)^2}\int_{\partial B_{\delta}}\partial_{\hat{\bar{n}}}\tr_{g_c}\bar{g}''(0)-\ddiv_{g_c}\bar{g}''(0)(\hat{n})dS_{g_c}.
\end{align*}
\end{lemma}
\begin{proof}
Direct computation yields
\begin{align*}
I''_{\epsilon,\delta}(0)=&\int_{B_{\delta}} D^2Q_{g_c} [g'(0), g'(0)] +DQ_{g_c}[g''(0)] + DQ_{g_c} [g'(0)] \tr_{g_c} g'(0)  \vol_{g_c} \\
+&\frac{1}{4}Q_{g_c}\int_{B_{\delta}}  (\tr_{g_c} g'(0))^2 - 2|g'(0)|^2_{g_c}  + 2 \tr_{g_c} g''(0)\vol_{g_c}.
\end{align*}
By Lemma \ref{lem:einstein Q} and Proposition \ref{prop:D2Q S direction}, $DQ_{g_c} [\bar{T}]=DQ_{g_c} [\bar{S}]=0$. Combining with $g'(0)=\bar{T}+ \bar{S}$, we have
\begin{align*}
I''_{\epsilon,\delta}(0)= & \int_{B_{\delta}} D^2Q_{g_c} [\bar{T}+\bar{S}, \bar{T}+\bar{S}] +DQ_{g_c}[g''(0)]  \vol_{g_c} \\
&+\frac{1}{4}Q_{g_c}\int_{B_{\delta}}  (\tr_{g_c} \bar{S})^2 - 2|\bar{T}+\bar{S}|^2_{g_c}  + 2 \tr_{g_c} g''(0)\vol_{g_c}.
\end{align*}

Because of Lemma \ref{lem:einstein R} and equation \eqref{eq:DQ einstein}, we can do integration by parts, 
\begin{align*}
\int_{B_\delta} DQ_{g_c}[g''(0)]\vol_{g_c}&=\int_{B_\delta}\alpha_n(\Delta+c(n)\lambda)DR_{g_c}[g''(0)]\vol_{g_c}\\
&=-4(n^2-4)\int_{B_\delta}\tr_{g_c}g''(0)\vol_{g_c}+\textup{BT},
\end{align*}
where the boundary terms are the ones in the statement of this lemma.

Also, taking trace $\tr_{g_c}$ of equation \eqref{eq:barg''}, we have that
$$
\tr_{g_c}g''(0)=  |\bar{T}+\bar{S}|^2_{g_c}  -\frac{4n}{n-4} \frac{w^2[S]}{u_{\epsilon}^2}.
$$
Combining the three equations above and \eqref{eq:trace barS}, we obtain the conclusion.
\end{proof}

\begin{proposition}\label{prop:secondvonsphere}
We have
\begin{align*}
I''_{\epsilon,\delta}\leq & -\int_{B_{\delta}}\frac{n^2-4}{2(n-1)}(|\bar{\nabla} \bar{T}|_{g_c}^2+8|\bar{T}|_{g_c}^2)+\frac{1}{(n-2)^2}|\Delta_{g_c} \bar{T}- 8 {\bar{T}}|_{g_c}^2 \vol_{g_c} \\
&+\frac{4\tmfc(n)}{n-4}\int_{B_{\delta}} \frac{w^2[S]}{u_{\epsilon}^2}\vol_{g_c}+C \sum_{k=2}^d|H^{(k)}|^2\delta^{-n+2k+4}\epsilon^{n-4}.
\end{align*}
\end{proposition}
\begin{proof}
Combining Propositions \ref{prop:D2Q TT direction}, \ref{prop:D2Q S direction}, equation \eqref{eq:TS IBP} and Lemma \ref{lem:first step in I''} together, we have that
\begin{align*}
I''(0)=& -\int_{B_{\delta}}\frac{n^2-4}{2(n-1)}(|\bar{\nabla} \bar{T}|_{g_c}^2+8|\bar{T}|_{g_c}^2)+\frac{1}{(n-2)^2}|\Delta_{g_c} \bar{T}- 8 {\bar{T}}|_{g_c}^2 \vol_{g_c}\\
&+\frac{4\tmfc(n)}{n-4}\int_{B_{\delta}} \frac{w^2[S]}{u_{\epsilon}^2}\vol_{g_c}+\textup{BT},
\end{align*}
where the boundary terms are 
\begin{align*}
\textup{BT}=&-\frac{1}{2(n-1)}\int_{\partial B_{\delta}}\partial_{\hat{\bar{n}}}(D^2R_{g_c}[\bar{T},\bar{T}]-DR_{g_c}[\mathcal{L}^2_Xg_c]+DR_{g_c}[\bar{g}''(0)])dS_{g_c}\\
&+\frac{n^2-4}{n-1}\int_{\partial B_{\delta}}\partial_{\hat{\bar{n}}}(\tr_{g_c}(\bar{T}\times_{g_c}\bar{T}+\mathcal{L}^2_Xg_c+\bar{g}''(0)))dS_{g_c}\\
&-\frac{n^2-4}{n-1}\int_{\partial B_{\delta}}\ddiv_{g_c}(\bar{T}\times_{g_c}\bar{T}+\mathcal{L}^2_Xg_c+\bar{g}''(0))(\hat{\bar{n}})dS_{g_c}\\
&+8(n^2-4)\int_{\partial B_{\delta}}\ddiv_{g_c}X\<X,\hat{\bar{n}}\>_{g_c}-2\bar{T}(X,\hat{\bar{n}})dS_{g_c}\\
:=&\RN{1}+\RN{2}+\RN{3}+\RN{4}.
\end{align*}

Next, we will estimate the boundary terms one by one. Recall the conformal changes $\bar{T}=u_{\epsilon}^{\frac{4}{n-4}}T$, $dS_{g_c}=u_{\epsilon}^{\frac{2n-2}{n-4}}dS$ and $\hat{\bar{n}}=u_{\epsilon}^{-\frac{2}{n-4}}\hat{n}$. By direct computations, we have
$$
\left|\RN{4}\right|\leq C\epsilon^{n}\sum_{k=2}^d|H^{(k)}|^2\delta^{-n+2k}.
$$
According to Lemmas \ref{lem:brendle splitting}, \ref{lem:Lie derivatives of g}, and equation \eqref{eq:barg''}, we can obtain the following estimate,
\begin{align*}
|(\bar{T}\times_{g_c}\bar{T})_{ij}|+|(\mathcal{L}^2_Xg_c)_{ij}|+|\bar{g}''(0)_{ij}|\leq Cu_{\epsilon}^{\frac{4}{n-4}}\sum_{k=2}^d|H^{(k)}|^2(\epsilon+|x|)^{2k}.
\end{align*}
Consequently, there is
$$
|\RN{2}|+|\RN{3}|\leq C\epsilon^{n-2}\sum_{k=2}^d|H^{(k)}|^2\delta^{-n+2k+2}.
$$
We need to use Lemma \ref{lem:einstein R} to deal with $\RN{1}$ term. The new terms can be dealt with by the following new estimate
$$
|\Delta_{g_c}(|\bar{T}|_{g_c}^2)|+|\bar{\nabla} \bar{T}|_{g_c}^2+|\bar{\nabla}_{i}\bar{T}_{jk}\bar{\nabla}^j\bar{T
}^{ik}|\leq Cu_{\epsilon}^{-\frac{4}{n-4}}\sum_{k=2}^d|H^{(k)}|^2(\epsilon+|x|)^{2k-2}.
$$
So, we can obtain
$$
\left|\RN{1}\right|\leq C\epsilon^{n-4}\sum_{k=2}^d|H^{(k)}|^2\delta^{-n+2k+4}.
$$
Combining the boundary estimates with interior computations, we can obtain the conclusion.
\end{proof}

Combining with Proposition \ref{prop:DW koiso identity}, we can obtain a pivotal corollary.
\begin{corollary}\label{coro:second v on sphere}
We have
\begin{align*}
I''_{\epsilon,\delta}(0)\leq & -\frac{1}{(n-2)(n-3)}\int_{B_{\delta}}|DW_{g_c}[\bar{T}]|^2_{g_c}\vol_{g_c}\\
&-\frac{(n-4)(n^2+2n-4)}{2(n-1)(n-2)}\int_{B_{\delta}}|\bar{\nabla} \bar{T}|_{g_c}^2+8|\bar{T}|_{g_c}^2\vol_{g_c} \\
&+\frac{4\tmfc(n)}{n-4}\int_{B_{\delta}} \frac{w^2[S]}{u_{\epsilon}^2}\vol_{g_c}+C \sum_{k=2}^d|H^{(k)}|^2\delta^{-n+2k+4}\epsilon^{n-4}.
\end{align*}
\end{corollary}
\noindent\textbf{Remark.} This corollary establishes the relation between the second variations of the following two total quadratic curvatures at the standard sphere:
$$
\int_{B_{\delta}}Q_g\vol_g \text{ and }\int_{B_{\delta}}|W_g|^2\vol_g.
$$

\begin{lemma}[{\cite[Proposition 8, 14]{MR2357502}}]
We have
\begin{align}
\int_{B_{5r/3}\setminus B_{4r/3}}|DW_{g_e}[H]|^2dx\geq & \frac{1}{C}\sum_{k=2}^d|H^{(k)}|^2r^{2k+n-4}, \text{ for all }r>0,\label{eq:DW eulidean lower bound}
\end{align}
and
\begin{equation}\label{eq:volume normalization}
\begin{aligned}
&\int_{B_{\delta}}\left(u_{\epsilon}^2+\frac{n+4}{n-4}w^2[S]\right)^{\frac{n}{n-4}}dx\\
\leq &\int_{B_{\delta}}\left(u_{\epsilon}+w[S]\right)^{\frac{2n}{n-4}}dx+C\epsilon^{n-2} \sum_{k=2}^d|H^{(k)}|^3\int_{B_{\delta}}(\epsilon+|x|)^{-2n+3k+4}dx\\
&+C\sum_{k=2}^d|H^{(k)}|\delta^{-n+k}\epsilon^n.
\end{aligned}
\end{equation}
\end{lemma}

\begin{proposition}\label{prop:DW sphere lower bound}
We have
\begin{align}
\int_{B_{2r}\setminus B_{r}}|DW_{g_c}[\bar{T}]|^2_{g_c}\vol_{g_c}\geq & \frac{1}{C}\epsilon^{n-4}\sum_{k=2}^d|H^{(k)}|^2r^{-n+2k+4}, \text{ for all }r>\epsilon.\label{eq:DW sphere lower bound}
\end{align}
Consequently, for all $\delta\geq 2\epsilon$, we obtain
$$
\int_{B_{\delta}}|DW_{g_c}[\bar{T}]|^2_{g_c}\vol_{g_c}\geq \frac{1}{C}\epsilon^{n-4} \sum_{k=2}^d|H^{(k)}|^2\int_{B_{\delta}}(\epsilon+|x|)^{-2n+2k+4}dx. 
$$
\end{proposition}
\noindent\textbf{Remark.} Using the arguments from \cite[Proposition 9]{MR2357502} and equation \eqref{eq: Qijk}, we can obtain that
$$
\int_{B_{\delta}}|\bar{\nabla} \bar{T}|_{g_c}^2\vol_{g_c}\geq \frac{1}{C}\epsilon^{n-2} \sum_{k=2}^d|H^{(k)}|^2\int_{B_{\delta}}(\epsilon+|x|)^{-2n+2k+2}dx. 
$$
It is not sharp enough to start the dichotomy argument (whether $p\in \mcz$ or not). This is the reason why we use the higher-order Koiso-Bochner formula.
\begin{proof}
Firstly, we will find the relation between $DW_{g_c}[\bar{T}]$ and $DW_{g_e}[H]$. According to the decomposition $H=T+S$, we have that
$$
DW_{g_e}[H]=DW_{g_e}[T]+DW_{g_e}[S].
$$
In the $T$ direction, consider a family of metric $g(t)$ with $g(0)=g_e$ and $g'(0)=T$. According to $W_{u_\epsilon^{\frac{4}{n-4}}g(t)}=u_\epsilon^{\frac{4}{n-4}}W_{g(t)}$, we have that
$$
DW_{g_c}[\bar{T}]=u_\epsilon^{\frac{4}{n-4}}DW_{g_e}[T].
$$
In the $S$ direction, consider $g(t)$ with $g(0)=g_e$ and $g'(0)=S$. According to $W_{(u_\epsilon+tw[S])^{\frac{4}{n-4}}g(t)}=(u_\epsilon+tw[S])^{\frac{4}{n-4}}W_{g(t)}$ and $W_{g_e}=0$, we have that
$$
DW_{g_c}[\bar{S}]=u_\epsilon^{\frac{4}{n-4}}DW_{g_e}[S].
$$
On the other hand, according to Proposition \ref{prop:D2Q S direction} and $W_{g_c}=0$, we obtain
$$
DW_{g_c}[\bar{S}]=DW_{g_c}[\cL_Xg_c]=0.
$$
Combining the four equations above, we can obtain that
\begin{equation}\label{eq:DW T and DW H}
DW_{g_c}[\bar{T}]=u_{\epsilon}^{\frac{4}{n-4}}DW_{g_e}[H].  
\end{equation}

Then, we can combine \eqref{eq:DW eulidean lower bound} and \eqref{eq:DW T and DW H} to obtain equation \eqref{eq:DW sphere lower bound}.
\end{proof}

\begin{proposition}\label{prop:I+I'+1/2I''}
We have
\begin{align*}
&\frac{n-4}{2}\left(I_{\epsilon,\delta}(0)+I_{\epsilon,\delta}'(0)+\frac{1}{2}I_{\epsilon,\delta}''(0)\right)\\
\leq & Y_4(S^n)\left(\int_{B_\delta}\left(u_{\epsilon}+w[S]\right)^{\frac{2n}{n-4}}dx\right)^{\frac{n-4}{n}}\\
-&\frac{1}{C}\epsilon^{n-4} \sum_{k=2}^d|H^{(k)}|^2\int_{B_{\delta}}(\epsilon+|x|)^{-2n+2k+4}dx+C \sum_{k=2}^d|H^{(k)}|\delta^{-n+k+4}\epsilon^{n-4}.
\end{align*}
\end{proposition}
\begin{proof}
We can use \eqref{eq:trace S IBP} to write $I_{\epsilon,\delta}'(0)$ as a boundary term. So,
$$
I_{\epsilon,\delta}'(0)\leq C \sum_{k=2}^d|H^{(k)}|\delta^{-n+k+4}\epsilon^{n-4}.
$$
Combining with Corollary \ref{coro:second v on sphere} and Proposition \ref{prop:DW sphere lower bound}, we have
\begin{align*}
&\frac{n-4}{2}\left(I_{\epsilon,\delta}(0)+I_{\epsilon,\delta}'(0)+\frac{1}{2}I_{\epsilon,\delta}''(0)\right)\\
\leq & \mfc(n)\int_{B_{\delta}}\left(u_{\epsilon}^2+\frac{n+4}{n-4}w^2[S]\right)u_{\epsilon}^{\frac{8}{n-4}}dx\\
-&\frac{1}{C}\epsilon^{n-4} \sum_{k=2}^d|H^{(k)}|^2\int_{B_{\delta}}(\epsilon+|x|)^{-2n+2k+4}dx+C \sum_{k=2}^d|H^{(k)}|\delta^{-n+k+4}\epsilon^{n-4}.
\end{align*}

By H\"older's inequality, there is
\begin{align*}
&\int_{B_{\delta}}\left(u_{\epsilon}^2+\frac{n+4}{n-4}w^2[S]\right)u_{\epsilon}^{\frac{8}{n-4}}dx\\
\leq &\left(\int_{B_{\delta}}\left(u_{\epsilon}^2+\frac{n+4}{n-4}w^2[S]\right)^{\frac{n}{n-4}}dx\right)^{\frac{n-4}{n}}\left(\int_{B_{\delta}}u_{\epsilon}^{\frac{2n}{n-4}}dx\right)^{\frac{4}{n}}.
\end{align*}
Using $Y_4(S^n)=\mfc(n)\left(\int_{\mathbb{R}^n}u_{\epsilon}^{\frac{2n}{n-4}}dx\right)^{\frac{4}{n}}$ and \eqref{eq:volume normalization}, we can obtain the conclusion.
\end{proof}

\subsection{Auxiliary estimates}

We prove basic estimates on curvature approximations and elliptic estimates $X$ in this subsection. We fix the background metric $g=\exp(h)$ to be in the conformal normal coordinate at a point $p \in M$. In this subsection, we abbreviate $w = w_{\epsilon, \delta}[S]$.

\begin{lemma}\label{lem:auxiliary lem1}
For $x\in B_{\delta}$, we have
\begin{align*}
\bigg | w \big(P_{\exp{h}} - \Delta^2 - DP_{g_e}[H] \big) u_\epsilon \bigg|+\bigg | u_\epsilon \big(P_{\exp{h}} - \Delta^2 - DP_{g_e}[H] \big) w \bigg| \\
\le \frac{1}{C} \epsilon^{n-4} \sum_{k=2}^d    |H^{(k)}|^2 (\epsilon + |x|)^{-2n+2k+4} + C \epsilon^{n-4} \delta^{-2n +2d +6}.
\end{align*}
\end{lemma}

\begin{proof}
According to Lemma \ref{lem:expansion of P}, $h=O(|x|^2)$ and $h=H+O(|x|^{d+1})$, 
\begin{align*}
&(P_{\exp(h)}-\Delta^2-DP_{g_e}[h]) u_{\epsilon}\\
\leq& C\Bigg(\sum_{\substack{0 \le \alpha_1,\alpha_2,\beta \le 4 \\ \alpha_1+\alpha_2+\beta=4}} |\partial^{\alpha_1}h||\partial^{\alpha_2}h| |\partial^{\beta}u_{\epsilon}|\Bigg)\\
\leq& C\Bigg(\sum_{\substack{0 \le \alpha_1,\alpha_2,\beta \le 4 \\ \alpha_1+\alpha_2+\beta=4}} |x|^{2-\alpha_1}(|\partial^{\alpha_2}H|+|x|^{d+1-\alpha_2}) |\partial^{\beta}u_{\epsilon}|\Bigg).
\end{align*}
Hence, we have
\begin{equation*}
\bigg | \big(P_{\exp{h}} - \Delta^2 - DP_{g_e}[h] \big) u_\epsilon \bigg| \le C \epsilon^{\frac{n-4}{2}} \sum_{k=2}^d    |H^{(k)}|(\epsilon + |x|)^{-n+k+2} + C\epsilon^{\frac{n-4}{2}}  (\epsilon + |x|)^{-n+3+d}.
\end{equation*}
Next, noting that $d \ge 2$, we have 
\begin{equation*}
DP_{g_e}[h-H] = \sum_{k=0}^4 O(|\partial^{4-k}(h-H)|) \partial^k = \sum_{k=0}^4 O(|x|^{\max\{0, d+k-3\}}  )\partial^k.
\end{equation*}
So that 
\begin{equation*}
\big | DP_{g_e}[h-H] u_\epsilon \big|  \le C\epsilon^{\frac{n-4}{2}}  (\epsilon + |x|)^{-n+d+1}.
\end{equation*}
Hence, we have
\begin{equation*}
\bigg | \big(P_{\exp{h}} - \Delta^2 - DP_{g_e}[H] \big) u_\epsilon \bigg| \le C \epsilon^{\frac{n-4}{2}} \sum_{k=2}^d    |H^{(k)}|(\epsilon + |x|)^{-n+k+2} + C\epsilon^{\frac{n-4}{2}}  (\epsilon + |x|)^{-n+1+d}.
\end{equation*}
Finally, combining with Lemma \ref{lem:brendle splitting} and Young's inequality, we estimate
\begin{align*}
&\bigg |w \big(P_{\exp{h}} - \Delta^2- DP_{g_e}[H] \big) u_\epsilon \bigg| \\
\le& C \epsilon^{n-4} \sum_{k=2}^d  \bigg(   |H^{(k)}|^2 (\epsilon + |x|)^{-2n+2k+6} + |H^{(k)}|  (\epsilon + |x|)^{-2n+k+d+5} \bigg) \\
\le& \frac{1}{C} \epsilon^{n-4} \sum_{k=2}^d    |H^{(k)}|^2 (\epsilon + |x|)^{-2n+2k+4} +  C \epsilon^{n-4} \delta^{-2n +2d +6}.
\end{align*}

A similar logic applies to $\big|u_\epsilon \big(P_{\exp{h}} - \Delta^2 - DP_{g_e}[H] \big) w_\epsilon \big|$.
\end{proof}

\begin{lemma}\label{lem:auxiliary lem2}
For $x\in B_{\delta}$, we have
\begin{align*}
&\bigg| u_\epsilon \bigg( P_{\exp(h)} - \Delta^2 - DP_{g_e}[h] - \frac{1}{2} D^2P_{g_e} [H, H]  
-  \frac{1}{2} DP_{g_e} [H\times H]\bigg) u_\epsilon \bigg| \nonumber \\
\le & \frac{1}{C} \epsilon^{n-4} \sum_{k=2}^d    |H^{(k)}|^2 (\epsilon + |x|)^{-2n+2k+4} +  C \epsilon^{n-4} \delta^{-2n +2d +6}.
\end{align*}
\end{lemma}

\begin{proof}
According to Lemma \ref{lem:expansion of P}, $h=O(|x|^2)$ and $h=H+O(|x|^{d+1})$, 
\begin{align*}
&\bigg( P_{\exp(h)} - \Delta^2 - DP_{g_e}[h] - \frac{1}{2} D^2P_{g_e} [h, h]  
-  \frac{1}{2} DP_{g_e} [h\times h]\bigg)u_{\epsilon}\\
\leq &C\Bigg(\sum_{\substack{0 \le \alpha_1,\alpha_2,\alpha_3,\beta \le 4 \\ \alpha_1+\alpha_2+\alpha_3+\beta=4}} |x|^{2-\alpha_1}(|\partial^{\alpha_2}H|+|x|^{d+1-\alpha_2}) (|\partial^{\alpha_3}H|+|x|^{d+1-\alpha_3})|\partial^{\beta}u_{\epsilon}|\Bigg).
\end{align*}
Hence, we have 
\begin{align*}
&  \bigg| u_\epsilon \bigg( P_{\exp(h)} - \Delta^2 - DP_{g_e}[h] - \frac{1}{2} D^2P_{g_e} [h, h]  
-  \frac{1}{2} DP_{g_e} [h\times h]\bigg) u_\epsilon \bigg| \\
\le&  C \epsilon^{n-4} \sum_{k=2}^d    |H^{(k)}|^2 (\epsilon + |x|)^{-2n+2k+6} +  C \epsilon^{n-4} \delta^{-2n +2d +8}.
\end{align*}

Next, using the algebraic estimate
\begin{align*}
&D^2 P_{g_e} [h, h] -  D^2 P_{g_e} [H, H] + DP_{g_e}[h\times h-H\times H ]\\
= &  \sum_{k=0}^4 O(|x|^{d+k-3} |H| +|x|^{2d+k-2}) \partial^k  
\end{align*}
and Young's inequality, we have 
\begin{align*}
& \bigg | u_\epsilon \bigg( D^2 P_{g_e} [h, h] -  D^2 P_{g_e} [H, H] + DP_{g_e}[h\times h-H\times H ] \bigg)  u_\epsilon \bigg|  \\
\le& C\epsilon^{n-4}  \sum_{k=2}^d   |H^{(k)}| (\epsilon + |x|)^{-2n+d+ k+5} +  C \epsilon^{n-4} \delta^{-2n +2d +6}\\
\leq &\frac{1}{C} \epsilon^{n-4} \sum_{k=2}^d    |H^{(k)}|^2 (\epsilon + |x|)^{-2n+2k+4} +  C \epsilon^{n-4} \delta^{-2n +2d +6}.
\end{align*}
This proves the assertion.
\end{proof}

\begin{lemma}
We have
\begin{equation}\label{eq:DP sphere vanish}
\int_{B_\delta} u_\epsilon DP_{g_e}[h] u_\epsilon, \hspace{2mm}  \int_{B_\delta} u_\epsilon DP_{g_e}[H] u_\epsilon \le C \epsilon^{n-4} \delta^{2d+6- n}.
\end{equation}
\end{lemma}

\begin{proof}
By $ x_ih_{ij} = 0$, we observe that $-\Delta(\partial_i(h_{ij}\partial_j u_\epsilon))-\partial_i (h_{ij}\partial_j \Delta u_\epsilon) = 0 $. Combining with Lemma \ref{lem:expansion of P} and $\tr h = O(|x|^{2d+2})$, we have
\begin{align*}
&u_\epsilon DP_{g_e}[h] u_\epsilon\\
= & u_\epsilon \bigg[ \frac{4}{n-2}D(\ric_{g_e})_{ij}[h] \partial^2_{ij} u_\epsilon - \alpha_n^1DR_{g_e}[h] \Delta u_\epsilon- \alpha_n^2\partial_kDR_{g_e}[h]\partial_k u_\epsilon \\
&-\frac{n-4}{4(n-1)}\Delta (DR_{g_e}[h]) u_\epsilon\bigg] + C \epsilon^{n-4} O( (\epsilon + |x|)^{-2n +2d +6 }),
\end{align*}
where the linearizations of the curvatures are
\begin{align*}
u_\epsilon D(\ric_{g_e})_{ij}\partial^2_{ij} u_\epsilon = & u_\epsilon \big[ h_{jk, ik} - \frac{1}{2} h_{ij, kk} + O(|x|^{2d})\big]  \partial^2_{ij} u_\epsilon,\\
DR_{g_e}[h]= &h_{ij, ij} +  O(|x|^{2d}).
\end{align*}

Moreover, we need the following identities, which can be derived by differentiating the identity $x_i h_{ij} = 0$:
\begin{align*}
h_{ij, k} x_i  &= -h_{jk},\\
h_{ij, kk}x_ix_j =  2 h_{kk}, \hspace {3mm} &h_{ik,jk}x_ix_j = h_{kk}-x_i h_{kk,i},\\
h_{ij,ikk}x_j &=  - h_{ii, jj} - 2h_{ij, ij}.
\end{align*}
Also, we can compute
\begin{equation}\label{eq:53}
\int_{B_r} h_{ij,ij} =  \frac{1}{r}\int_{\partial B_r} h_{ij,i}x_j = - \frac{1}{r} \int_{\partial B_r} h_{ii} = O(r^{n+2d}),
\end{equation}
and
\begin{equation}\label{eq:54}
\int_{\partial B_r} h_{ij,ij}   = O(r^{n+2d-1}).
\end{equation}

Then, the assertion follows by applying integration by parts. For example, by (\ref{eq:53}) and (\ref{eq:54}),
\begin{align*}
&\int_{B_\delta}  u_\epsilon \Delta(DR_{g_e}[h]) u_\epsilon\\ 
= & \int_{B_\delta} h_{ij,ijkk} u_\epsilon^2 + O(|x|^{2d-2})  u_\epsilon^2 \\
\le & \int_{B_\delta} - h_{ij, ikk} \partial_j u_\epsilon^2 + \int_{\partial B_\delta} h_{ij, ikk} \frac{x_j}{\delta} u_\epsilon^2 + C \epsilon^{n-4} \delta^{2d+6 - n } \\
\le & 2 (-n+4)  \epsilon^{n-4} \int_{B_\delta} h_{ij, ij} (\epsilon^2 +|x|^2)^{-n+3} - \frac{2}{\delta} \int_{\partial B_\delta} h_{ij, ij} u_\epsilon^2 + C \epsilon^{n-4} \delta^{2d+6 - n } \\
= & C \epsilon^{n-4} \delta^{2d+6 - n }.
\end{align*} 
The same method works for $H$ and this proves the assertion.
\end{proof}

Recall that when $x\in B_{\delta}$, we define $\bar{g}(t)=(u_\epsilon + tw )^{\frac{4}{n-4}} \exp(tH)$ and our background metric is $g=\exp(h)$. So, in the rest of this subsection, we will estimates the errors from higher order terms and the errors caused by changing $H$ to $h$. One of the tricky parts is that the cubic order term of the following interpolation is difficult to keep track on.
\begin{align*}
&(u_\epsilon + tw )^{\frac{4}{n-4}} \exp(th ) \\
=&  g_c + t\bigg[\frac{4}{n-4} \frac{w}{u_\epsilon} +  u_\epsilon^{\frac{4}{n-4}} h   \bigg] + \frac{1}{2} t^2 u_\epsilon^{\frac{4}{n-4}} \bigg[h^2 + \frac{2(-n+8)}{(n-4)^2} \frac{w^2}{u_\epsilon^2} + \frac{8}{n-4} \frac{w}{u_\epsilon} h  \bigg] \\
&+ O(\text{cubic order of $w$ and $h$}). 
\end{align*}
So, we need to use one more property of the total $Q$-curvature. For any metric $\hat{g}$ defined on the set $B_\delta$, define $$E_\delta[\hat{g}] = \int_{B_\delta} Q_{\hat{g}} \vol_{\hat{g}}.$$ 
Notice that $E_\delta [(u_\epsilon +  sw )^{\frac{4}{n-4}} \exp(t H ) ]$ is quadratic with respect to $sw$. So, we define the total $Q$-curvature of a two-parameter family of metrics:
\begin{align*}
I_{\epsilon, \delta}(s, t):=&E_\delta [(u_\epsilon +  sw )^{\frac{4}{n-4}} \exp( tH ) ] \\
=& \frac{2}{n-4}\int_{B_\delta}  (u_\epsilon +sw  ) P_{ \exp( tH )} (u_\epsilon+sw )  \vol_{\exp(tH) }.
\end{align*}

There is a basic algebraic lemma.
\begin{lemma}\label{lem:algebraicspl}
Let $F(s, t)$ be a smooth function in two variables which is quadratic w.r.t to $s$ variable, i.e.
$$F(s, t) = F(0, t) + s \partial_s F(0, t) + \frac{1}{2} s^2 \partial_s^2 F(0, t).$$
Then we have
\begin{align*}
&F(1, 1) - F(0, 0) - \frac{d}{dt} F(t, t) \bigg|_{t = 0} -\frac{1}{2}\frac{d^2}{dt^2} F(t, t) \bigg|_{t=0} \\
= & \bigg (F(0, 1) -F(0, 0) -  \partial_t F (0, 0) -\frac{1}{2} \partial_t^2 F(0, 0) \bigg) \\
& + \bigg(\partial_s F(0, 1)  - \partial_s F(0,0)- \partial_t \partial_s F(0, 0) \bigg) \\
& + \frac{1}{2} \bigg( \partial_s^2 F(0, 1) - \partial_s^2 F(0, 0) \bigg).
\end{align*}
\end{lemma}
According to this lemma, we have:
\begin{equation}\label{eq:main term and error term}
\frac{2}{n-4}\int_{B_{\delta}}(u_{\epsilon}+w_{\epsilon,\delta})P_{g}(u_{\epsilon}+w_{\epsilon,\delta})\vol_g - I_{\epsilon,\delta}(0) - I_{\epsilon, \delta}'(0) - \frac{1}{2} I''_{\epsilon, \delta}(0)=\RN{1}+\RN{2}+\RN{3},
\end{equation}
where
\begin{align*}
\RN{1}:=&\frac{1}{2} \bigg(\partial_s^2 I_{\epsilon, \delta}(0, 1) - \partial_s^2 I_{\epsilon, \delta}(0, 0) \bigg)\\
&+\frac{2}{n-4}\left(\int_{B_\delta}  w  P_{ \exp( h )} w \vol_{\exp(h) } - \int_{B_\delta}  w  P_{ \exp( H )} w   \vol_{\exp(H) }\right),\\
\RN{2}:=& \partial_s I_{\epsilon, \delta}(0, 1) - \partial_s I_{\epsilon, \delta}(0, 0) - \partial_t \partial_s I_{\epsilon, \delta}(0, 0)\\
&+\frac{2}{n-4} \bigg( \int_{B_\delta}    w P_{ \exp( h )}  u_\epsilon   \vol_{\exp(h) } + \int_{B_\delta}    u_\epsilon P_{ \exp( h )} w   \vol_{\exp(h) } \\
& \hspace{1.5cm}- \int_{B_\delta}   w P_{ \exp( H )} u_\epsilon \vol_{\exp(H) } - \int_{B_\delta}    u_\epsilon P_{ \exp( H )}w  \vol_{\exp(H) }\bigg),\\
\RN{3}:=& I_{\epsilon, \delta}(0, 1) - I_{\epsilon, \delta}(0, 0) - \partial_t I_{\epsilon, \delta}(0, 0) - \frac{1}{2}\partial_t^2 I_{\epsilon, \delta}(0, 0)\\
&+ \frac{2}{n-4}\bigg(\int_{B_{\delta}} u_\epsilon  P_{\exp(h)} u_\epsilon  \vol_{\exp(h)} - \int_{B_{\delta}} u_\epsilon  P_{\exp(H)} u_\epsilon  \vol_{\exp(H)}  \bigg).
\end{align*}

\begin{lemma}\label{lem:rhs}
We have
\begin{align*}
|\RN{1}|+|\RN{2}|+|\RN{3}|&\leq \frac{1}{C} \epsilon^{n-4} \sum_{k=2}^{d} |H^{(k)}|^2\int_{B_{\delta}}  ( \epsilon + |x|)^{-2n+2k+4}  dx + C \epsilon^{n-4} \delta^{2d+6-n}.
\end{align*}
\end{lemma}

\begin{proof}
According to Lemmas \ref{lem:expansion of P}, \ref{lem:brendle splitting}, and $h=O(|x|^2)$,
\begin{align*}
\RN{1}= &  \int_{B_\delta}  w  P_{ \exp( h )} w   \vol_{\exp(h) } -   \int_{B_\delta}  w  \Delta^2w  \vol_{g_e } \\
\le & C \epsilon^{n-4} \int_{B_\delta}  \sum_{k=2}^{d} ( \epsilon + |x|)^{-2n+2k+6} |H^{(k)}|^2 dx.
\end{align*} 
Letting $\delta$ to be small enough, we can obtain the estimate for term \RN{1}.

For term \RN{2}, we have
\begin{align*}
\RN{2}= &  \int_{B_\delta}    u_\epsilon P_{ \exp( h )} w   ( \vol_{\exp(h) } - dx ) + \int_{B_\delta}    u_\epsilon \bigg( P_{ \exp( h )} - \Delta^2- DP_{g_e} [H] \bigg)w dx  \\
+&\int_{B_\delta}    w P_{ \exp( h )} u_\epsilon  ( \vol_{\exp(h) } - dx ) + \int_{B_\delta}    w \bigg( P_{ \exp( h )} - \Delta^2- DP_{g_e} [H] \bigg) u_{\epsilon} dx.
\end{align*}
Combining with Lemma \ref{lem:auxiliary lem1} and $\tr h = O(|x|^{2d+2})$, we can obtain the conclusion for term \RN{2}.

Lastly, we have
\begin{align*}
\RN{3}= & \int_{B_{\delta}} u_\epsilon  P_{\exp(h)} u_\epsilon  (\vol_{\exp(h)} -dx)\\
+&\int_{B_{\delta}} u_\epsilon \bigg(P_{\exp(h)}- \Delta^2 - DP_{g_e}[H] -\frac{1}{2} D^2P_{g_e} [H, H]  
-\frac{1}{2} DP_{g_e} [H\times H]\bigg) u_\epsilon  dx.
\end{align*}
Combining with Lemma \ref{lem:auxiliary lem2}, \eqref{eq:DP sphere vanish}, and $\tr h = O(|x|^{2d+2})$, we can obtain the the estimate for term \RN{3}.

This proves the assertion.
\end{proof}

Finally, we can combine Lemma \ref{lem:rhs} and Proposition \ref{prop:I+I'+1/2I''} with equation \eqref{eq:main term and error term} to obtain the following localized estimate.
\begin{proposition}\label{prop:localized estimate}
We have
\begin{align*}
&\int_{B_{\delta}}(u_{\epsilon}+w_{\epsilon,\delta})P_g(u_{\epsilon}+w_{\epsilon,\delta})\vol_g\\
\leq & Y_4(S^n)\left(\int_{B_\delta}(u_{\epsilon}+w_{\epsilon,\delta})^{\frac{2n}{n-4}}\vol_g\right)^{\frac{n-4}{n}}-\frac{1}{C}\epsilon^{n-4} \sum_{k=2}^d|H^{(k)}|^2\int_{B_{\delta}}(\epsilon+|x|)^{-2n+2k+4}dx\\
&+C \sum_{k=2}^d|H^{(k)}|\delta^{-n+k+4}\epsilon^{n-4}+C\delta^{2d+6-n}\epsilon^{n-4}.
\end{align*}
\end{proposition}

\subsection{Green’s function modifications}

We give the definitions of the modified test bubbles and the boundary integrals of Green's functions.
\begin{definition}[Modified test bubbles]\label{def:test bubble Green}
\begin{equation}\label{testfunction2}
v_{p,\epsilon,\delta}:=\chi_\delta(u_{\epsilon}+w_{\epsilon,\delta})+(1-\chi_{\delta})\epsilon^{\frac{n-4}{2}}G_p.
\end{equation}
\end{definition}

\begin{definition}\label{def:mci}
\begin{align*}
\mci(p,\delta):=&-\int_{\partial B_{\delta}}|x|^{4-n}\partial_{\hat{n}}\Delta G_p(x)-\partial_{\hat{n}}(|x|^{4-n}) \Delta G_p(x)dS\\
&-\int_{\partial B_{\delta}}\Delta(|x|^{4-n})\partial_{\hat{n}} G_p(x)-\partial_{\hat{n}}\Delta(|x|^{4-n})G_p(x)dS. 
\end{align*}
Apparently, when $\delta>\tilde{\delta}>0$, there is
$$
\mci(p,\delta)-\mci(p,\tilde{\delta})=-\int_{B_{\delta}\setminus B_{\tilde{\delta}}}|x|^{4-n}\Delta^2 G_p(x)dx.
$$
\end{definition}

\begin{lemma}\label{lem:thm1.2 (b)}
There is a positive constant $C$ such that for $\delta>\tilde{\delta}>0$,
$$
\sup_{p\in \mcz} |\mci(p,\delta)-\mci(p,\tilde{\delta})| \leq C\delta^{2d+6-n}.
$$
Consequently, there exists a continuous function $\mci: \mcz \rightarrow \mathbb{R}$ such that 
$$
\sup_{p\in \mcz} |\mci(p,\delta)-\mci(p)| \rightarrow 0 \text{ as }\delta \rightarrow0.
$$ 
\end{lemma}
\begin{proof}
By Definition \ref{def:mci}, we have
\begin{align*}
&\mci(p,\delta)-\mci(p,\tilde{\delta})\\
=&\int_{B_{\delta}\setminus B_{\tilde{\delta}}}|x|^{4-n}(P_g-\Delta^2-DP_{g_e}[h])G_p(x)-|x|^{4-n}DP_{g_e}[h]G_p(x)dx.
\end{align*}
Since $p\in \mcz$, there is $G_p(x)=|x|^{4-n}+O(|x|^{d+5-n})$. Combining with Lemma \ref{lem:expansion of P}, we can obtain
$$
\left|\int_{B_{\delta}\setminus B_{\tilde{\delta}}}|x|^{4-n}(P_g-\Delta^2-DP_{g_e}[h])G_p(x)dx\right|\leq C\delta^{2d+6-n}.
$$
Using the same proof as equation \eqref{eq:DP sphere vanish}, we know that 
$$
\left|\int_{B_{\delta}\setminus B_{\tilde{\delta}}}|x|^{4-n}DP_{g_e}[h]G_p(x)dx\right|\leq C\delta^{2d+6-n}.
$$
Combining the three equations above, we can obtain the conclusion.
\end{proof}

Now, we are ready to do global analysis.
\begin{proposition}\label{prop:global estimate}
We have
\begin{align*}
\int_{M\setminus B_{\delta}}v_{p,\epsilon,\delta}P_gv_{p,\epsilon,\delta}\vol_g\leq &-\epsilon^{n-4}\mci(p,\delta)+C \sum_{k=2}^d|H^{(k)}|\delta^{-n+k+4}\epsilon^{n-4}\\
&+C\delta^{2d+6-n}\epsilon^{n-4}+C\left(\frac{\epsilon}{\delta}\right)^{n-2}.
\end{align*}
\end{proposition}
\begin{proof}
For simplicity, we write $v_p=v_{p,\epsilon,\delta}$ in this proof. Since $v_{p}=\epsilon^{\frac{n-4}{2}}G_p$ and $P_gG_p\equiv0$ in $M\setminus B_{2\delta}$, we just need to estimate the integral in the annulus region. We write it as four parts.
\begin{align*}
&\int_{B_{2\delta}\setminus B_{\delta}}v_{p}P_gv_{p}\vol_g\\
=&\int_{B_{2\delta}\setminus B_{\delta}}(v_{p}-\epsilon^{\frac{n-4}{2}}G_p)P_gv_{p}\vol_g+\int_{B_{2\delta}\setminus B_{\delta}}\epsilon^{\frac{n-4}{2}}(G_{p}\Delta^2v_{p}-v_p\Delta^2G_p)\vol_g\\
&+\int_{B_{2\delta}\setminus B_{\delta}}\epsilon^{\frac{n-4}{2}}G_{p}(P_g-\Delta^2)(v_{p}-\epsilon^{\frac{n-4}{2}}G_p)\vol_g\\
&+\int_{B_{2\delta}\setminus B_{\delta}}\epsilon^{\frac{n-4}{2}}(v_{p}-\epsilon^{\frac{n-4}{2}}G_p)(\Delta^2-P_g)G_{p}\vol_g\\
:=&\RN{1}+\RN{2}+\RN{3}+\RN{4}.
\end{align*}

In general, there are $\vol_g=(1+O(|x|^{d+1}))dx$ and 
$$
|G_p(x)-|x|^{4-n}|\leq C\sum_{k=2}^d|H^{(k)}||x|^{4-n+k}+C|x|^{d+5-n}.
$$
For $x\in M\setminus B_{\delta}$, there is
\begin{align}
&|v_p(x)-\epsilon^{\frac{n-4}{2}}G_p(x)|+\delta^4|P_gv_p(x)|\nonumber\\
\leq &C \sum_{k=2}^d|H^{(k)}|\delta^{4-n+k}\epsilon^{\frac{n-4}{2}}+C\delta^{d+5-n}\epsilon^{\frac{n-4}{2}}+C\delta^{2-n}\epsilon^{\frac{n}{2}}.\label{eq:v and Green's function estimate}
\end{align}
So, we can deal with term \RN{1} and term \RN{2}.
\begin{align*}
|\RN{1}|\leq& C \sum_{k=2}^d|H^{(k)}|^2\delta^{4-n+2k}\epsilon^{n-4}+C\delta^{2d+6-n}\epsilon^{n-4}+C\delta^{-n}\epsilon^{n},\\
|\RN{2}+\epsilon^{n-4}\mci(p,\delta)|\leq &C\sum_{k=2}^d|H^{(k)}|\delta^{4-n+k}\epsilon^{n-4}+C\delta^{2d+6-n}\epsilon^{n-4}+C\delta^{2-n}\epsilon^{n-2}.
\end{align*}

Combining Lemma \ref{lem:expansion of P} with \eqref{eq:v and Green's function estimate}, we have
$$
|\RN{3}|+|\RN{4}|\leq C \sum_{k=2}^d|H^{(k)}|^2\delta^{4-n+2k}\epsilon^{n-4}+C\delta^{2d+6-n}\epsilon^{n-4}+C\delta^{2-n}\epsilon^{n-2}.
$$

Adding the four terms together, we can obtain the conclusion.
\end{proof}

\begin{corollary}\label{cor:thm1.2 (a)}
We have
\begin{align*}
\int_{M}v_{p,\epsilon,\delta}P_gv_{p,\epsilon,\delta}\vol_g\leq & Y_4(S^n)\left(\int_{M}v_{p,\epsilon,\delta}^{\frac{2n}{n-4}}\vol_g\right)^{\frac{n-4}{n}}-\epsilon^{n-4}\mci(p,\delta)\\
&-\frac{1}{C}\epsilon^{n-4} \int_{B_{\delta}}|W_g(x)|^2(\epsilon+|x|)^{8-2n}dx\\
&+C\delta^{2d+6-n}\epsilon^{n-4}+C\left(\frac{\epsilon}{\delta}\right)^{n-4}\log^{-1}\left(\frac{\delta}{\epsilon}\right).
\end{align*}
\end{corollary}
\begin{proof}
Combining Propositions \ref{prop:localized estimate} and \ref{prop:global estimate}, we can obtain that
\begin{align*}
&\int_{M}v_{p,\epsilon,\delta}P_gv_{p,\epsilon,\delta}\vol_g\\
\leq & Y_4(S^n)\left(\int_{M}v_{p,\epsilon,\delta}^{\frac{2n}{n-4}}\vol_g\right)^{\frac{n-4}{n}}-\epsilon^{n-4}\mci(p,\delta)\\
&-\frac{1}{C}\epsilon^{n-4} \sum_{k=2}^d|H^{(k)}|^2\int_{B_{\delta}}(\epsilon+|x|)^{-2n+2k+4}dx\\
&+C \sum_{k=2}^d|H^{(k)}|\delta^{-n+k+4}\epsilon^{n-4}+C\delta^{2d+6-n}\epsilon^{n-4}+C\left(\frac{\epsilon}{\delta}\right)^{n-2}.
\end{align*}

According to an estimate $|W_g(x)|\leq \sum_{k=2}^d |H^{(k)}||x|^{k-2}+C|x|^{d-1}$, there is
\begin{align*}
&\int_{B_{\delta}}|W_g(x)|^2(\epsilon+|x|)^{8-2n}dx\\
\leq & C\sum_{k=2}^d|H^{(k)}|^2\int_{B_{\delta}}(\epsilon+|x|)^{-2n+2k+4}dx+C\delta^{2d+6-n}.
\end{align*}

We set $\theta_k=1$ if $k=\frac{n-4}{2}$ and $\theta_k=0$ otherwise. By Young's inequality, for $2\leq k\leq d$, there are
\begin{align*}
C |H^{(k)}|\delta^{-n+k+4}\epsilon^{n-4}\leq &\frac{1}{C}|H^{(k)}|^2\epsilon^{n-4}\int_{B_{\delta}}(\epsilon+|x|)^{-2n+2k+4}dx\\
&+C\left(\frac{\epsilon}{\delta}\right)^{2n-8-2k}\log^{-\theta_k}\left(\frac{\delta}{\epsilon}\right).
\end{align*}

Combining the three estimates above, we can obtain the conclusion.
\end{proof}

\subsection{Proof of Theorem \ref{thm:energy estimate}}

Before the proof of Theorem \ref{thm:energy estimate}, let us recall the definition of $Q$-mass and the corresponding positive mass theorem.

\begin{definition}\label{def:Q-mass}
Define an asymptotically flat manifold $(M^n,\hat{g})$ with decay rate $\tau>0$ to be $M=M_0\cup M_{\infty}$, where $M_0$ is compact and $M_{\infty}$ is diffeomorphic to $\mathbb{R}^n\setminus \overline{B_1}$ such that $\hat{g}_{ij}(y)-\delta_{ij}=O^{(4)}(|y|^{-\tau})$ as $|y|\to \infty$. Moreover, if $\tau>\frac{n-4}{2}$ and $Q_{\hat{g}} \in L^1(M,\hat{g})$, we can define $Q$-mass to be
\begin{equation*}
m(\hat{g}):=\lim_{r\to\infty} \int_{\partial B_r}(\hat{g}_{ii,jjk}-\hat{g}_{ij,ijk})(y)\frac{y_k}{|y|}dS.
\end{equation*}
\end{definition}

\begin{proposition}[{\cite[Theorem A]{MR4375794}}]\label{PMT2}
Assume that $(M^n,\hat{g})$ is an asymptotically flat manifold with decay rate $\tau>\frac{n-4}{2}$ satisfying
\begin{itemize}
\item[(a)] $n\ge 5$;
\item[(b)] $Y_{\hat{g}}>0$, $Q_{\hat{g}}\in L^1(M,\hat{g})$, and $Q_{\hat{g}}\geq 0$.
\end{itemize}
Then the fourth-order mass $m(\hat{g})$ is nonnegative. Moreover, $m(\hat{g})=0$ if and only if $(M,\hat{g})$ is isometric to the Euclidean space $\mathbb{R}^n$.
\end{proposition}

According to \cite[Proposition 8.3]{gong2025compactnessnoncompactnesstheoremsfourth} and the definition of $\mci(p,\delta)$ (Definition \ref{def:mci}), we can find the relation between the $Q$-mass and the limit of the boundary integral.

\begin{proposition}\label{prop:thm1.2 (c)}
Let $n\geq 5$ and $(M^n,g)$ be a closed manifold such that $\ker P_g = 0$. Suppose that $p \in \mcz\subset M$. Then, $(M\setminus\{p\}, \hat{g}:=G_g^{\frac{4}{n-4}}g)$ is an asymptotically flat manifold with decay rate $\tau >\frac{n-4}{2}$ such that $Q_{\hat{g}}=0$.
Moreover,
$$
m(\hat{g})=\frac{4(n-1)}{n-4}\mci(p).
$$
\end{proposition}

\begin{proof}[Proof of Theorem \ref{thm:energy estimate}]

(a) is from Corollary \ref{cor:thm1.2 (a)}.

(b) is from Lemma \ref{lem:thm1.2 (b)}.

(c) is from Proposition \ref{prop:thm1.2 (c)}.

(d) is from Definition \ref{def:test bubble Green}, Lemma \ref{lem:brendle splitting}, and equation \eqref{eq:w[s] proof}.
\end{proof}

\section{Setup for the blow-up analysis}

In this section, we prepare the framework for the blow-up analysis carried out in the next section.

\subsection{Estimates on test bubbles in dimensions $5 \le n \le 7$}
In this subsection, we assume that $5 \le n \le 7$, and that the metric $g \in [g_0]$ is chosen so that it induces conformal normal coordinates at a point $p \in M$. In these coordinates, the metric satisfies
\[
\det(g) = 1 + O(r^{N}),
\]
where $r$ denotes the distance function from $p$, and $N$ is a sufficiently large positive integer.

The test bubble constructed in this case is defined by
\begin{equation}\label{testfunction}
v_{p,\epsilon} (y) = \big(\frac{\mfc(n)}{\mu_\infty} \big)^{\frac{n-4}{8}} \epsilon^{\frac{n-4}{2}} \bigg( \chi_\delta    (\epsilon^2 + r^2 )^{-\frac{n-4}{2}} + ( 1- \chi_\delta )  G_p \bigg)
\end{equation}
where $\mfc(n) = n(n^2-4)(n-4)$, $\chi : \mathbb{R} \to [0,1]$ is a cutoff function satisfying $\chi = 1$ on $[0,1]$ and $\chi = 0$ on $[2,\infty)$, and $\chi_\delta(r) = \chi(r/\delta)$. Here $G_p$ denotes the Green’s function of the Paneitz operator with pole at $p$, and $\delta>0$ is chosen so that $\epsilon \ll \delta$.

We recall the following properties of the Green’s function: for $0 \le k \le 3$,
\begin{equation*}
\big| \nabla_g^k \big( G_p - r^{4-n} - A_p \big) \big| \le C r^{1-k},
\end{equation*}
as shown in \cite[Proposition~2.5]{MR3420504}. The constant $A_p$ depends continuously on $p$, and in particular admits a positive lower bound uniformly in $p$.

We also define
\[
u_\epsilon
:= \left( \frac{\epsilon}{\epsilon^2 + r^2} \right)^{\frac{n-4}{2}} .
\]

\begin{lemma}
In conformal normal coordinates, for any radial function $f=f(r)$, we have
\begin{equation*}
P_g f
= \Delta^2 f
+ O(r^2)\, f''
+ O(r)\, f'
+ O(|Q_g|)\, f .
\end{equation*}
Moreover, for the metric $g$ in these coordinates,
\[
Q_g(p) = C(n)\, |W_g(p)|^2 .
\]
\end{lemma}

\begin{proof}
We recall from \cite[Lemma~4.2]{MR3420504} that, for a radial function $f$,
\[
\Delta_g f = \Delta f + O(r^{N-1})\, f',
\]
\[
\Delta_g^2 f
= \Delta^2 f
+ O(r^{N-1})\, f'''
+ O(r^{N-2})\, f''
+ O(r^{N-3})\, f',
\]
and that the curvature quantities satisfy
\[
\ric_g\!\left(\frac{\partial}{\partial r},\frac{\partial}{\partial r}\right)=O(r^2), 
\qquad
|R_g| = O(r^2), 
\qquad
|\nabla_g R_g| = O(r).
\]

\end{proof}

\begin{lemma}\label{lem:uepsilon}
In $B_{2\delta}(p)$, we have
\begin{equation*}
P_g u_\epsilon - \mfc(n) u_\epsilon^{\frac{n+4}{n-4}}
= O(1)\, u_\epsilon .
\end{equation*}
\end{lemma}

\begin{proof}
In Euclidean space, the standard bubble satisfies
\[
\Delta^2 u_\epsilon - \mfc(n) u_\epsilon^{\frac{n+4}{n-4}} = 0 .
\]
By \cite[Lemma~4.2]{MR3420504}, we have the expansion
\begin{align*}
P_g u_\epsilon - \mfc(n) u_\epsilon^{\frac{n+4}{n-4}}
&= O\!\left( \epsilon^{-4} r^{N-2} (\epsilon^2 + r^2)^3 \right)
u_\epsilon^{\frac{n+4}{n-4}}
+ O(1)\, u_\epsilon .
\end{align*}

Observe that
\[
\epsilon^{-4} r^{N-2} (\epsilon^2 + r^2)^3
\, u_\epsilon^{\frac{8}{n-4}}
= r^{N-2} (\epsilon^2 + r^2)^{-1}
\le C ,
\]
uniformly in $B_{2\delta}(p)$, provided $\delta$ is chosen sufficiently small. Therefore,
\[
P_g u_\epsilon - \mfc(n) u_\epsilon^{\frac{n+4}{n-4}}
= O(1)\, u_\epsilon ,
\]
as claimed.
\end{proof}

\begin{lemma} \label{lem:uepg}
In $B_{2\delta}(p) \setminus B_{\delta} (p)$, we have
\begin{align*}
&P_{g} \big( \chi_\delta ( u_{\epsilon} - \epsilon^{\frac{n-4}{2}} G_p) \big) + \epsilon^{\frac{n-4}{2}} A_p   P_{g}  \chi_\delta  \\
= &  \mfc (n) \chi_\delta u_{\epsilon}^{\frac{n+4}{n-4}}  + O(1) u_{\epsilon} +  O(\epsilon^{\frac{n}{2}} \delta^{-n-2} ) +O(\epsilon^{\frac{n-4}{2}}\delta^{-3}).
\end{align*}
\end{lemma}

\begin{proof}

Set $U_{\epsilon} = u_{\epsilon} -    \epsilon^{\frac{n-4}{2}} G_p$. 
On $B_{2\delta}(x) \setminus B_{\delta} (x)$, we compute
\begin{align*}
P_{g} \big( \chi_\delta U_{\epsilon}\big) {(y)} = & \chi_\delta (P_{g} - P_{g}1) u_{\epsilon}   +   U_{\epsilon}   P_{g} \chi_\delta \\
+ &\sum_{i=1}^3 O(|\nabla^{4-i} \chi_\delta \cdot \nabla^i   U_{\epsilon}|  )  + O( |\nabla \chi_\delta \cdot \nabla U_{\epsilon}|).  
\end{align*}

Using $|\nabla^i \chi_\delta |  = O(\delta^{-i})$ and  $\epsilon \ll \delta$, we observe that
\begin{align*}
U_{\epsilon} = &\epsilon^{\frac{n-4}{2}} ( r^2 +\epsilon^2 )^{-\frac{n-4}{2}} - \epsilon^{\frac{n-4}{2}} \big( r^{4-n}  + A_p + O(r) \big) \\
=& O(\epsilon^{\frac{n}{2}} \delta^{2-n} )  -A_p\epsilon^{\frac{n-4}{2}}  + O(\epsilon^{\frac{n-4}{2}}\delta); 
\end{align*}
$$ |\nabla^i ( u_{\epsilon} - \epsilon^{\frac{n-4}{2}} r^{4-n} ) | \le C \epsilon^{\frac{n}{2}} \delta^{2-n-i} + C \epsilon^{\frac{n-4}{2}} \delta^{1-i} \text{ for $i \le 4$};$$
$$P_{g} \chi_\delta = O(\delta^{-4}).$$

Collecting these estimates yields
\begin{align*}
&P_{g} \big( \chi_\delta U_{\epsilon} \big)+\epsilon^{\frac{n-4}{2}} A_p   P_{g}  \chi_\delta \\
= &\chi_\delta P_{g} U_{\epsilon} + O(\epsilon^{\frac{n}{2}} \delta^{-n-2} ) +O(\epsilon^{\frac{n-4}{2}}\delta^{-3}) \\ =& \mfc(n) \chi_\delta u_{\epsilon}^{\frac{n+4}{n-4}}  + O(1) u_{\epsilon} + O(\epsilon^{\frac{n}{2}} \delta^{-n-2} ) +O(\epsilon^{\frac{n-4}{2}}\delta^{-3}).
\end{align*}

\end{proof}

\begin{proposition} \label{prop:approx}
We have
\begin{align*}
& \bigg| P_{g} v_{p,\epsilon}  -\mu_\infty v_{p,\epsilon} ^{\frac{n+4}{n-4}} + \bigg( \frac{\mfc(n)}{\mu_\infty} \bigg)^{\frac{n-4}{8}} \epsilon^{\frac{n-4}{2}} A_p \cdot P_{g}  \chi_\delta  \bigg| \\
\le & C \bigg(\frac{\epsilon}{\epsilon^2 + r^2} \bigg)^{\frac{n-4}{2}}  \1_{\{r \le 2\delta \}} +C  \epsilon^{\frac{n-4}{2}} \delta^{-3} \1_{\{\delta \le r \le 2\delta \}}  +  C\bigg(\frac{\epsilon}{\epsilon^2 + r^2} \bigg)^{\frac{n+4}{2}} \1_{\{r \ge \delta \}}.
\end{align*}
\end{proposition}

\begin{proof}

Set
$U_1 = \big(\frac{\mfc(n) }{\mu_\infty} \big)^{\frac{n-4}{8}} \epsilon^{\frac{n-4}{2}} \chi_\delta (\epsilon^2 + r^2 )^{-\frac{n-4}{2}} $. On the region $\{ r \le \delta \}$, the estimate follows directly from Lemma~\ref{lem:uepsilon}.  
On the annulus $\{ \delta \le r \le 2\delta \}$, Lemma~\ref{lem:uepg} yields
\begin{align*}
&  P_{g} v_{p,\epsilon} -\mu_\infty v_{p,\epsilon}^{\frac{n+4}{n-4}} + \bigg( \frac{\mfc(n)}{\mu_\infty} \bigg)^{\frac{n-4}{8}} \epsilon^{\frac{n-4}{2}} A_p \cdot P_{g}  \chi_\delta  \\
= & \mu_\infty \chi_\delta^{-\frac{8}{n-4}} U_1^{\frac{n+4}{n-4}} - \mu_\infty  v_{p,\epsilon}^{\frac{n+4}{n-4}} + O(1) \epsilon^{\frac{n-4}{2}} u_{\epsilon} + O(\epsilon^{\frac{n}{2}} \delta^{-n-2} ) +O (  \epsilon^{\frac{n-4}{2}} \delta^{-3} ).
\end{align*}

Since
$$\mu_\infty \chi_\delta^{-\frac{8}{n-4}} U_1^{\frac{n+4}{n-4}} - \mu_\infty  v_{p,\epsilon}^{\frac{n+4}{n-4}} = O \bigg( \big( \frac{\epsilon}{\epsilon^2 + r^2}\big)^{\frac{n+4}{2}} \bigg),$$
the desired bound follows in this region.

Finally, on the region $\{ r \ge 2\delta \}$, we use the elementary estimate
$$\epsilon^{\frac{n-4}{2}} G_p\le C \big( \frac{\epsilon}{\epsilon^2 + r^2} \big)^{\frac{n-4}{2}} \text{ on the set $\{ r \ge \delta \}$ .}$$
which immediately yields the stated inequality.

\end{proof}

\begin{proposition}
If $\delta$ is sufficiently small, then the Paneitz--Sobolev quotient of the test function $v_{p,\epsilon}$ satisfies
\begin{equation*}
\cf[v_{p, \epsilon}] \le Y_4 (S^n) - c\epsilon^{n-4} + C\delta \epsilon^{n-4} + C\delta^{-n+2} \epsilon^{n-2}
\end{equation*}
for some constant $c>0$ independent of $\epsilon$ and $\delta$.
\end{proposition}

\begin{proof}

Using the estimates
$$\int_{\{r \le 2\delta \}} \big(\frac{\epsilon}{\epsilon^2 + r^2} \big)^{\frac{n-4}{2}} \le C \delta^{4} \epsilon^{\frac{n}{2} -2},  \hspace{2mm} \int_{\{r \le 2\delta \}} \big(\frac{\epsilon}{\epsilon^2 + r^2} \big)^{n-4} \le C \delta^{8-n}  \epsilon^{n-4},$$
and 
$$\int_{\{r \ge \delta \}} \big(\frac{\epsilon}{\epsilon^2 + r^2} \big)^{n} \le C \delta^{-n} \epsilon^{n},$$
together with Proposition~\ref{prop:approx}, we obtain
\begin{align*}
&\int_M v_{p,\epsilon} \big( P_{g} v_{p,\epsilon} - \mu_\infty v_{p,\epsilon}^{\frac{n+4}{n-4}} \big) dv_g \\
\le & - C A_p \int_M \epsilon^{\frac{n-4}{2}} v_{p,\epsilon} P_{g}   \chi_\delta   + C  (\delta + \delta^{-n+8}) \epsilon^{n-4} + C \delta^{-n+2} \epsilon^{n-2}.
\end{align*}

We now estimate the main term:
\begin{align*}
&\int_M \epsilon^{\frac{n-4}{2}} v_{p,\epsilon} P_{g}   \chi_\delta dv_g\\
= & C \epsilon^{\frac{n-4}{2}} \int_M \big( \chi_\delta (u_{\epsilon} -\epsilon^{\frac{n-4}{2}}  G_p ) + \epsilon^{\frac{n-4}{2}} G_p  \big)  P_{g}    \chi_\delta \\
= & C \epsilon^{n-4} \bigg[ C(n) + \int_M \bigg( \big(\frac{1}{\epsilon^2 + r^2}\big)^{\frac{n-4}{2}} - \frac{1}{r^{n-4}} - A_p + O(\delta) \bigg) P_g \chi_\delta \bigg] \\
=  &  C(n)\epsilon^{n-4}  + O(\delta^{n-4} \epsilon^{n-4}) + O( \delta^{-2}\epsilon^{n-2}),
\end{align*}
where $C(n)>0$ is a dimensional constant.

Consequently,
\begin{equation}\label{eq:48}
\int_M v_{p,\epsilon} \big( P_{g} v_{p,\epsilon} - \mu_\infty v_{p,\epsilon}^{\frac{n+4}{n-4}} \big) \le   (-C A_p +  C \delta) \epsilon^{n-4}  + C  \delta^{-n+2} \epsilon^{n-2}.
\end{equation}

Next, we estimate the denominator:
\begin{align*}
\int_M v_{p,\epsilon}^{\frac{2n}{n-4}} dv_g \le & \bigg(\frac{ \mfc(n)}{\mu_\infty} \bigg)^{\frac{n}{4}} \bigg[ \int_{\{r \le \delta\}} \epsilon^n ( \epsilon^2 + r^2 )^{-n} + C  \int_{\{r \ge \delta\}} \epsilon^n ( \epsilon^2 + r^2 )^{-n} \bigg] \\
\le &\bigg(\frac{ Y_4(S^n) }{\mu_\infty} \bigg)^{\frac{n}{4}} + C \delta^{-n} \epsilon^{n}.
\end{align*}
Combining this estimate with \eqref{eq:48} yields the desired inequality.

\end{proof}

\subsection{Concentration--compactness principle}

\begin{definition}[Test bubble]
Let $\phi_p$ be the conformal factor that produces conformal normal coordinates at the point $p\in M$. We define
\begin{equation*}
\bar{u}_{(p,\epsilon)} :=
\begin{cases}
\phi_p\, v_{p,\epsilon}, & \text{if $5\le n \le 7$, where $v_{p,\epsilon}$ is defined in \eqref{testfunction},}\\[1mm]
\phi_p\, v_{p,\epsilon,\delta}, & \text{if $n\ge 8$, where $v_{p,\epsilon,\delta}$ is defined in \eqref{testfunction2}.}
\end{cases}
\end{equation*}
\end{definition}

We note that $\phi_p$ may be chosen so that it admits uniform positive lower and upper bounds, independent of $p$.

We now verify that Struwe's decomposition applies to our non-local flow. Given a sequence of times $t_\nu\to\infty$, we write
\[
u_\nu := u(t_\nu), \qquad g_\nu := u_\nu^{\frac{4}{n-4}} g_0 .
\]

\begin{proposition}[{\cite{MR1867939}}]{\label{lem:concom}}
Let $t_\nu \rightarrow \infty$ be a sequence of times. After passing to a subsequence, there exist a nonnegative integer $m$, a nonnegative smooth function $u_\infty$, and a sequence of $m$-tuples
\[
(p_{k,\nu}^*, \epsilon_{k,\nu}^*)_{1 \le k \le m}
\]
with the following properties:

\begin{enumerate}[label=(\alph*),leftmargin=*]

\item 
The function $u_\infty$ satisfies the equation
\[
- u_\infty
+ \mu_\infty P_{g_0}^{-1}
\!\left( u_\infty^{\frac{n+4}{n-4}} \right) = 0 .
\]

\item 
For all $i \neq j$, we have
\[
\frac{\epsilon_{i,\nu}^*}{\epsilon_{j,\nu}^*}
+ \frac{\epsilon_{j,\nu}^*}{\epsilon_{i,\nu}^*}
+ \frac{ d(p_{i,\nu}^*, p_{j,\nu}^*)^2 }
{ \epsilon_{i,\nu}^* \epsilon_{j,\nu}^* }
\longrightarrow \infty .
\]

\item 
We have
\[
\Big\|
u_\nu - u_\infty
- \sum_{k=1}^{m}
\bar{u}_{(p_{k,\nu}^*, \epsilon_{k,\nu}^*)}
\Big\|_{W^{2,2}(M)}
\longrightarrow 0 .
\]

\item 
Either $u_\infty \equiv 0$ or $u_\infty > 0$ everywhere on $M$.

\item 
Moreover,
\begin{equation*}
\mathcal{F}_\infty
= \bigg(
\mathcal{F}[u_\infty]^{\frac{n}{4}}
+ m\, Y_4(S^n)^{\frac{n}{4}}
\bigg)^{\frac{4}{n}} .
\end{equation*}

\end{enumerate}

\end{proposition}

\begin{proof}
For our flow, the quantity $\int_M u P_{g_0} u \, \vol_{g_0}$ is constant along the flow, while
$\int_M u^{\frac{2n}{n-4}} \, \vol_{g_0}$ remains uniformly bounded. Moreover, by
Proposition~\ref{lem:LnormPf},
\begin{equation*}
\int_M 
\big| P_{g_0} u_\nu - \mu_\infty u_\nu^{\frac{n+4}{n-4}} \big|^{\frac{2n}{n+4}}
\, \vol_{g_0}
\longrightarrow 0 .
\end{equation*}
Hence, $\{u_\nu\}$ is a Palais--Smale sequence, and the concentration--compactness result of
\cite{MR1867939} applies.

Property~(d) follows from the strong maximum principle established in \cite{MR3420504}.

To verify~(e), we use Lemma~\ref{lem:interaction} below and compute
\begin{align*}
V_\infty
&= \lim_{\nu \to \infty} \int_M u_\nu^{\frac{2n}{n-4}} \, \vol_{g_0} \\
&= \lim_{\nu \to \infty}
\bigg(
\int_M u_\infty^{\frac{2n}{n-4}} \, \vol_{g_0}
+ \sum_{k=1}^{m}
\int_M
\bar{u}_{(p_{k,\nu}^*, \epsilon_{k,\nu}^*)}^{\frac{2n}{n-4}}
\, \vol_{g_0}
\bigg) \\
&=
\Big( \frac{\mathcal{F}[u_\infty]}{\mu_\infty} \Big)^{\frac{n}{4}}
+ m \Big( \frac{Y_4(S^n)}{\mu_\infty} \Big)^{\frac{n}{4}} .
\end{align*}

\end{proof}

Below, we denote by $d(p,\cdot)$ the distance function from the point $p$ with respect to the metric $g_0$.

\begin{lemma}\label{lemma 5.8}
For $\epsilon$ sufficiently small, we have
\begin{align*}
&\bigg| P_{g_0 }  \bar{u}_{(p, \epsilon)} -\mu_\infty  \bar{u}_{(p, \epsilon)}^{\frac{n+4}{n-4}} \bigg|\nonumber\\
\le &  C \bigg(\frac{\epsilon}{\epsilon^2 + d(p, \cdot)^2} \bigg)^{\frac{n-4}{2}}  \1_{\{d(p, \cdot)\le 2\delta \}} + C \epsilon^{\frac{n-4}{2}} \delta^{-4}  \1_{\{\delta \le d(p, \cdot) \le 2\delta \}} \nonumber \\
& + C \bigg(\frac{\epsilon}{\epsilon^2 + d(p, \cdot)^2} \bigg)^{\frac{n+4}{2}} \1_{\{d(p, \cdot) \ge \delta \}}.
\end{align*}
\end{lemma}

\begin{proof}
For the case $5 \le n \le 7$, this estimate follows directly from Proposition~\ref{prop:approx}. The case $n\ge 8$ easily follows from Lemma \ref{lem:brendle splitting} and \ref{lem:correction term}.
\end{proof}

\begin{lemma} \label{lem:interaction}
We have the estimate:
\begin{equation*}
\int_M \bar{u}_{(p_i, \epsilon_i)} \bar{u}_{(p_j, \epsilon_j)}^{\frac{n+4}{n-4}} \dvol \le C \bigg( \frac{\epsilon_j^2 + d(p_i, p_j)^2 }{\epsilon_i \epsilon_j }\bigg)^{-\frac{n-4}{2}}.
\end{equation*}
\end{lemma}

\begin{proof}
The proof is standard; see \cite[Lemma~B.4]{MR2168505}.
\end{proof}

\begin{lemma} \label{lem:inter2}
We have the estimate
\begin{align*}
&\int_M \bar{u}_{(p_i, \epsilon_i)} \big| \pa \bar{u}_{(p_j, \epsilon_j)} - \mu_\infty \bar{u}_{(p_j, \epsilon_j)}^{\frac{n+4}{n-4}}\big| \\
\le & C \left( \delta^4 + \delta^{n-4} + \frac{\epsilon_j^4}{\delta^4} \right) \left( \frac{\epsilon_j^2 + d(p_i, p_j)^2 }{\epsilon_i \epsilon_j } \right)^{-\frac{n-4}{2}}.
\end{align*}
\end{lemma}

\begin{proof}
We only need to consider the case $5 \le n \le 7$. By Lemma~\ref{lemma 5.8}, we have
\begin{align*}
\bigg| P_{g_0 }  \bar{u}_{(p_j, \epsilon_j)} -\mu_\infty  \bar{u}_{(p_j, \epsilon_j)}^{\frac{n+4}{n-4}} \bigg| (y) \le &  C (1+ \delta^{n-8}) \bigg(\frac{\epsilon_j}{\epsilon_j^2 + d(p_j, y)^2} \bigg)^{\frac{n-4}{2}}  \1_{\{d(p_j, y)\le 4\delta \}}  \\
& + C \bigg(\frac{\epsilon_j}{\epsilon_j^2 + d(p_j, y)^2} \bigg)^{\frac{n+4}{2}} \1_{\{d(p_j, y) \ge \frac{\delta}{2} \}}.
\end{align*}

On the set $\{ y \mid 2d(p_i, y) \le \epsilon_j + d(p_i, p_j) \}\cap \{d(y, p_j )\le 4\delta \}$, we have
$$\epsilon_j + d(y, p_j) \ge  \epsilon_j +d(p_i, p_j) -d(p_i, y) \ge \frac{1}{2} ( \epsilon_j + d(p_i, p_j)).$$
Consequently,
$$d(p_i, y) \le \frac{1}{2}(\epsilon_j + d(p_i, p_j) ) \le \epsilon_j +d(y, p_j) \le 8 \delta.$$

Using this observation, we estimate
\begin{align*}
& \int_{\{d(y, p_j) \le 4\delta\}} \frac{\epsilon_i^{\frac{n-4}{2}}}{(\epsilon_i^2 + d(p_i, y)^2 )^{\frac{n-4}{2}}} \cdot \frac{\epsilon_j^{\frac{n-4}{2}}}{(\epsilon_j^2 + d(p_j, y)^2 )^{\frac{n-4}{2}}}  \dvol \\
\le & \int_{\{ y \mid 2d(p_i, y) \le \epsilon_j + d(p_i, p_j) \} \cap \{d(y, p_j )\le 4\delta \}} \frac{\epsilon_i^{\frac{n-4}{2}}}{(\epsilon_i^2 + d(p_i, y)^2 )^{\frac{n-4}{2}}} \cdot \frac{\epsilon_j^{\frac{n-4}{2}}}{(\epsilon_j^2 + d(p_j, y)^2 )^{\frac{n-4}{2}}} \\
& + \int_{\{ y \mid 2d(p_i, y) \ge \epsilon_j + d(p_i, p_j) \} \cap \{d(y, p_j )\le 4\delta \}} \frac{\epsilon_i^{\frac{n-4}{2}}}{(\epsilon_i^2 + d(p_i, y)^2 )^{\frac{n-4}{2}}} \cdot \frac{\epsilon_j^{\frac{n-4}{2}}}{(\epsilon_j^2 + d(p_j, y)^2 )^{\frac{n-4}{2}}} \\
\le & C \int_{\{ d(p_i, y) \le 8\delta \}} \frac{\epsilon_i^{\frac{n-4}{2}}}{(\epsilon_i^2 + d(p_i, y)^2 )^{\frac{n-4}{2}}} \cdot \frac{\epsilon_j^{\frac{n-4}{2}}}{(\epsilon_j^2 + d(p_i, p_j)^2 )^{\frac{n-4}{2}}} \\
& + C \int_{\{ d(y, p_j) \le 4 \delta \}} \frac{\epsilon_i^{\frac{n-4}{2}}}{(\epsilon_i^2 + d(p_i, p_j)^2 )^{\frac{n-4}{2}}} \cdot \frac{\epsilon_j^{\frac{n-4}{2}}}{(\epsilon_j^2 + d(p_j, y)^2 )^{\frac{n-4}{2}}} \\
\le &  C \delta^4 \frac{\epsilon_i^{\frac{n-4}{2}} \epsilon_j^{\frac{n-4}{2}}}{(\epsilon_j^2 + d(p_i, p_j)^2 )^{\frac{n-4}{2}}}.
\end{align*}

On the set $\{ y \mid 2d(p_i, y) \le \epsilon_j + d(p_i, p_j) \}\cap \{d(y, p_j )\ge \frac{\delta}{2} \}$, we again have
$$\epsilon_j + d(y, p_j) \ge  \epsilon_j +d(p_i, p_j) -d(p_i, y) \ge \frac{1}{2} ( \epsilon_j + d(p_i, p_j)),$$
and therefore
\begin{align*}
& \int_{\{d(y, p_j) \le 4\delta\}} \frac{\epsilon_i^{\frac{n-4}{2}}}{(\epsilon_i^2 + d(p_i, y)^2 )^{\frac{n-4}{2}}} \cdot \frac{\epsilon_j^{\frac{n+4}{2}}}{(\epsilon_j^2 + d(p_j, y)^2 )^{\frac{n+4}{2}}}  \dvol \\
\le & C \frac{\epsilon_j^4}{\delta^4} \frac{\epsilon_i^{\frac{n-4}{2}} \epsilon_j^{\frac{n-4}{2}}}{(\epsilon_j^2 + d(p_i, p_j)^2 )^{\frac{n-4}{2}}}.
\end{align*}
Combining the above estimates completes the proof.
\end{proof}

\section{Blow-up analysis}

\subsection{Second variation of the energy}

In this subsection, we establish a uniform estimate for the second variation of the Paneitz--Sobolev quotient along the flow.

We begin by recalling some basic spectral facts.

\begin{proposition}
Assume that $u_\infty \not\equiv 0$. Then there exist a sequence of smooth functions
$\{ \psi_a \}_{a\in\mathbb{N}}$ and a sequence of positive real numbers
$\{ \lambda_a \}_{a\in\mathbb{N}}$ with the following properties:

\begin{enumerate}[label=(\alph*), leftmargin=*]

\item
For each $a\in\mathbb{N}$, the function $\psi_a$ satisfies
\begin{equation*}
P_{g_0} \psi_a - \lambda_a\, u_\infty^{\frac{8}{n-4}} \psi_a = 0 .
\end{equation*}

\item
For all $a,b\in\mathbb{N}$, we have the orthonormality relation
\begin{equation*}
\int_M u_\infty^{\frac{8}{n-4}} \psi_a \psi_b \, \vol = \delta_{ab} .
\end{equation*}

\item
The linear span of $\{ \psi_a \mid a\in\mathbb{N} \}$ is dense in $L^2(M)$.

\item
The eigenvalues satisfy $\lambda_a \to \infty$ as $a \to \infty$.
\end{enumerate}
\end{proposition}

\begin{proof}
Consider the linear operator
\[
\phi \longmapsto u_\infty^{-\frac{8}{n-4}}\, P_{g_0} \phi .
\]
This operator is symmetric with respect to the weighted inner product
\[
(\phi_1,\phi_2)
\longmapsto
\int_M u_\infty^{\frac{8}{n-4}} \phi_1 \phi_2 \, \vol .
\]
The result then follows from standard spectral theory for self-adjoint operators.
\end{proof}

Let $A \subset \mathbb{N}$ be the finite set defined by
\[
A := \Big\{ a \in \mathbb{N} \;:\; \lambda_a \le \frac{n+4}{n-4}\, \mu_\infty \Big\}.
\]
We define the projection operator $\Pi$ by
\begin{equation*}
\Pi \phi
:= \sum_{a \notin A}
\bigg( \int_M \psi_a \phi \, \vol \bigg)
u_\infty^{\frac{8}{n-4}} \psi_a
= \phi
- \sum_{a \in A}
\bigg( \int_M \psi_a \phi \, \vol \bigg)
u_\infty^{\frac{8}{n-4}} \psi_a .
\end{equation*}

\begin{lemma}\label{lp}
Assume that $u_\infty \not\equiv 0$. For every $1 \le p < \infty$, there exists a constant
$C>0$ (depending only on $p$) such that
\begin{equation*}
\| \phi \|_{L^p(M)}
\le
C \bigg\|
P_{g_0} \phi
- \frac{n+4}{n-4}\, \mu_\infty u_\infty^{\frac{8}{n-4}} \phi
\bigg\|_{L^p(M)}
+
C \sup_{a \in A}
\bigg|
\int_M u_\infty^{\frac{8}{n-4}} \psi_a \phi \, \vol
\bigg|.
\end{equation*}
\end{lemma}

\begin{proof}
We prove the statement for the case $p=1$; the general case follows similarly.

Suppose the conclusion were false. Then there exists a sequence $\{\phi_i\}$ such that
\[
\|\phi_i\|_{L^1(M)} = 1,
\]
and
\[
\bigg\|
P_{g_0} \phi_i
- \frac{n+4}{n-4}\, \mu_\infty u_\infty^{\frac{8}{n-4}} \phi_i
\bigg\|_{L^1(M)}
+ \sup_{a \in A}
\bigg|
\int_M u_\infty^{\frac{8}{n-4}} \psi_a \phi_i \, \vol
\bigg|
\longrightarrow 0 .
\]

By compactness of Radon measures, there exists a signed Radon measure $\Phi$ such that
\[
\phi_i \rightharpoonup \Phi
\qquad \text{and} \qquad
\|\Phi\| = 1 .
\]
(Here we decompose $\phi_i = \phi_i^+ - \phi_i^-$ and apply compactness to each part.)

For each $a \in \mathbb{N}$, we compute
\begin{align*}
\bigg( \lambda_a - \frac{n+4}{n-4}\, \mu_\infty \bigg)
\Phi\!\left( u_\infty^{\frac{8}{n-4}} \psi_a \right)
&=
\lim_{i \to \infty}
\int_M
\phi_i
\bigg(
P_{g_0}
- \frac{n+4}{n-4}\, \mu_\infty u_\infty^{\frac{8}{n-4}}
\bigg) \psi_a
\, \vol \\
&= 0 .
\end{align*}

If $a \notin A$, then
\(
\lambda_a - \frac{n+4}{n-4}\mu_\infty > 0
\),
and hence
\[
\Phi\!\left( u_\infty^{\frac{8}{n-4}} \psi_a \right) = 0 .
\]
For $a \in A$, the same conclusion follows from the vanishing of the projection terms by assumption.
Therefore,
\[
\Phi\!\left( u_\infty^{\frac{8}{n-4}} \psi_a \right) = 0
\qquad \text{for all } a \in \mathbb{N}.
\]

Since finite linear combinations of the eigenfunctions $\{\psi_a\}$ are dense in $C^0(M)$,
this implies $\Phi = 0$, which contradicts $\|\Phi\| = 1$.
The contradiction completes the proof.

\end{proof}

\begin{lemma}\label{lem:lpestimate}
Assume that $u_\infty \not\equiv 0$.

\begin{enumerate}[label=(\alph*), leftmargin=*]

\item
There exists a constant $C>0$ such that
\begin{align*}
\| \phi\|_{L^{\frac{n+4}{n-4}}(M)}
\le\;&
C \bigg\|
\Pi \bigg(
P_{g_0} \phi
- \frac{n+4}{n-4}\, \mu_\infty u_\infty^{\frac{8}{n-4}} \phi
\bigg)
\bigg\|_{L^{\frac{n(n+4)}{n^2+16}}(M)}
\nonumber\\
&\quad
+ C \sup_{a \in A}
\bigg|
\int_M u_\infty^{\frac{8}{n-4}} \psi_a \phi \, \vol
\bigg|.
\end{align*}

\item
There exists a constant $C>0$ such that
\begin{align*}
\| \phi\|_{L^{1}(M)}
\le\;&
C \bigg\|
\Pi \bigg(
P_{g_0} \phi
- \frac{n+4}{n-4}\, \mu_\infty u_\infty^{\frac{8}{n-4}} \phi
\bigg)
\bigg\|_{L^{1}(M)}
\nonumber\\
&\quad
+ C \sup_{a \in A}
\bigg|
\int_M u_\infty^{\frac{8}{n-4}} \psi_a \phi \, \vol
\bigg|.
\end{align*}

\end{enumerate}
\end{lemma}

\begin{proof}
This follows from Lemma~\ref{lp}, standard $L^p$ estimates for elliptic operators, and the identity
\begin{align*}
& \bigg\|  \pa \phi -\frac{n+4}{n-4} \mu_\infty u_\infty^{\frac{8}{n-4}} \phi  \bigg \|_{L^p(M)} \\
= &  \bigg\| \Pi \bigg( \pa \phi -\frac{n+4}{n-4} \mu_\infty u_\infty^{\frac{8}{n-4}} \phi \bigg) - \sum_{a \in A} \big(\lambda_a - \frac{n+4}{n-4} \mu_\infty \big) \big( \int_M \wgt \psi_a \phi \big) \wgt \psi_a\bigg \|_{L^p(M)} \\
\le & \bigg\| \Pi \bigg( \pa \phi -\frac{n+4}{n-4} \mu_\infty u_\infty^{\frac{8}{n-4}} \phi \bigg) \bigg \|_{L^p(M)} + C\sup_{a \in A} \bigg| \int_M u_\infty^{\frac{8}{n-4}} \psi_a \phi \bigg|
\end{align*}
which holds by triangle inequality.
\end{proof}

\begin{lemma}
Assume that $u_\infty \not\equiv 0$. Then there exists a positive constant $\zeta>0$ with the following property: for every vector
$z=(z_a)_{a\in A}\in \mathbb{R}^A$ satisfying $|z|\le \zeta$, there exists a smooth function $\bar u_z$ such that
\begin{equation*}
\int_M u_\infty^{\frac{8}{n-4}} \big(\bar u_z - u_\infty\big)\,\psi_a \, \vol
= z_a
\qquad \text{for all } a\in A,
\end{equation*}
and
\begin{equation*}
\Pi\bigg(
P_{g_0}\bar u_z - \mu_\infty \bar u_z^{\frac{n+4}{n-4}}
\bigg)=0 .
\end{equation*}
Moreover, the map $z\mapsto \bar u_z$ is real analytic.
\end{lemma}

\begin{proof}
Near $u_\infty$, the Paneitz--Sobolev functional is real analytic, the conclusion follows directly from the implicit function theorem.
\end{proof}

We now refine Struwe's decomposition. In the case $u_\infty \not\equiv 0$, we proceed as follows.

For each $\nu\in\mathbb{N}$, let $\mathcal{A}_\nu$ denote the set of all tuples
\[
\big(z,(p_k,\epsilon_k,\alpha_k)_{1\le k\le m}\big)
\in \mathbb{R}^A \times (M\times \mathbb{R}_+ \times \mathbb{R}_+)^m
\]
such that
\begin{equation*}
|z|\le \zeta
\end{equation*}
and, for each $1\le k\le m$,
\begin{equation*}
d(p_k,p_{k,\nu}^*) \le \epsilon_{k,\nu}^*,
\qquad
\frac12 \le \frac{\epsilon_k}{\epsilon_{k,\nu}^*} \le 2,
\qquad
\frac12 \le \alpha_k \le 2 .
\end{equation*}

Then there exists a tuple
\[
\big(z_\nu,(p_{k,\nu},\epsilon_{k,\nu},\alpha_{k,\nu})_{1\le k\le m}\big)\in\mathcal{A}_\nu
\]
such that
\begin{align}\label{ineq:bestapprox}
& \bigg\|
u_\nu - \bar u_{z_\nu} - \sum_{k=1}^m \alpha_{k,\nu}\,\bar u_{(p_{k,\nu},\epsilon_{k,\nu})}
\bigg\|_{W^{2,2}(M)} \nonumber \\
\le & 
\bigg\|
u_\nu - \bar u_{z} - \sum_{k=1}^m \alpha_{k}\,\bar u_{(p_{k},\epsilon_{k})}
\bigg\|_{W^{2,2}(M)}
\end{align}
for all $\big(z,(p_k,\epsilon_k,\alpha_k)_{1\le k\le m}\big)\in \mathcal{A}_\nu$.

In the case $u_\infty \equiv 0$, this can be viewed as a special case of the above construction, with $z_\nu \equiv 0$ and $\bar u_{z_\nu}\equiv 0$. More precisely, we proceed as follows.

For each $\nu\in\mathbb{N}$, let $\mathcal{A}_\nu$ denote the set of all $m$-tuples
\[
(p_k,\epsilon_k,\alpha_k)_{1\le k\le m} \in (M\times \mathbb{R}_+ \times \mathbb{R}_+)^m
\]
such that
\begin{equation}\label{def:Anu}
d(p_k,p_{k,\nu}^*) \le \epsilon_{k,\nu}^*,
\qquad
\frac12 \le \frac{\epsilon_k}{\epsilon_{k,\nu}^*} \le 2,
\qquad
\frac12 \le \alpha_k \le 2,
\end{equation}
for all $1\le k\le m$. Since $\mathcal{A}_\nu$ is compact, there exists an $m$-tuple
\[
(p_{k,\nu},\epsilon_{k,\nu},\alpha_{k,\nu})_{1\le k\le m}\in\mathcal{A}_\nu
\]
such that
\begin{equation*}
\bigg\|
u_\nu - \sum_{k=1}^m \alpha_{k,\nu}\,\bar u_{(p_{k,\nu},\epsilon_{k,\nu})}
\bigg\|_{W^{2,2}(M)}
\le
\bigg\|
u_\nu - \sum_{k=1}^m \alpha_{k}\,\bar u_{(p_{k},\epsilon_{k})}
\bigg\|_{W^{2,2}(M)}
\end{equation*}
for all $(p_k,\epsilon_k,\alpha_k)_{1\le k\le m}\in \mathcal{A}_\nu$.

Without loss of generality, we assume that $\epsilon_{i,\nu} \le \epsilon_{j,\nu}$ whenever $i \le j$.

\begin{proposition}\label{lem:approx}
\begin{enumerate}[label=(\alph*), leftmargin=*]

\item
For all $i \neq j$, we have
\begin{equation*}
\frac{\epsilon_{i,\nu}}{\epsilon_{j,\nu}}
+ \frac{\epsilon_{j,\nu}}{\epsilon_{i,\nu}}
+ \frac{ d(p_{i,\nu}, p_{j,\nu})^2 }{ \epsilon_{i,\nu}\epsilon_{j,\nu} }
\longrightarrow \infty \qquad \text{as } \nu \to \infty .
\end{equation*}

\item
We have
\begin{equation*}
\bigg\|
u_\nu - \bar u_{z_\nu}
- \sum_{k=1}^{m} \alpha_{k,\nu}\,\bar u_{(p_{k,\nu},\epsilon_{k,\nu})}
\bigg\|_{W^{2,2}(M)}
\longrightarrow 0
\qquad \text{as } \nu \to \infty .
\end{equation*}

\end{enumerate}
\end{proposition}

\begin{proof}
Part~(a) follows immediately from \eqref{def:Anu}. Part~(b) follows from the minimizing property
\eqref{ineq:bestapprox} together with Proposition~\ref{lem:concom}.
\end{proof}

\begin{proposition}
We have
\begin{equation*}
|z_\nu| = o(1),
\end{equation*}
and
\begin{equation*}
d(p_{k,\nu}, p_{k,\nu}^*) \le o(1)\,\epsilon_{k,\nu}^*, 
\qquad 
\frac{\epsilon_{k,\nu}}{\epsilon_{k,\nu}^*} = 1 + o(1),
\qquad 
\alpha_{k,\nu} = 1 + o(1),
\end{equation*}
for all $1 \le k \le m$. In particular, the tuple
\[
(p_{k,\nu}, \epsilon_{k,\nu}, \alpha_{k,\nu})_{1\le k \le m}
\]
is an interior point of\/ $\mathcal{A}_\nu$ for all sufficiently large $\nu$.
\end{proposition}

\begin{proof}
The conclusion follows from the estimate
\begin{align*}
& \bigg\|
\bar u_{z_\nu} + \sum_{k=1}^m \alpha_{k, \nu} \bar u_{(p_{k,\nu},\epsilon_{k,\nu})}
- u_\infty - \sum_{k=1}^m  \bar u_{(p^*_{k,\nu},\epsilon^*_{k,\nu})}
\bigg\|_{W^{2,2}(M)} \\
\le\;&
\bigg\|
u_\nu - u_\infty - \sum_{k=1}^m  \bar u_{(p_{k,\nu}^*,\epsilon_{k,\nu}^*)}
\bigg\|_{W^{2,2}(M)}
+ 
\bigg\|
u_\nu - \bar u_{z_\nu} - \sum_{k=1}^m \alpha_{k, \nu} \bar u_{(p_{k,\nu},\epsilon_{k,\nu})}
\bigg\|_{W^{2,2}(M)} \\
=\;& o(1),
\end{align*}
together with Lemma~\ref{lem:interaction}.
\end{proof}

We decompose $u_\nu$ as
\begin{equation*}
u_\nu = v_\nu + w_\nu,
\end{equation*}
where
\begin{equation*}
v_\nu := \uznu + \sum_{k=1}^m \alpha_{k,\nu}\,\bar{u}_{(p_{k,\nu}, \epsilon_{k,\nu})},
\end{equation*}
and
\begin{equation*}
w_\nu := u_\nu - \uznu - \sum_{k=1}^m \alpha_{k,\nu}\,\bar{u}_{(p_{k,\nu}, \epsilon_{k,\nu})}.
\end{equation*}
By Proposition~\ref{lem:approx}, we have $\|w_\nu\|_{W^{2,2}(M)} = o(1)$.

We also remark that the case $u_\infty \equiv 0$ can be treated as a special case of the above decomposition by taking $\uznu \equiv 0$.

\begin{proposition}\label{lem:orth1}
We have:

\begin{enumerate}[label=(\alph*), leftmargin=*]

\item
For every $a \in A$, we have
\begin{equation*}
\bigg| \int_M u_\infty^{\frac{8}{n-4}} \psi_a\, w_\nu \, \dvol \bigg|
\le o(1)\int_M |w_\nu| \, \dvol .
\end{equation*}

\item
For every $1 \le k \le m$, we have
\begin{equation*}
\bigg| \int_M \bar{u}_{(p_{k,\nu},\epsilon_{k,\nu})}^{\frac{n+4}{n-4}}\, w_\nu \, \dvol \bigg|
\le o(1)
\bigg( \int_M |w_\nu|^{\frac{2n}{n-4}} \, \dvol \bigg)^{\frac{n-4}{2n}} .
\end{equation*}

\item
For every $1 \le k \le m$, we have
\begin{align*}
\bigg| &\int_M \bar{u}_{(p_{k, \nu}, \epsilon_{k, \nu})}^{\frac{n+4}{n-4}} \frac{\epsilon_{k, \nu}^2 - d(p_{k, \nu}, x)^2}{\epsilon_{k, \nu}^2 + d(p_{k, \nu}, x)^2} w_v \dvol \bigg| \nonumber\\
&\le o(1) \bigg( \int_M |w_\nu|^{\frac{2n}{n-4}} \dvol \bigg)^{\frac{n-4}{2n}}        .
\end{align*}

\item
For every $1 \le k \le m$, we have
\begin{align*}
\bigg| &\int_M \bar{u}_{(p_{k, \nu}, \epsilon_{k, \nu})}^{\frac{n+4}{n-4}} \frac{\epsilon_{k, \nu} \mathrm{exp}_{p_{k, \nu}}^{-1}(x)}{\epsilon_{k, \nu}^2 + d(p_{k, \nu}, x)^2} w_v \dvol \bigg| \nonumber\\
&\le o(1) \bigg( \int_M |w_\nu|^{\frac{2n}{n-4}} \dvol \bigg)^{\frac{n-4}{2n}}.       
\end{align*}

\end{enumerate}
\end{proposition}

\begin{proof}
Let $\tilde{\psi}_{a,z} := \frac{\partial}{\partial z_a}\bar u_z$. 
Taking the variation with respect to $z_a$, we obtain
\[
\int_M w_\nu\, \pa \tilde{\psi}_{a,z}\, \dvol = 0 .
\]
This implies
\begin{align*}
\lambda_a \int_M u_\infty^{\frac{8}{n-4}} \psi_a\, w_\nu \, \dvol
&= \int_M w_\nu\, \pa \psi_a \, \dvol \\
&= \int_M w_\nu\, \pa\big(\psi_a - \tilde{\psi}_{a,z}\big)\, \dvol .
\end{align*}
Since $\lambda_a>0$, this proves part~(a).

From the minimizing property \eqref{ineq:bestapprox}, taking the variation with respect to
$\alpha_{k,\nu}$ yields
\[
\int_M w_\nu\, P_{g_0}\bar u_{(p_{k,\nu},\epsilon_{k,\nu})}\, \dvol = 0 .
\]
Using Lemma~\ref{lemma 5.8}, we obtain
\[
\bigg\|
P_{g_0}\bar u_{(p_{k,\nu},\epsilon_{k,\nu})}
- \mu_\infty \bar u_{(p_{k,\nu},\epsilon_{k,\nu})}^{\frac{n+4}{n-4}}
\bigg\|_{L^{\frac{2n}{n+4}}(M)} = o(1).
\]
Consequently,
\begin{align*}
\bigg|
\int_M \bar u_{(p_{k,\nu},\epsilon_{k,\nu})}^{\frac{n+4}{n-4}}\, w_\nu \, \dvol
\bigg|
&=
\bigg|
\int_M
\bigg(
\bar u_{(p_{k,\nu},\epsilon_{k,\nu})}^{\frac{n+4}{n-4}}
- \frac{1}{\mu_\infty} P_{g_0}\bar u_{(p_{k,\nu},\epsilon_{k,\nu})}
\bigg)
w_\nu \, \dvol
\bigg| \\
&\le o(1)\,\|w_\nu\|_{L^{\frac{2n}{n-4}}(M)}
\end{align*}
by H\"older's inequality. This proves part~(b).

The proofs of parts~(c) and~(d) are analogous and follow from the same argument, using the
corresponding variations with respect to the scale and translation parameters.
\end{proof}

Now we prove a uniform estimate for the second variation operator of the Paneitz--Sobolev quotient at $v_\nu$.

\begin{proposition}\label{lem:secondv}
For $\nu$ sufficiently large, we have
\begin{equation*}
\frac{n+4}{n-4}\,\mu_\infty
\int_M
\bigg(
u_\infty^{\frac{8}{n-4}}
+ \sum_{k=1}^m \bar{u}_{(p_{k,\nu},\epsilon_{k,\nu})}^{\frac{8}{n-4}}
\bigg)
w_\nu^{2}\, \dvol
\le (1-c)\, \| w_\nu \|_{W^{2,2}(M)}^{2},
\end{equation*}
for some constant $c>0$ independent of $\nu$.
\end{proposition}

\begin{proof}

Let us first consider the case $u_\infty \not\equiv 0$, and argue by contradiction.

After rescaling, we obtain a sequence of functions $\{\tilde w_\nu\}_{\nu\in\mathbb{N}}$ such that
$\|\tilde w_\nu\|_{W^{2,2}(M)}=1$ and
\begin{equation*}
\lim_{\nu\to\infty}
\frac{n+4}{n-4}\,\mu_\infty
\int_M
\bigg(
u_\infty^{\frac{8}{n-4}}
+ \sum_{k=1}^m \bar u_{(p_{k,\nu},\epsilon_{k,\nu})}^{\frac{8}{n-4}}
\bigg)\tilde w_\nu^{2}\, \dvol
\ge 1 .
\end{equation*}
Moreover,
\begin{equation*}
\int_M |\tilde w_\nu|^{\frac{2n}{n-4}} \, \dvol
\le Y_4(M,g_0)^{-\frac{n}{n-4}} .
\end{equation*}

We choose a sequence $\{N_\nu\}$ such that $N_\nu\to\infty$, $N_\nu \epsilon_{j,\nu}\to 0$ for all $j$, and
\begin{equation*}
\frac{1}{N_\nu}\,
\frac{\epsilon_{j,\nu}+d(p_{i,\nu},p_{j,\nu})}{\epsilon_{i,\nu}}
\longrightarrow \infty
\end{equation*}
for all $i<j$. For instance, one may take
\[
N_\nu
:= \min_{i<j}
\left\{
\left(\frac{\epsilon_{j,\nu}+d(p_{i,\nu},p_{j,\nu})}{\epsilon_{i,\nu}}\right)^{\!1/2},
\ \epsilon_{j,\nu}^{-1/2}
\right\}.
\]

Define
\begin{equation*}
\Omega_{j,\nu}
:= B_{N_\nu \epsilon_{j,\nu}}(p_{j,\nu})
\setminus \bigcup_{i=1}^{j-1} B_{N_\nu \epsilon_{i,\nu}}(p_{i,\nu}).
\end{equation*}

\medskip
\noindent
\textbf{Case 1.}
There exists an integer $1\le j\le m$ such that
\begin{equation*}
\lim_{\nu\to\infty}
\int_M
\bar u_{(p_{j,\nu},\epsilon_{j,\nu})}^{\frac{8}{n-4}}\, \tilde w_\nu^{2}\, \dvol
>0,
\end{equation*}
and
\begin{align*}
\lim_{\nu\to\infty}
\int_{\Omega_{j,\nu}}
&\Big(
(\Delta_{g_0}\tilde w_\nu)^2
- \big(4A_{g_0}-(n-2)\sigma_1(A_{g_0})\big)(\nabla \tilde w_\nu,\nabla \tilde w_\nu)
+ \frac{n-4}{2} Q_{g_0}\tilde w_\nu^2
\Big)\, \dvol
\nonumber\\
\le\;&
\lim_{\nu\to\infty}
\frac{n+4}{n-4}\,\mu_\infty
\int_M
\bar u_{(p_{j,\nu},\epsilon_{j,\nu})}^{\frac{8}{n-4}}\, \tilde w_\nu^{2}\, \dvol .
\end{align*}

We now define a sequence of functions $\hat w_\nu : T_{p_{j,\nu}}M \to \mathbb{R}$ by
\[
\hat w_\nu(\xi)
:= \epsilon_{j,\nu}^{\frac{n-4}{2}}\,
\tilde w_\nu\big(\exp_{p_{j,\nu}}(\epsilon_{j,\nu}\xi)\big),
\qquad \xi\in T_{p_{j,\nu}}M .
\]
By the scaling properties and Sobolev inequalities, we have
\begin{align}
\limsup_{\nu\to\infty}
\int_{\{ \xi\in T_{p_{j,\nu}}M : |\xi|\le N_\nu\}}
|\hat w_\nu(\xi)|^{\frac{2n}{n-4}} \, d\xi
&\le Y_4(M,g_0)^{-\frac{n}{n-4}},
\label{ineq30}\\
\limsup_{\nu\to\infty}
\int_{\{ \xi\in T_{p_{j,\nu}}M : |\xi|\le N_\nu\}}
|\nabla_{g_e} \hat w_\nu(\xi)|^{\frac{2n}{n-2}} \, d\xi
&\le C(M,g_0).
\label{ineq31}
\end{align}

We note that the quantity $\int_{\mathbb{R}^n} |\nabla_{g_e} \hat w_\nu(\xi)|^{\frac{2n}{n-2}}\, d\xi$ is also scale invariant. 
Estimates \eqref{ineq30} and \eqref{ineq31} therefore imply that, for every $L>0$, the norm
$\|\hat w_\nu\|_{W^{1,2}(B_L(0))}$ is uniformly bounded for all sufficiently large $\nu$.

Combining this with $\|\tilde w_\nu\|_{\wtt}=1$ and the scale invariance of the $W^{2,2}$-norm, we may extract a subsequence and find
$\hat w \in W^{2,2}_{\mathrm{loc}}(\mathbb{R}^n)$ such that
\[
\hat w_\nu \rightharpoonup \hat w
\qquad \text{weakly in } W^{2,2}_{\mathrm{loc}}(\mathbb{R}^n).
\]
(Here one uses a Bochner identity to control the $L^2$ norm of the Hessian by the $L^2$ norm of the Laplacian on bounded domains.)
Moreover, the following hold:
\begin{equation}\label{ineq32}
\int_{\mathbb{R}^n} \bigg( \frac{1}{1+|\xi|^2 }\bigg)^4 \hat w(\xi)^2\, d\xi > 0,
\end{equation}
and
\begin{equation}\label{ineq33}
\int_{\mathbb{R}^n} (\Delta_{g_e} \hat w)^2\, d\xi
\le n(n+4)\int_{\mathbb{R}^n} \bigg( \frac{1}{1+|\xi|^2 }\bigg)^4 \hat w(\xi)^2\, d\xi .
\end{equation}
Furthermore, by Proposition~\ref{lem:orth1}, we have
\begin{align}\label{ineq34}
& \int_{\mathbb{R}^n} \bigg( \frac{1}{1+|\xi|^2 }\bigg)^{\frac{n+4}{2}} \hat w(\xi)\, d\xi =0, \nonumber\\
& \int_{\mathbb{R}^n} \bigg( \frac{1}{1+|\xi|^2 }\bigg)^{\frac{n+4}{2}}
\frac{1 - |\xi|^2}{1+|\xi|^2}\, \hat w(\xi)\, d\xi = 0, \nonumber\\
& \int_{\mathbb{R}^n} \bigg( \frac{1}{1+|\xi|^2 }\bigg)^{\frac{n+4}{2}}
\frac{\xi}{1+|\xi|^2}\,\hat w(\xi)\, d\xi = 0 .
\end{align}

We may equivalently work on the unit sphere $\mathbb{S}^n \subset \mathbb{R}^{n+1}$ via stereographic projection from the north pole. 
Under this identification, define
\[
\hat v(\xi)
:= \bigg( \frac{2}{1+|\xi|^2} \bigg)^{-\frac{n-4}{2}} \hat w(\xi).
\]

Note that
\[
g_{\mathbb{S}^n}=\bigg(\frac{2}{1+|\xi|^2}\bigg)^{2}|d\xi|^2 .
\]
Moreover, since $P_{g_{\mathbb{S}^n}}$ is a polynomial in $\Delta_{g_{\mathbb{S}^n}}$, its eigenfunctions coincide with those of $\Delta_{g_{\mathbb{S}^n}}$.

Using the conformal covariance of the Paneitz operator, the inequality \eqref{ineq33} is equivalent to
\[
\int_{\mathbb{S}^n} \hat v\, P_{g_{\mathbb{S}^n}} \hat v \, \vol_{g_{\mathbb{S}^n}}
\le \frac{n(n^2-4)(n+4)}{16}\int_{\mathbb{S}^n} \hat v^{2}\, \vol_{g_{\mathbb{S}^n}} .
\]

The identities \eqref{ineq34} imply that $\hat v$ is $L^2(\mathbb{S}^n)$-orthogonal to
$1$ and to the coordinate functions $x_1,\dots,x_{n+1}$ on $\mathbb{R}^{n+1}$, which correspond to the first and second eigenspaces of $P_{g_{\mathbb{S}^n}}$.
It is well known that
\[
P_{g_{\mathbb{S}^n}}
= \Delta_{g_{\mathbb{S}^n}}^{2}
- \frac{n^{2}-2n-4}{2}\,\Delta_{g_{\mathbb{S}^n}}
+ \frac{n(n^{2}-4)(n-4)}{16},
\]
and that the next eigenvalue of $\Delta_{g_{\mathbb{S}^n}}$ after $0$ and $n$ is $2(n+1)$.
A straightforward computation then shows that, under the orthogonality conditions \eqref{ineq34},
\[
\int_{\mathbb{S}^n} \hat v\, P_{g_{\mathbb{S}^n}} \hat v \, \vol_{g_{\mathbb{S}^n}}
>
\frac{n(n^2-4)(n+4)}{16}\int_{\mathbb{S}^n} \hat v^{2}\, \vol_{g_{\mathbb{S}^n}},
\]
unless $\hat v\equiv 0$. This contradicts the previous inequality, and hence completes the argument.

\noindent\textbf{Case 2.}
Assume that
\begin{equation*}
\lim_{\nu\to\infty} \int_M u_\infty^{\frac{8}{n-4}} \tilde w_\nu^{2}\, \dvol > 0,
\end{equation*}
and that
\begin{align*}
\lim_{\nu\to\infty}
\int_{M\setminus \bigcup_{j=1}^m \Omega_{j,\nu}}
&\Big(
(\Delta_{g_0}\tilde w_\nu)^2
- \big(4A_{g_0}-(n-2)\sigma_1(A_{g_0})\big)(\nabla \tilde w_\nu,\nabla \tilde w_\nu)
+ \frac{n-4}{2} Q_{g_0}\tilde w_\nu^2
\Big)
\nonumber\\
\le\;&
\lim_{\nu\to\infty}
\frac{n+4}{n-4}\,\mu_\infty
\int_M
u_\infty^{\frac{8}{n-4}} \tilde w_\nu^{2}\, \dvol .
\end{align*}

Let $\tilde w$ denote the weak limit of the sequence $\{\tilde w_\nu\}$. Then
\[
\int_M u_\infty^{\frac{8}{n-4}} \tilde w^{2}\, \dvol > 0,
\]
and 
\[
\int_M \tilde w\, \pa \tilde w \, \dvol
\le
\frac{n+4}{n-4}\,\mu_\infty
\int_M u_\infty^{\frac{8}{n-4}} \tilde w^{2}\, \dvol .
\]
by the dominated convergence theorem. 
Equivalently,
\[
\sum_a \lambda_a
\bigg( \int_M u_\infty^{\frac{8}{n-4}} \psi_a \tilde w \, \dvol \bigg)^2
\le
\frac{n+4}{n-4}\,\mu_\infty
\int_M u_\infty^{\frac{8}{n-4}} \tilde w^{2}\, \dvol .
\]
However, by Proposition~\ref{lem:orth1}, we have
\[
\int_M u_\infty^{\frac{8}{n-4}} \psi_a \tilde w \, \dvol = 0
\qquad \text{for all } a\in A.
\]
It follows that $\tilde w \equiv 0$, which contradicts
\(
\int_M u_\infty^{\frac{8}{n-4}} \tilde w^{2}\, \dvol > 0.
\)

Finally, we note that the argument in \textbf{Case 1} applies verbatim to the case
$u_\infty \equiv 0$, and thus completes the proof in that setting as well.

\end{proof}

\begin{corollary}\label{lem:secondv2}
For $\nu$ sufficiently large, we have
\begin{equation*}
\frac{n+4}{n-4}\,\mu_\infty \int_M v_\nu^{\frac{8}{n-4}} w_\nu^2 \, \dvol
\le (1-c)\, \| w_\nu \|_{\wtt}^{2},
\end{equation*}
for some constant $c>0$ independent of $\nu$.
\end{corollary}

\begin{proof}
By definition of $v_\nu$ and Lemma \ref{lem:interaction}, we have
$$\int_M \big| v_\nu^{\frac{8}{n-4}} - \wgt - \sum_{j=1}^m \bar{u}_{(p_{k, \nu}, \epsilon_{k, \nu})}^{\frac{8}{n-4}} \big|^{\frac{n}{4}} \dvol = o(1).$$

The assertion follows from Proposition \ref{lem:secondv} and Paneitz--Sobolev inequality:
\begin{align*}
&\frac{n+4}{n-4} \mu_\infty \int_M v_\nu^{\frac{8}{n-4}} w_\nu^2 \\
\le& \frac{n+4}{n-4} \mu_\infty \int_M \big| v_\nu^{\frac{8}{n-4}} - \sum_{j=1}^m \bar{u}_{(p_{k, \nu}, \epsilon_{k, \nu})}^{\frac{8}{n-4}} \big|w_\nu^2  + (1-c)  \| w_\nu \|_{\wtt} \\
\le& \bigg| \int_M \big| v_\nu^{\frac{8}{n-4}} - \sum_{j=1}^m \bar{u}_{(p_{k, \nu}, \epsilon_{k, \nu})}^{\frac{8}{n-4}} \big|^{\frac{n}{4}} \bigg|^{\frac{4}{n}} \cdot \|w_\nu \|_{\wtt}^2  + (1-c)  \| w_\nu \|_{\wtt}^2. 
\end{align*}
\end{proof}

\subsection{Energy estimate}

In this subsection, we establish the energy estimate for $v_\nu$. We begin with a finite-dimensional \loja--Simon inequality.

\begin{lemma}\label{lem:finiteloja}
Assume that $u_{\infty} \not\equiv 0$. Then there exists a real number $0<\gamma<1$ such that
\begin{equation*}
\cf[\bar{u}_z] - \cf[u_\infty]
\le
C \sup_{a \in A}
\bigg|
\int_M \psi_a \big( \pa \bar{u}_z - \mu_\infty \bar{u}_z^{\frac{n+4}{n-4}} \big)
\bigg|^{1+\gamma}
\end{equation*}
provided that $z$ is sufficiently small.
\end{lemma}

\begin{proof}
In a neighborhood of $u_\infty$, the Paneitz--Sobolev quotient is real analytic. Therefore, the \L ojasiewicz inequality applies to the Paneitz--Sobolev quotient restricted to the finite-dimensional space spanned by $\bar{u}_z$, which yields the desired estimate.
\end{proof}

We now estimate the difference $u_\nu-\bar{u}_{z_\nu}$. The key point is that it is more convenient to estimate
$u_\nu-\bar{u}_{z_\nu}+f_\nu$ rather than $u_\nu-\bar{u}_{z_\nu}$.

\begin{lemma}\label{lem:proj}
We have the following identity:
\begin{align*}
&\Pi \bigg( \pa (u_\nu - \bar{u}_{z_\nu} + f_\nu) -\frac{n+4}{n-4} \mu_\infty \wgt (u_\nu - \bar{u}_{z_\nu} +f_\nu ) \bigg)\nonumber \\  
= & \Pi \bigg( -\frac{n+4}{n-4}\mu_\infty \wgt f_\nu -(\mu_\infty -\mu(t_\nu))u_\nu^{\frac{n+4}{n-4}} \nonumber\\
&+\frac{n+4}{n-4} \mu_\infty (\bar{u}_{z_\nu}^{\frac{8}{n-4}} -u_\infty^{\frac{8}{n-4}} ) ( u_\nu- \bar{u}_{z_\nu} ) \nonumber \\
& - \mu_\infty\big(\bar{u}_{z_\nu}^{\frac{n+4}{n-4}} + \frac{n+4}{n-4} \bar{u}_{z_\nu}^{\frac{8}{n-4}} ( u_\nu - \bar{u}_{z_\nu} ) -u_\nu^{\frac{n+4}{n-4}} \big) \bigg).
\end{align*}
\end{lemma}

\begin{proof}
Using the identities
$$- \pa u_\nu +\mu_\infty u_\nu^{\frac{n+4}{n-4}} = \pa f_\nu + (\mu_\infty - \mu(t_\nu) ) u_\nu^{\frac{n+4}{n-4}}$$
and 
$$\Pi \big( \pa \bar{u}_{z_\nu} - \mu_\infty \bar{u}_{z_\nu}^{\frac{n+4}{n-4}} \big) = 0,$$
we obtain
\begin{align*}
&\Pi \bigg( \pa (u_\nu - \bar{u}_{z_\nu}) -\frac{n+4}{n-4} \mu_\infty \wgt (u_\nu - \bar{u}_{z_\nu}) \bigg) \\  
= & \Pi \bigg( -\pa f -(\mu_\infty -\mu(t_\nu))u_\nu^{\frac{n+4}{n-4}} +\frac{n+4}{n-4} \mu_\infty (\bar{u}_{z_\nu}^{\frac{8}{n-4}} -u_\infty^{\frac{8}{n-4}} ) ( u_\nu- \bar{u}_{z_\nu} ) \\
& - \mu_\infty\big(\bar{u}_{z_\nu}^{\frac{n+4}{n-4}} + \frac{n+4}{n-4} \bar{u}_{z_\nu}^{\frac{8}{n-4}} ( u_\nu - \bar{u}_{z_\nu} ) -u_\nu^{\frac{n+4}{n-4}} \big) \bigg).
\end{align*}
Adding $\Pi\big(\pa f_\nu - \frac{n+4}{n-4}\mu_\infty \wgt f_\nu\big)$ to both sides yields the stated identity.
\end{proof}

\begin{lemma}\label{lem:estimate1}
The difference $u_\nu - \bar{u}_{z_\nu}$ satisfies the estimate
\begin{equation*}
\| u_\nu - \bar{u}_{z_\nu} \|^{\frac{n+4}{n-4}}_{L^{\frac{n+4}{n-4}}(M)} \le  C \bigg( \| f_\nu\|_{\wtt}^{\frac{n+4}{n-4}} + (\mu(t_\nu) -\mu_\infty )^{\frac{n+4}{n-4}} + \sum_{k=1}^m \epsilon_{k, \nu}^{\frac{n-4}{2}} \bigg).
\end{equation*}
\end{lemma}

\begin{proof}
By Lemmas~\ref{lem:proj} and~\ref{lem:lpestimate}, we estimate
\begin{align*}
& \|u_\nu - \uznu +f_\nu \|_{L^{\frac{n+4}{n-4}}(M) } \\
\le & C \sup_{a \in A} \bigg| \int_M \wgt \psi_a (u_\nu - \uznu + f_\nu ) \dvol \bigg| \\ 
& + C \| \wgt f_\nu \|_{L^{\frac{n(n+4)}{n^2+16}}(M)} + C \| (\mu_\infty -\mu(t_\nu))u_\nu^{\frac{n+4}{n-4}} \|_{L^{\frac{n(n+4)}{n^2+16}}(M)}  \\
& + C \|(\bar{u}_{z_\nu}^{\frac{8}{n-4}} -u_\infty^{\frac{8}{n-4}} ) ( u_\nu- \bar{u}_{z_\nu} )  \|_{L^{\frac{n(n+4)}{n^2+16}}(M)} \\
& + C \bigg\| \bar{u}_{z_\nu}^{\frac{n+4}{n-4}} + \frac{n+4}{n-4} \bar{u}_{z_\nu}^{\frac{8}{n-4}} ( u_\nu - \bar{u}_{z_\nu} ) -u_\nu^{\frac{n+4}{n-4}} \bigg\|_{L^{\frac{n(n+4)}{n^2+16}}(M)} \\
\le &  C \sup_{a \in A} \bigg| \int_M \wgt \psi_a (u_\nu - \uznu  ) \dvol \bigg| + C \| f_\nu \|_{\wtt} + C (\mu(t_\nu) - \mu_\infty) \\
& + C \bigg\| \bar{u}_{z_\nu}^{\frac{n+4}{n-4}} + \frac{n+4}{n-4} \bar{u}_{z_\nu}^{\frac{8}{n-4}} ( u_\nu - \bar{u}_{z_\nu} ) -u_\nu^{\frac{n+4}{n-4}} \bigg\|_{L^{\frac{n(n+4)}{n^2+16}}(M)}
\end{align*}
where we used that $\frac{n(n+4)}{n^2+16}< \frac{n+4}{n-4}$ and that $\| \uznu - u_\infty\|_{C^\infty (M)}$ is small.

Using the pointwise estimate
\begin{align*}
&\bigg| \bar{u}_{z_\nu}^{\frac{n+4}{n-4}} + \frac{n+4}{n-4} \bar{u}_{z_\nu}^{\frac{8}{n-4}} ( u_\nu - \bar{u}_{z_\nu} ) -u_\nu^{\frac{n+4}{n-4}} \bigg| \\
\le& C \uznu^{\max\{0, \frac{8}{n-4} -1\}} |u_\nu - \uznu |^{\min\{\frac{n+4}{n-4}, 2\}} + C |u_\nu -\uznu |^{\frac{n+4}{n-4}},
\end{align*}
we obtain
\begin{align*}
& \bigg\| \bar{u}_{z_\nu}^{\frac{n+4}{n-4}} + \frac{n+4}{n-4} \bar{u}_{z_\nu}^{\frac{8}{n-4}} ( u_\nu - \bar{u}_{z_\nu} ) -u_\nu^{\frac{n+4}{n-4}} \bigg\|_{L^{\frac{n(n+4)}{n^2+16}}(M)} \\
\le & C\bigg\| |u_\nu - \uznu |^{\min\{\frac{n+4}{n-4}, 2\}} + |u_\nu -\uznu |^{\frac{n+4}{n-4}} \bigg\|_{L^{\frac{n(n+4)}{n^2+16}}(M)} \\
\le & C \bigg\| |u_\nu - \uznu |^{\min\{\frac{n+4}{n-4}, 2\}} + |u_\nu -\uznu |^{\frac{n+4}{n-4}} \bigg\|_{L^{\frac{n(n+4)}{n^2+16}}(\cup_{k=1}^m B_{N\epsilon_{k, \nu}}(p_{k, \nu}))} \\
& + C\bigg\| |u_\nu - \uznu |^{\min\{\frac{n+4}{n-4}, 2\}} + |u_\nu -\uznu |^{\frac{n+4}{n-4}} \bigg\|_{L^{\frac{n(n+4)}{n^2+16}}(M\setminus \cup_{k=1}^m B_{N\epsilon_{k, \nu}}(p_{k, \nu}))} \\
& \le C \sum_{k=1}^m (N\epsilon_{k, \nu})^{\frac{(n-4)^2}{2(n+4)}} \bigg\| |u_\nu - \uznu |^{\min\{\frac{n+4}{n-4}, 2\}} + |u_\nu -\uznu |^{\frac{n+4}{n-4}} \bigg\|_{L^{\frac{2n}{n+4 }}(M)} \\
& + C\bigg\| |u_\nu - \uznu |^{\min\{\frac{8}{n-4}, 1\}} + |u_\nu -\uznu |^{\frac{8}{n-4}} \bigg\|_{L^{\frac{n}{4}}(M\setminus \cup_{k=1}^m B_{N\epsilon_{k, \nu}}(p_{k, \nu}))} \\
& \cdot  \| u_\nu - \uznu \|_{L^{\frac{n+4}{n-4}}(M\setminus \cup_{k=1}^m B_{N\epsilon_{k, \nu}}(p_{k, \nu}))}
\end{align*}
by H\"older's inequality.
Since 
\begin{align*}
&\|u_\nu -\uznu \|_{L^{\frac{2n}{n-4}}(M\setminus \cup_{k=1}^m B_{N\epsilon_{k, \nu}}(p_{k, \nu}))} \\
\le & \sum_{k=1}^m \alpha_{k, \nu} \| \bar{u}_{(p_{k, \nu}, \epsilon_{k, \nu})} \|_ {L^{\frac{2n}{n-4}}(M\setminus \cup_{k=1}^m B_{N\epsilon_{k, \nu}}(p_{k, \nu}))} + \|w_\nu \|_{L^{\frac{2n}{n-4}}(M)} \\
\le & CN^{-\frac{n-4}{2}} + o(1),
\end{align*}
it follows that 
\begin{align*}
&\bigg\| \bar{u}_{z_\nu}^{\frac{n+4}{n-4}} + \frac{n+4}{n-4} \bar{u}_{z_\nu}^{\frac{8}{n-4}} ( u_\nu - \bar{u}_{z_\nu} ) -u_\nu^{\frac{n+4}{n-4}} \bigg\|_{L^{\frac{n(n+4)}{n^2+16}}(M)} \\
\le& C\sum_{k=1}^m (N\epsilon_{k, \nu})^{\frac{(n-4)^2}{2(n+4)}} + (CN^{-2} +o(1) ) \|u_\nu - \uznu \|_{L^{\frac{n+4}{n-4}}(M)}.
\end{align*}

Finally, we estimate 
\begin{align*}
& \sup_{a \in A} \bigg| \int_M \wgt \psi_a (u_\nu - \uznu  ) \dvol \bigg| \\
= & \sup_{a \in A} \bigg| \int_M \wgt \psi_a \bigg(\sum_{k=1}^m \alpha_{k, \nu} \bar{u}_{(p_{k, \nu}, \epsilon_{k, \nu})} + w_\nu \bigg) \dvol \bigg| \\
\le & C \sum_{k=1}^m \epsilon_{k, \nu}^{\frac{n-4}{2}} + o(1) \| w_\nu \|_{L^1 (M)} \\
\le & C \sum_{k=1}^m \epsilon_{k, \nu}^{\frac{n-4}{2}} + o(1) \|u_\nu - \uznu  -\sum_{k=1}^m \alpha_{k, \nu} \bar{u}_{(p_{k, \nu}, \epsilon_{k, \nu})}  \|_{L^1 (M)} \\
\le & C \sum_{k=1}^m \epsilon_{k, \nu}^{\frac{n-4}{2}} + o(1) \|u_\nu - \uznu    \|_{L^1 (M)}.
\end{align*}

Collecting the above estimates, we conclude that
\begin{align*}
&\|u_\nu -\uznu +f \|_{L^{\frac{n+4}{n-4}}(M)} \\
\le&  C \| f_\nu \|_{\wtt} + C (\mu(t_\nu) - \mu_\infty) +C \sum_{k=1}^m \epsilon_{k, \nu}^{\frac{n-4}{2}}  \\
& +C\sum_{k=1}^m (N\epsilon_{k, \nu})^{\frac{(n-4)^2}{2(n+4)}} + (CN^{-2} +o(1) ) \|u_\nu - \uznu \|_{L^{\frac{n+4}{n-4}}(M)}.
\end{align*}
Choosing $N$ sufficiently large and applying the triangle inequality, we obtain the desired estimate. Indeed,
$$  \|u_\nu -\uznu  \|_{L^{\frac{n+4}{n-4}}(M)} \le \|u_\nu -\uznu  +f \|_{L^{\frac{n+4}{n-4}}(M)} + \| f\|_{\wtt}. $$

\end{proof}

\begin{lemma} \label{lem:estimate2}
The difference $u_\nu -\bar{u}_{z_\nu}$ satisfies the estimate
\begin{equation*}
\|u_z - \bar{u}_{z_\nu}\|_{L^1(M)} \le C ( \| f\|_{\wtt} + \mu(t_\nu) - \mu_\infty) + C \sum_{k=1}^m \epsilon_{k, \nu}^{\frac{n-4}{2}} 
\end{equation*}
for all sufficiently large $\nu$.
\end{lemma}

\begin{proof}
The proof is analogous to that of Lemma~\ref{lem:estimate1}.

\end{proof}

\begin{lemma}\label{lem:lowereigen}
For $\nu$ sufficiently large, we have
\begin{equation*}
\sup_{a \in A} \bigg| \int_M \psi_a \big( \pa \bar{u}_{z_\nu} - \mu_\infty \bar{u}_{z_\nu}^{\frac{n+4}{n-4}} \big) \dvol \big| \le C ( \| f \|_{\wtt} + \mu(t_\nu) - \mu_\infty )+ C \sum_{k=1}^m \epsilon_{k, \nu}^{\frac{n-4}{2}}.
\end{equation*}

\end{lemma}

\begin{proof}
By integration by parts,
\begin{align*}
& \int_M \psi_a \big( \pa  \bar{u}_{z_\nu} - \mu_\infty \bar{u}_{z_\nu}^{\frac{n+4}{n-4}} \big)\\
=&    \int_M \psi_a \big( \pa u_\nu - \mu_\infty u_\nu^{\frac{n+4}{n-4}} \big) + \lambda_a \int_M \wgt \psi_a (\bar{u}_{z_\nu} - u_\nu ) \\
&+ \mu_\infty \int_M \psi_a ( u_\nu^{\frac{n+4}{n-4}} - \bar{u}_{z_\nu}^{\frac{n+4}{n-4}} ) \\
=& \int_M \psi_a \big( - \pa f - ( \mu_\infty - \mu(t_\nu)) u_\nu^{\frac{n+4}{n-4}} \big) + \lambda_a \int_M \wgt \psi_a (\bar{u}_{z_\nu} - u_\nu ) \\
& + \mu_\infty \int_M \psi_a ( u_\nu^{\frac{n+4}{n-4}} - \bar{u}_{z_\nu}^{\frac{n+4}{n-4}} ).
\end{align*}

Using the pointwise estimate
$$|u_\nu^{\frac{n+4}{n-4}} - \bar{u}_{z_\nu}^{\frac{n+4}{n-4}} | \le C \bar{u}_{z_\nu}^{\frac{8}{n-4}} |u_z - \bar{u}_{z_\nu}| + C |u_z - \bar{u}_{z_\nu}|^{\frac{n+4}{n-4}},$$
we prove that 
\begin{align*}
& \bigg| \int_M \psi_a \big( \pa  \bar{u}_{z_\nu} - \mu_\infty \bar{u}_{z_\nu}^{\frac{n+4}{n-4}} \big) \bigg| \\
\le& C \bigg( \| f\|_{\wtt}  + (\mu(t_\nu) - \mu_\infty) + \|\bar{u}_{z_\nu} -u_z \|_{L^1(M)} + \| \bar{u}_{z_\nu} - u_z \|_{L^{\frac{n+4}{n-4}}(M)}^{\frac{n+4}{n-4}}\bigg) .
\end{align*}

The desired estimate now follows from Lemmas~\ref{lem:estimate1} and~\ref{lem:estimate2}.

\end{proof}

\begin{proposition}\label{lem:uznuenergy}
For $\nu$ sufficiently large, the Paneitz--Sobolev quotient of $\uznu$ satisfies
\begin{equation*}
\cf[\uznu] - \cf[u_\infty] \le C \bigg( \| f_\nu \|_{\wtt}^{1+\gamma} + (\mu(t_\nu) - \mu_\infty)^{1+\gamma} +   \sum_{k=1}^m \epsilon_{k, \nu}^{\frac{n-4}{2}(1+\gamma)} \bigg).
\end{equation*}
\end{proposition}
\begin{proof}
This is an immediate consequence of Lemmas~\ref{lem:finiteloja} and~\ref{lem:lowereigen}.
\end{proof}

\begin{proposition}\label{lem:venergy}
The Paneitz--Sobolev quotient of $v_\nu$ satisfies the estimate
\begin{equation*}
\cf[v_\nu] \le \bigg( \cf[\uznu]^{\frac{n}{4}} + \sum_{k=1}^m \cf[\bar{u}_{(p_{k, \nu}, \epsilon_{k, \nu})}]^{\frac{n}{4}} \bigg)^{\frac{4}{n}} - c \sum_{k=1}^m \epsilon_{k, \nu}^{\frac{n-4}{2}}
\end{equation*}
for all sufficiently large $\nu$.
\end{proposition}

\begin{proof}
A direct computation yields
\begin{align*}
& \cf[v_\nu] \bigg( \int_M v_\nu^{\frac{2n}{n-4}} \bigg)^{\frac{n}{n-4}}  \\
= & \int_M \mu[\uznu]\uznu^{\frac{2n}{n-4}}  +\int_M \sum_{k=1}^m \alpha_{k, \nu}^2 \mu[\bar{u}_{(p_{k, \nu}, \epsilon_{k, \nu})}] \bar{u}_{(p_{k, \nu}, \epsilon_{k, \nu})}^{\frac{2n}{n-4}} \\
& +2 \int_M \sum_{k=1}^m \alpha_{k, \nu} \uznu \pa \bar{u}_{(p_{k, \nu}, \epsilon_{k, \nu})} + 2 \int_M \sum_{i< j} \alpha_{i, \nu} \alpha_{j, \nu} \bar{u}_{(p_{i, \nu}, \epsilon_{i, \nu})} \pa \bar{u}_{(p_{j, \nu}, \epsilon_{j, \nu})}.
\end{align*}

Moreover,
\begin{align*}
& \bigg( \cf[\uznu]^{\frac{n}{4}} +  \sum_{k=1}^m \cf[\bar{u}_{(p_{k, \nu}, \epsilon_{k, \nu})}]^{\frac{n}{4}} \bigg)^{\frac{4}{n}} \bigg( \int_M v_\nu^{\frac{2n}{n-4}} \dvol \bigg)^{\frac{n-4}{n}} \\
= & \bigg( \int_M \big( \mu[\uznu]^{\frac{n}{4}} \uznu^{\frac{2n}{n-4}} + \sum_{k=1}^m \mu[\bar{u}_{(p_{k, \nu}, \epsilon_{k, \nu})}]^{\frac{n}{4}} \bar{u}_{(p_{k, \nu}, \epsilon_{k, \nu})}^{\frac{2n}{n-4}} \big)\dvol \bigg)^\frac{4}{n} \cdot \bigg( \int_M v_\nu^{\frac{2n}{n-4}} \dvol \bigg)^{\frac{n-4}{n}} \\
\ge &  \int_M \bigg( \mu[\uznu]^{\frac{n}{4}} \uznu^{\frac{2n}{n-4}} + \sum_{k=1}^m \mu[\bar{u}_{(p_{k, \nu}, \epsilon_{k, \nu})}]^\frac{n}{4} \bar{u}_{(p_{k, \nu}, \epsilon_{k, \nu})}^{\frac{2n}{n-4}} \bigg)^{\frac{4}{n}} v_\nu^2 \dvol \\
\ge&  \int_M  \mu[\uznu] \uznu^{\frac{2n}{n-4}} + \sum_{k=1}^m \alpha_{k, \nu}^2 \mu[\bar{u}_{(p_{k, \nu}, \epsilon_{k, \nu})}] \bar{u}_{(p_{k, \nu}, \epsilon_{k, \nu})}^{\frac{2n}{n-4}}  \\
& + 2 \int_M \sum_{k=1}^m \alpha_{k, \nu} \bigg(\mu[\uznu]^{\frac{n}{4}} \uznu^{\frac{2n}{n-4}} + \mu[\bar{u}_{(p_{k, \nu}, \epsilon_{k, \nu})}]^{\frac{n}{4}} \bar{u}_{(p_{k, \nu}, \epsilon_{k, \nu})}^{\frac{2n}{n-4}} \bigg)^{\frac{4}{n}} \uznu \bar{u}_{(p_{k, \nu}, \epsilon_{k, \nu})} \\
& + 2 \int_M \sum_{i<j} \alpha_{i, \nu} \alpha_{j, \nu} \bigg( \mu[\bar{u}_{(p_{i, \nu}, \epsilon_{i, \nu})}]^{\frac{n}{4}} \bar{u}_{(p_{i, \nu}, \epsilon_{i, \nu})}^{\frac{2n}{n-4}} + \mu[\bar{u}_{(p_{j, \nu}, \epsilon_{j, \nu})}]^{\frac{n}{4}} \bar{u}_{(p_{j, \nu}, \epsilon_{j, \nu})}^{\frac{2n}{n-4}} \bigg)^{\frac{4}{n}} \bar{u}_{(p_{i, \nu}, \epsilon_{i, \nu})} \bar{u}_{(p_{j, \nu}, \epsilon_{j, \nu})} 
\end{align*}
where we used H\"older's inequality.

Fix $1 \le k \le m$. Using the inequality
\begin{align*}
& \bigg(\mu[\uznu]^{\frac{n}{4}} \uznu^{\frac{2n}{n-4}} + \mu[\bar{u}_{(p_{k, \nu}, \epsilon_{k, \nu})}]^{\frac{n}{4}} \bar{u}_{(p_{k, \nu}, \epsilon_{k, \nu})}^{\frac{2n}{n-4}} \bigg)^{\frac{4}{n}} \uznu \bar{u}_{(p_{k, \nu}, \epsilon_{k, \nu})} \\
& \ge \mu[\uznu] \uznu^{\frac{n+4}{n-4}}  \bar{u}_{(p_{k, \nu}, \epsilon_{k, \nu})} + c\epsilon_{k, \nu}^{-\frac{n+4}{2}} \1_{\{d(p_{k, \nu}, x) \le \epsilon_{k, \nu}\}},
\end{align*}
we obtain
\begin{align*}
& \int_M \bigg(\mu[\uznu]^{\frac{n}{4}} \uznu^{\frac{2n}{n-4}} + \mu[\bar{u}_{(p_{k, \nu}, \epsilon_{k, \nu})}]^{\frac{n}{4}} \bar{u}_{(p_{k, \nu}, \epsilon_{k, \nu})}^{\frac{2n}{n-4}} \bigg)^{\frac{4}{n}} \uznu \bar{u}_{(p_{k, \nu}, \epsilon_{k, \nu})} \\
& \ge \int_M \mu[\uznu] \uznu^{\frac{n+4}{n-4}}  \bar{u}_{(p_{k, \nu}, \epsilon_{k, \nu})} \dvol + c\epsilon_{k, \nu}^{\frac{n-4}{2}} 
\end{align*}
for $\nu$ sufficiently large.

Consider a pair $i<j$. We can find positive constants $c$ and $C$, independent of $\nu$, such that
$$\bar{u}_{(p_{i, \nu}, \epsilon_{i, \nu})} (x)^{\frac{n+4}{n-4}} \bar{u}_{(p_{j, \nu}, \epsilon_{j, \nu})} (x) \ge c \bigg(\frac{\epsilon_{j, \nu}^2 + d(p_{i, \nu}, p_{j, \nu} )^2}{\epsilon_{i, \nu} \epsilon_{j, \nu}} \bigg)^{-\frac{n-4}{2}} \epsilon_{i, \nu}^{-n}$$
and
$$\bar{u}_{(p_{i, \nu}, \epsilon_{i, \nu})} (x) \bar{u}_{(p_{j, \nu}, \epsilon_{j, \nu})} (x)^{\frac{n+4}{n-4}} \le C \bigg(\frac{\epsilon_{j, \nu}^2 + d(p_{i, \nu}, p_{j, \nu} )^2}{\epsilon_{i, \nu} \epsilon_{j, \nu}} \bigg)^{-\frac{n-4}{2}} \epsilon_{i, \nu}^{-n}$$
whenever $d(p_{i, \nu}, x) \le \epsilon_{i, \nu}$ and $\nu$ is sufficiently large. It follows that
\begin{align*}
& \bigg( \mu[\bar{u}_{(p_{i, \nu}, \epsilon_{i, \nu})}]^{\frac{n}{4}} \bar{u}_{(p_{i, \nu}, \epsilon_{i, \nu})}^{\frac{2n}{n-4}} + \mu[\bar{u}_{(p_{j, \nu}, \epsilon_{j, \nu})}]^{\frac{n}{4}} \bar{u}_{(p_{j, \nu}, \epsilon_{j, \nu})}^{\frac{2n}{n-4}} \bigg)^{\frac{4}{n}} \bar{u}_{(p_{i, \nu}, \epsilon_{i, \nu})} \bar{u}_{(p_{j, \nu}, \epsilon_{j, \nu})}\\
\ge & \mu[\bar{u}_{(p_{j, \nu}, \epsilon_{j, \nu})}] \bar{u}_{(p_{i, \nu}, \epsilon_{i, \nu})} \bar{u}_{(p_{j, \nu}, \epsilon_{j, \nu})}^{\frac{n+4}{n-4}} + c \bigg(\frac{\epsilon_{j, \nu}^2 + d(p_{i, \nu}, p_{j, \nu} )^2}{\epsilon_{i, \nu} \epsilon_{j, \nu}} \bigg)^{-\frac{n-4}{2}} \epsilon_{i, \nu}^{-n} \1_{\{d(p_{i\nu}, x ) \le \epsilon_{i, \nu} \}}
\end{align*}
and
\begin{align*}
& \int_M \bigg( \mu[\bar{u}_{(p_{i, \nu}, \epsilon_{i, \nu})}]^{\frac{n}{4}} \bar{u}_{(p_{i, \nu}, \epsilon_{i, \nu})}^{\frac{2n}{n-4}} + \mu[\bar{u}_{(p_{j, \nu}, \epsilon_{j, \nu})}]^{\frac{n}{4}} \bar{u}_{(p_{j, \nu}, \epsilon_{j, \nu})}^{\frac{2n}{n-4}} \bigg)^{\frac{4}{n}} \bar{u}_{(p_{i, \nu}, \epsilon_{i, \nu})} \bar{u}_{(p_{j, \nu}, \epsilon_{j, \nu})}\\
\ge & \int_M \mu[\bar{u}_{(p_{j, \nu}, \epsilon_{j, \nu})}] \bar{u}_{(p_{i, \nu}, \epsilon_{i, \nu})} \bar{u}_{(p_{j, \nu}, \epsilon_{j, \nu})}^{\frac{n+4}{n-4}} \dvol + c \bigg(\frac{\epsilon_{j, \nu}^2 + d(p_{i, \nu}, p_{j, \nu} )^2}{\epsilon_{i, \nu} \epsilon_{j, \nu}} \bigg)^{-\frac{n-4}{2}} 
\end{align*}
for $\nu$ sufficiently large.

Summing the inequalities above, we obtain
\begin{align*}
& \bigg( \cf[\uznu]^{\frac{n}{4}} +  \sum_{k=1}^m \cf[\bar{u}_{(p_{k, \nu}, \epsilon_{k, \nu})}]^{\frac{n}{4}} \bigg)^{\frac{4}{n}} \bigg( \int_M v_\nu^{\frac{2n}{n-4}} \dvol \bigg)^{\frac{n-4}{n}} \\
\ge&  \int_M  \bigg( \mu[\uznu] \uznu^{\frac{2n}{n-4}} + \sum_{k=1}^m \alpha_{k, \nu}^2 \mu[\bar{u}_{(p_{k, \nu}, \epsilon_{k, \nu})}] \bar{u}_{(p_{k, \nu}, \epsilon_{k, \nu})}^{\frac{2n}{n-4}} + 2 \alpha_{k, \nu} \mu[\uznu] \uznu^{\frac{n+4}{n-4}} \bar{u}_{(p_{k, \nu}, \epsilon_{k, \nu})} \\
& + 2  \sum_{i<j} \alpha_{i, \nu} \alpha_{j, \nu} \mu[\bar{u}_{(p_{j, \nu}, \epsilon_{j, \nu})}] \bar{u}_{(p_{i, \nu}, \epsilon_{i, \nu})}\bar{u}_{(p_{j, \nu}, \epsilon_{j, \nu})}^{\frac{n+4}{n-4}} \bigg) \dvol\\
& + c \sum_{k=1}^m \epsilon_{k, \nu}^{\frac{n-4}{2}} + c \sum_{i<j} \bigg(\frac{\epsilon_{j, \nu}^2 + d(p_{i, \nu}, p_{j, \nu} )^2}{\epsilon_{i, \nu} \epsilon_{j, \nu}} \bigg)^{-\frac{n-4}{2}}.
\end{align*}

Hence,
\begin{align*}
&\cf[v_\nu] \bigg( \int_M v_\nu^{\frac{2n}{n-4}} \bigg)^{\frac{n}{n-4}}  \\
\le & \bigg( \cf[\uznu]^{\frac{n}{4}} +  \sum_{k=1}^m \cf[\bar{u}_{(p_{k, \nu}, \epsilon_{k, \nu})}]^{\frac{n}{4}}  \bigg)^{\frac{4}{n}} \bigg( \int_M v_\nu^{\frac{2n}{n-4}} \dvol \bigg)^{\frac{n-4}{n}} \\
& + 2 \int_M \sum_{k=1}^m \alpha_{k, \nu}\bigg( \pa \uznu - \mu[\uznu] \uznu^{\frac{n+4}{n-4}} \bigg)\bar{u}_{(p_{k, \nu}, \epsilon_{k, \nu})} \\
& +2 \int_M \sum_{i<j} \alpha_{i, \nu} \alpha_{j, \nu} \bar{u}_{(p_{i, \nu}, \epsilon_{i, \nu})} \bigg(\pa \bar{u}_{(p_{j, \nu}, \epsilon_{j, \nu})} - \mu[\bar{u}_{(p_{j, \nu}, \epsilon_{j, \nu})}]  \bar{u}_{(p_{j, \nu}, \epsilon_{j, \nu})}^{\frac{n+4}{n-4}}  \bigg)\dvol \\
& - c \sum_{k=1}^m \epsilon_{k, \nu}^{\frac{n-4}{2}} - c \sum_{i<j} \bigg(\frac{\epsilon_{j, \nu}^2 + d(p_{i, \nu}, p_{j, \nu} )^2}{\epsilon_{i, \nu} \epsilon_{j, \nu}} \bigg)^{-\frac{n-4}{2}}.
\end{align*}

We note that
\begin{equation*}
\int_M \bigg| \pa \uznu - \mu[\uznu] \uznu^{\frac{n+4}{n-4}} \bigg| \bar{u}_{(p_{k, \nu}, \epsilon_{k, \nu})} \dvol \le o(1) \epsilon_{k, \nu}^{\frac{n-4}{2}}.
\end{equation*}

Since $\mu [ \bar{u}_{(p_{j, \nu}, \epsilon_{j, \nu})}] = \mu_\infty + o(1)$, it follows from Lemma~\ref{lem:inter2} that
\begin{align*}
&\int_M \sum_{i<j} \bar{u}_{(p_{i, \nu}, \epsilon_{i, \nu})} \bigg|\pa \bar{u}_{(p_{j, \nu}, \epsilon_{j, \nu})} - \mu[\bar{u}_{(p_{j, \nu}, \epsilon_{j, \nu})}]  \bar{u}_{(p_{j, \nu}, \epsilon_{j, \nu})}^{\frac{n+4}{n-4}}  \bigg|\dvol \\
\le & \int_M \sum_{i<j} \bar{u}_{(p_{i, \nu}, \epsilon_{i, \nu})} \bigg|\pa \bar{u}_{(p_{j, \nu}, \epsilon_{j, \nu})} - \mu_\infty  \bar{u}_{(p_{j, \nu}, \epsilon_{j, \nu})}^{\frac{n+4}{n-4}}  \bigg|\dvol \\
& + \big|\mu[\bar{u}_{(p_{j, \nu}, \epsilon_{j, \nu})}] - \mu_\infty\big| \int_m \bar{u}_{(p_{i, \nu}, \epsilon_{i, \nu})} \bar{u}_{(p_{j, \nu}, \epsilon_{j, \nu})}^{\frac{n+4}{n-4}}  \dvol \\
\le & C\left( \delta^4  + \delta^{n-4} + \frac{\epsilon_{j, \nu}^4}{\delta^4} \right) \bigg(\frac{\epsilon_{j, \nu}^2 + d(p_{i, \nu}, p_{j, \nu} )^2}{\epsilon_{i, \nu} \epsilon_{j, \nu}} \bigg)^{-\frac{n-4}{2}} + o(1) \bigg(\frac{\epsilon_{j, \nu}^2 + d(p_{i, \nu}, p_{j, \nu} )^2}{\epsilon_{i, \nu} \epsilon_{j, \nu}} \bigg)^{-\frac{n-4}{2}} 
\end{align*}
for $i<j$.
Choosing $\delta$ sufficiently small, and noting that the constant $c$ is independent of $\delta$, we obtain the desired estimate.
\end{proof}

\begin{corollary} \label{lem:nonzeroenergy}
For $\nu$ sufficiently large, the Paneitz--Sobolev constant of $v_\nu$ satisfies
\begin{equation*}
\cf[v_\nu] \le \bigg( \cf[u_\infty]^{\frac{n}{4}} + m Y_4(S^n)^{\frac{n}{4}} \bigg)^{\frac{4}{n}} +  C \bigg( \| f_\nu \|_{\wtt}^{1+\gamma} + (\mu(t_\nu) - \mu_\infty)^{1+\gamma}  \bigg).
\end{equation*}

\end{corollary}

\begin{proof}
By Propositions~\ref{lem:uznuenergy} and~\ref{lem:venergy}, we have
\begin{align*}
\cf[v_\nu] \le &  \bigg( \cf[\uznu]^{\frac{n}{4}} + \sum_{k=1}^m \cf[\bar{u}_{(p_{k, \nu}, \epsilon_{k, \nu})}]^{\frac{n}{4}} \bigg)^{\frac{4}{n}} - c \sum_{k=1}^m \epsilon_{k, \nu}^{\frac{n-4}{2}} \\
\le&   \bigg[  \bigg(\cf[u_\infty] + C\big( \| f_\nu \|_{\wtt}^{1+\gamma} + (\mu(t_\nu) - \mu_\infty)^{1+\gamma} +   \sum_{k=1}^m \epsilon_{k, \nu}^{\frac{n-4}{2}(1+\gamma)} \big)\bigg)^{\frac{n}{4}} \\
& + m Y_4(S^n)^{\frac{n}{4}} \bigg]^{\frac{4}{n}}  - c \sum_{k=1}^m \epsilon_{k, \nu}^{\frac{n-4}{2}}.
\end{align*}
Using the elementary algebraic inequalities
$$(x+y)^{\frac{n}{4}} \le x^{\frac{n}{4}} + C y x^{\frac{n}{4} -1}, \hspace{4mm} (x+y)^{\frac{4}{n}} \le x^{\frac{4}{n}} + C y x^{\frac{4}{n}-1}$$
which hold for $0<y\ll x$, we obtain the desired estimate.

We also note that in the case $u_\infty \equiv 0$, the same argument yields the stronger inequality
\begin{equation*}
\cf[v_\nu] \le \big( m Y_4(S^n)^{\frac{n}{4}} \big)^{\frac{4}{n}}.
\end{equation*}

\end{proof}

\section{Proof of Theorem \ref{thm:globalconv}}

\begin{lemma}\label{lem:algebraic}
Let $x \ge y > 0$ be real numbers. Then the following inequalities hold:
\begin{align}
\frac{1}{y} - \frac{1}{x}
\le\;&
\frac{n}{n-4}
\bigg(
\frac{1}{y^{\frac{n}{n-4}}}
-
\frac{1}{x^{\frac{n}{n-4}}}
\bigg)
\,y^{-\frac{4}{n}},
\label{eq:algebraic lemma 1}\\
x^{\frac{4}{n}} - y^{\frac{4}{n}}
\le\;&
\frac{4}{n-4}
\bigg(
\frac{1}{y^{\frac{n}{n-4}}}
-
\frac{1}{x^{\frac{n}{n-4}}}
\bigg)
\,x .
\label{eq:algebraic lemma 2}
\end{align}
\end{lemma}

\begin{proof}
\eqref{eq:algebraic lemma 1} is equivalent to
$$x-y \le \frac{n}{n-4} \bigg( x^{\frac{n-4}{n}} - y^{\frac{n-4}{n}} \bigg) x^{\frac{4}{n}}.$$
Normalizing the inequality by setting $y=1$, $ x= 1+\alpha$,
the inequality is equivalent to
$$(1+\alpha)^{\frac{4}{n}} \le 1+ \frac{4}{n} \alpha,$$
which holds by Bernoulli's inequality.  The proof of \eqref{eq:algebraic lemma 2} follows in the same way.
\end{proof}

\begin{proposition}\label{keyestnonzero}
Let $\{t_\nu\}$ be a sequence with $t_\nu \to \infty$. Then there exist a number $0<\gamma<1$ and a constant $C>0$ such that, after passing to a subsequence, we have
\begin{equation}\label{eq:keyestnonzero}
\mathcal{F}(t_\nu)-\mathcal{F}_\infty
\le
C\bigg(\int_M f(t_\nu)\,\pa f(t_\nu)\,\dvol\bigg)^{\frac{1+\gamma}{2}}.
\end{equation}
\end{proposition}

\begin{proof}
We compute:
\begin{align*}
&\cf (t_\nu) \bigg(\int_M u_\nu^{\frac{2n}{n-4}} \bigg)^{\frac{n-4}{n}} \\
=&  \int_M v_\nu \pa v_\nu + 2 w_\nu \pa u_\nu -  w_\nu \pa w_\nu \\
= & \int_M v_\nu \pa v_\nu +  2 w_\nu \big(-\pa f(t_\nu) + \mu(t_\nu) u_\nu^{\frac{n+4}{n-4}}\big) -  w_\nu \pa w_\nu \\
= & \cf [v_\nu] \bigg(\int_M v_\nu^{\frac{2n}{n-4}} \bigg)^{\frac{n-4}{n}} + 2 \int_M \ \mu(t_\nu) u_\nu^{\frac{n+4}{n-4}} w_\nu -  \frac{n+4}{n-4} \mu_\infty v_\nu^{\frac{8}{n-4}} w_\nu^2  \\
& -2 \int_M w_\nu\pa f(t_\nu)  -  \int_M w_\nu \bigg(  \pa w_\nu  -  \frac{n+4}{n-4} \mu_\infty v_\nu^{\frac{8}{n-4}} w_\nu \bigg).
\end{align*}

Next, we estimate
\begin{align*}
&  \cf [v_\nu] \bigg(\int_M v_\nu^{\frac{2n}{n-4}} \bigg)^{\frac{n-4}{n}} + 2 \int_M \ \mu(t_\nu) u_\nu^{\frac{n+4}{n-4}} w_\nu -  \frac{n+4}{n-4} \mu_\infty v_\nu^{\frac{8}{n-4}} w_\nu^2 \\
=&  \cf_\infty \bigg(\int_M u_\nu^{\frac{2n}{n-4}} \bigg)^{\frac{n-4}{n}}   + 2 (\mu(t_\nu) - \mu_\infty) \int_M u_\nu^{\frac{n+4}{n-4}} w_\nu \\
& + \mu_\infty \int_M 2 u_\nu^{\frac{n+4}{n-4}} w_\nu -  \frac{n+4}{n-4}  v_\nu^{\frac{8}{n-4}} w_\nu^2  + \frac{n-4}{n} v_\nu^{\frac{2n}{n-4}} - \frac{n-4}{n} u_\nu^\frac{2n}{n-4} \\
& + \bigg[ \cf [v_\nu]- \frac{n-4}{n} \mu_\infty \bigg(\int_M v_\nu^{\frac{2n}{n-4}} \bigg)^{\frac{4}{n}} \bigg] \bigg(\int_M v_\nu^{\frac{2n}{n-4}} \bigg)^{\frac{n-4}{n}} - \frac{4}{n} \mu_\infty \int_M u_\nu^{\frac{2n}{n-4}}  \\
& - \bigg[ \cf_\infty - \mu_\infty \bigg(\int_M u_\nu^{\frac{2n}{n-4}}  \bigg)^{\frac{4}{n}} \bigg] \bigg(\int_M u_\nu^{\frac{2n}{n-4}} \bigg)^{\frac{n-4}{n}}.
\end{align*}

Recalling that $\int_M u \pa u$ is preserved along the flow, by Lemma~\ref{lem:algebraic} we have
\begin{align*}
(\mu(t_\nu) - \mu_\infty) \int_M u_\nu^{\frac{n+4}{n-4}} w_\nu &= o(1)  \bigg(\frac{1}{V(t_\nu)}  - \frac{1}{V_\infty} \bigg) \int_M u \pa u \\
& \le o(1)  \bigg(\frac{1}{V(t_\nu)^{\frac{n-4}{n}}}  - \frac{1}{V_\infty^{\frac{n-4}{n}} }\bigg) \int_M u \pa u \\
& = o(1) ( \cf ( t_\nu) - \cf_\infty ).
\end{align*}

Moreover, by Young’s inequality,
\begin{align*}
& \bigg[ \cf [v_\nu]- \frac{n-4}{n} \mu_\infty \bigg(\int_M v_\nu^{\frac{2n}{n-4}} \bigg)^{\frac{4}{n}} \bigg] \bigg(\int_M v_\nu^{\frac{2n}{n-4}} \bigg)^{\frac{n-4}{n}} - \frac{4}{n} \mu_\infty \int_M u_\nu^{\frac{2n}{n-4}} \\
= & ( \cf [v_\nu]-  \cf_\infty  ) \bigg(\int_M v_\nu^{\frac{2n}{n-4}} \bigg)^{\frac{n-4}{n}} -  \mu_\infty \bigg[ \int_M  \frac{4}{n} u_\nu^{\frac{2n}{n-4}} + \frac{n-4}{n} v_\nu^{\frac{2n}{n-4}} - V_\infty^{\frac{4}{n}} \bigg(\int_M v_\nu^{\frac{2n}{n-4}} \bigg)^{\frac{n-4}{n}} \bigg] \\
\le & ( \cf [v_\nu]-  \cf_\infty  ) \bigg(\int_M v_\nu^{\frac{2n}{n-4}} \bigg)^{\frac{n-4}{n}} + \mu_\infty (V_\infty^{\frac{4}{n}} - V(t_\nu)^{\frac{4}{n}} ) \bigg(\int_M v_\nu^{\frac{2n}{n-4}} \bigg)^{\frac{n-4}{n}}.
\end{align*}

Then, again by Lemma~\ref{lem:algebraic}, we estimate
\begin{align*}
& \mu_\infty (V_\infty^{\frac{4}{n}} - V(t_\nu)^{\frac{4}{n}} ) \bigg(\int_M v_\nu^{\frac{2n}{n-4}} \bigg)^{\frac{n-4}{n}} - \bigg[ \cf_\infty - \mu_\infty \bigg(\int_M u_\nu^{\frac{2n}{n-4}}  \bigg)^{\frac{4}{n}} \bigg] \bigg(\int_M u_\nu^{\frac{2n}{n-4}} \bigg)^{\frac{n-4}{n}} \\
= & \mu_\infty (V_\infty^{\frac{4}{n}} - V(t_\nu)^{\frac{4}{n}} ) \bigg[ \bigg(\int_M v_\nu^{\frac{2n}{n-4}} \bigg)^{\frac{n-4}{n}} - \bigg(\int_M u_\nu^{\frac{2n}{n-4}} \bigg)^{\frac{n-4}{n}}\bigg] \\
\le & C (\cf(t_\nu) -\cf_\infty) \cdot \bigg| \bigg(\int_M u_\nu^{\frac{2n}{n-4}} \bigg) - \bigg(\int_M v_\nu^{\frac{2n}{n-4}} \bigg)\bigg| \cdot \min\bigg\{\int_M v_\nu^{\frac{2n}{n-4}}, \int_M u_\nu^{\frac{2n}{n-4}} \bigg\}^{-\frac{4}{n}} \\
\le & C (\cf(t_\nu) -\cf_\infty) \cdot \bigg|\int_M w_\nu v_\nu^{\frac{n+4}{n-4}} + w_\nu^2 v_\nu^{\frac{8}{n-4}} + |w_\nu|^{\frac{2n}{n-4}}  \bigg| \cdot \min\bigg\{\int_M v_\nu^{\frac{2n}{n-4}}, \int_M u_\nu^{\frac{2n}{n-4}} \bigg\}^{-\frac{4}{n}} \\
= & o(1) (\cf(t_\nu) -\cf_\infty).
\end{align*}

Using a purely algebraic inequality (which may be viewed as a Taylor expansion with respect to $w_\nu$),
\begin{align*}
&\bigg | 2 (v_\nu + w_\nu)^{\frac{n+4}{n-4}} w_\nu - \frac{n+4}{n-4}   v_\nu^{\frac{8}{n-4}} w_\nu^2  + \frac{n-4}{n} v_\nu^{\frac{2n}{n-4}} - \frac{n-4}{n} (v_\nu + w_\nu)^\frac{2n}{n-4} \bigg| \\
\le & C v_\nu ^{\max \{0, \frac{8}{n-4}-1 \}} |w_\nu|^{\min \{\frac{2n}{n-4}, 3 \}} + C |w_\nu | ^{\frac{2n}{n-4}},
\end{align*}
we obtain
\begin{align*}
\mu [v_\nu]  & \bigg(\int_M v_\nu^{\frac{2n}{n-4}} \bigg)  + \int_M \frac{2n}{n-4} \mu(t_\nu) u_\nu^{\frac{n+4}{n-4}} w_\nu - \frac{n(n+4)}{(n-4)^2}  \mu_\infty v_\nu^{\frac{8}{n-4}} w_\nu^2 \\
&\le C \bigg(\int_M |w_\nu|^{\frac{2n}{n-4}} \dvol \bigg)^{\frac{n-4}{n} \min \{ \frac{n}{n-4}, \frac{3}{2} \}} \\
& \le o(1) \int_M w_\nu \pa w_\nu.
\end{align*}

Applying Corollary~\ref{lem:secondv2} and collecting the above estimates, we conclude that
\begin{align*}
&\cf (t_\nu) \bigg(\int_M u_\nu^{\frac{2n}{n-4}} \bigg)^{\frac{n-4}{n}} \\
\le & \cf_\infty \bigg(\int_M u_\nu^{\frac{2n}{n-4}} \bigg)^{\frac{n-4}{n}} + C(\cf [v_\nu]- \cf_\infty ) \bigg(\int_M v_\nu^{\frac{2n}{n-4}} \bigg)^{\frac{n-4}{n}}    \\
& - C \int_M w_\nu \pa w_\nu + C \bigg( \int_M f(t_\nu) \pa f(t_\nu) \bigg)^{\tfrac{1}{2}} \bigg( \int_M w_\nu \pa w_\nu \bigg)^{\tfrac{1}{2}}   \\
\le & \cf_\infty \bigg(\int_M u_\nu^{\frac{2n}{n-4}} \bigg)^{\frac{n-4}{n}} + C(\cf [v_\nu]- \cf_\infty ) \bigg(\int_M v_\nu^{\frac{2n}{n-4}} \bigg)^{\frac{n-4}{n}} \\
& + C  \int_M f(t_\nu) \pa f(t_\nu).
\end{align*}

Finally, by Corollary~\ref{lem:nonzeroenergy}, we conclude that
\begin{align*}
\cf (t_\nu) - \cf_\infty \le &  C \| f_\nu\|_{\wtt}^{1+\gamma} + C (\mu(t_\nu) -\mu_\infty )^{1+\gamma} + C  \| f_\nu \|_{\wtt}^2 \\
\le & C \| f_\nu\|_{\wtt}^{1+\gamma} + C (\cf (t_\nu) - \cf_\infty )^{1+\gamma}.
\end{align*}
This completes the proof.
\end{proof}

\begin{proposition}\label{7.3}
There exist real numbers $0<\gamma<1$ and $t_0>0$ such that
\begin{equation*}
\cf(t) - \cf_\infty \le \|f(t) \|_{\wtt}^{1+\gamma}
\end{equation*}
for all $t \ge t_0$.
\end{proposition}

\begin{proof}
Suppose, to the contrary, that the statement is false. Then there exists a sequence of times $\{t_\nu\}$ with $t_\nu \ge \nu$ such that
\[
\cf(t_\nu ) - \cf_\infty \ge  \|f(t_\nu) \|_{\wtt}^{1+\frac{1}{\nu}} .
\]
Applying Proposition~\ref{keyestnonzero} to this sequence, we obtain an infinite subset $I \subset \mathbb{N}$, a real number $0<\gamma<1$, and a constant $C>0$ such that
\[
\cf(t_\nu ) - \cf_\infty \le C \|f(t_\nu) \|_{\wtt}^{1+\gamma}
\]
for all $\nu \in I$. Consequently,
\[
1 \le C \|f(t_\nu) \|_{\wtt}^{\gamma - \frac{1}{\nu}} .
\]
On the other hand, this contradicts the fact that $\|f(t)\|_{\wtt} \to 0$ as $t \to \infty$.
\end{proof}

\begin{proposition}\label{7.4}
We have
\begin{equation*}
\int_0^\infty \| f(t, \cdot)\|_{\wtt} \, dt < \infty .
\end{equation*}
\end{proposition}

\begin{proof}
Let $\theta = \frac{\gamma}{1+\gamma} \in \big(0, \tfrac{1}{2}\big)$. By Lemma~\ref{volmon} and Proposition~\ref{lem:LSineq}, we compute
\begin{align*}    
- \frac{d}{dt} \big( \mathcal{F}(t) - \mathcal{F}_\infty \big)^\theta
&= - \theta \big( \mathcal{F}(t) - \mathcal{F}_\infty \big)^{\theta -1 } \frac{d\mathcal{F}(t)}{dt} \\
&\ge - C \big\| f(t) \big\|_{W^{2,2}(M)}^{-1} \frac{d\mathcal{F}(t)}{dt} \\ 
&\ge C \big\| f(t) \big\|_{W^{2,2}(M)},
\end{align*}
where $C$ denotes a constant independent of $t$.

Integrating this inequality, we obtain
\[
C \int_{T_0}^{\infty} \| f \|_{W^{2,2}(M)} \, dt
\le \big( \mathcal{F}(T_0) - \mathcal{F}_\infty \big)^\theta < \infty
\]
for sufficiently large $T_0$.
\end{proof}

\begin{proposition}\label{lem:novolcon}
Given any $\eta_0>0$, there exists a real number $r>0$ such that
\[
\int_{B_r(x)} u(t)^{\frac{2n}{n-4}} \le \eta_0
\]
for all $x \in M$ and all $t \ge 0$.
\end{proposition}
\begin{proof}
We observe that
$$\frac{\partial u^{\frac{2n}{n-4}}}{\partial t} = \frac{2n}{n-4} f u^{\frac{n+4}{n-4}} = \frac{2n}{n-4} f\big(\frac{P_{g_0}f + P_{g_0} u}{\mu(t)}\big)$$
which implies that, for $0 \le t \le T$,
\begin{align*}
&\int_{B_r(x)} u(T)^{\frac{2n}{n-4}}- \int_{B_r(x)} u(t)^{\frac{2n}{n-4}} \\
= & C(n) \int_t^T \bigg[ \int_{B_r(x)} \frac{1}{\mu(t)} f (P_{g_0}f + P_{g_0} u ) \bigg]dt \\
\le& \frac{C(n) }{\mu_\infty} \int_t^\infty \bigg[ \int_M  |f P_{g_0}f| + |f| |P_{g_0} u| \bigg]dt \\
\le& C \int_t^\infty \bigg[ \| P_{g_0} f \|_{\ltwo} \|f\|_{\lone} + \|f\|_{\wtt}  \| P_{g_0} u \|_{\ltwo}  \bigg] dt.
\end{align*}

By the triangle inequality, we estimate
\begin{align*}
\| P_{g_0} u \|_{\ltwo} \le& \| P_{g_0} f \|_{\ltwo} + C \| u^{\frac{n+4}{n-4}} \|_{\ltwo} \\
=& \| P_{g_0} f \|_{\ltwo} + C \| u \|_{\lone}^{\frac{n+4}{n-4}} 
\end{align*}
which is uniformly bounded by Lemma~\ref{volmon} and Proposition~\ref{lem:LnormPf}.

Consequently, we obtain
$$\int_{B_r(x)} u(T)^{\frac{2n}{n-4}} \le \int_{B_r(x)} u(t)^{\frac{2n}{n-4}}+ C \int_t^\infty  \|f\|_{\wtt}   dt.$$

By Proposition~\ref{7.4}, $\int_t^\infty \|f\|_{W^{2,2}} \, dt$ can be made arbitrarily small for all sufficiently large $t$. This completes the proof.

\end{proof}

\begin{proof}[Proof of Theorem \ref{thm:globalconv}]

By Propositions~\ref{lem:novolcon} and~\ref{lem:globalC0bound}, we obtain
\[
\sup_{[0,\infty)\times M} u < \infty .
\]
Then, by Proposition~\ref{prop:C0impl}, both $\|u(t)\|_{C^{5,\alpha}(M)}$ and 
$\big\|\tfrac{\partial u(t)}{\partial t}\big\|_{C^{5,\alpha}(M)}$ are uniformly bounded. 
It follows that for any sequence of times there exists a subsequence converging to a smooth, positive limit. 
Proposition~\ref{7.4} ensures the uniqueness of such a limit. 
Hence, the flow converges to a unique limit, which solves the constant $Q$-curvature equation.

\end{proof}

\medskip \noindent \textbf{Acknowledgments.}
J. Wei is partially supported by GRF grant entitled ``On critical and super-critical Fujita equations''.

L. Gong wants to thank Seunghyeok Kim and Jingang Xiong for their interests and encouragements to this work. The authors are grateful Mingxiang Li for numerous insightful discussions and bringing Lin and Yuan's papers \cite{MR3529121,MR4380034} to our attention.

\bibliographystyle{abbrv}
\bibliography{Ref}

\end{document}